\documentclass[11pt]{article}
\usepackage{amsmath,amsfonts,amssymb,amsthm,enumerate,graphicx}
\usepackage{tikz,authblk}
\usepackage{subfig}
\usepackage{url}
\usepackage{float}
\usepackage{easyReview}
\usepackage[normalem]{ulem}

\newtheorem{theorem}{{\bf Theorem}}[section]

\newtheorem{example}[theorem]{{\bf Example}}

\newtheorem{proposition}[theorem]{{\bf Proposition}}

\newtheorem{claim}{{\bf Claim}}

\newcommand{\ZZ}{ \ensuremath{\mathbb{Z}}}

\begin{document}

\title{A class of maps on the torus and their vertex orbits}


	\author {Marbarisha M. Kharkongor}
 	\author { Dipendu Maity}
 	\affil{Department of Science and Mathematics,
 		Indian Institute of Information Technology Guwahati, Bongora, Assam-781\,015, India.\linebreak
 		\{marbarisha.kharkongor, dipendu\}@iiitg.ac.in/\{marbarisha.kharkongor, dipendumaity\}@gmail.com.}
 
\date{\today}

\maketitle

\begin{abstract}
A tiling (edge-to-edge) of the plane is a family of tiles that cover the plane without gaps or overlaps. Vertex figure of a vertex in a tiling to be the union of all edges incident to that vertex. A tiling is $k$-vertex-homogeneous if any two vertices with congruent vertex figures are symmetric with each other and the vertices form precisely $k$ transitivity classes with respect to the group of all symmetries of the tiling. In this article, we discuss that if a map is the quotient of a plane's $k$-vertex-homogeneous lattice ($k \ge 4$) then what would be the sharp bounds of the number of vertex orbits.
\end{abstract}

\noindent {\small {\em MSC 2010\,:} 52C20, 52B70, 51M20, 57M60.

\noindent {\em Keywords:} Polyhedral map on torus; $k$-vertex-homogeneous tilings; Symmetric group.}

\section{Introduction}
Tilings of the plane by only regular polygons define the most symmetric tilings. Such tilings have been known since antiquity, but most of the classification results involving them have appeared only in the last two decades.  
Tiling is associated with it in a natural way to vertices, edges and tiles. 
The vertices, edges and tiles together of a tiling is called the elements of that tiling. Regularity of a tiling is defined by equivalence classes of elements of the tiling under the symmetry group of the tiling. In group theory, these equivalence classes are called the element's orbits under the symmetry group. 
We know that there are six equilateral triangles, four squares or three regular hexagons at a vertex, yielding the three regular (flag-transitivity) tilings. If the requirement of flag-transitivity is relaxed to one of vertex-transitivity, while the condition that the tiling is edge-to-edge is kept, there are eight additional tilings possible, known as Archimedean, uniform or semiregular tilings. Similarly, the tilings may be classified by the number of orbits of vertices, edges and tiles. If there are $k$ orbits of vertices, a tiling is known as $k$-uniform or $k$-isogonal; if there are $t$ orbits of tiles, as $t$-isohedral; if there are $e$ orbits of edges, as $e$-isotoxal. Note that $k$-uniform tilings with the same vertex figures can be further identified by their wallpaper group symmetry.
We say that two vertices of the tiling are symmetric with each other if they are in the same orbit. A tiling is called \emph{$k$-vertex-homogeneous} if any two vertices with congruent vertex figures are symmetric with each other and the vertices form precisely $k$ transitivity classes with respect to the group of all symmetries of the tiling. We know from \cite{chavey89, Otto1977, GS1977, GS1981} that there are exactly $135$ distinct vertex-homogeneous tilings on the plane. Since the plane is the universal cover of the torus and the vertex-homogeneous tilings have finite fundamental domain, it is interesting to ask:  if a map on the torus is the quotient of vertex-homogeneous tilings then what would be the bounds on the vertex orbits. That is, let $E$ be a vertex-homogeneous tiling on the plane and $X = E/H$ for some fixed element free group $H$. In this article, we will discuss and determine the sharp bounds of the vertex orbits of these quotient maps $X$ on the torus.

We know from \cite{Otto1977, GS1977, GS1981} that there are $11$ $1$-vertex-homogeneous tilings, $20$ $2$-vertex-homogeneous, $39$ $3$-vertex-homogeneous, $33$ $4$-vertex-homogeneous, $15$ $5$-vertex-homogeneous, $10$ $6$-vertex-homogeneous and $7$ $7$-vertex-homogeneous non isomorphic tilings on $\mathbb{R}^2$ and there does not exist $k$-vertex-homogeneous tilings on $\mathbb{R}^2$ if $k \ge 8$ (see \cite{Otto1977}, \cite{GS1977}, \cite{GS1981}).  

By a $map$, we mean a $2$-cell (called {\em faces}) decomposition of a compact connected surface. A surjective mapping $\eta \colon Y \to X$ from a map $Y$ to a map $X$ is called a $covering$ if it preserves adjacency and sends vertices, edges, faces of $Y$ to vertices, edges, faces of $X$ respectively. Let $X$ be a map on the torus and $\eta \colon Y \to X $ be a covering map. Let the vertices of $X$ form $m$ Aut$(X)$-orbits. If $Y$ is a $1$-uniform tiling on the plane then $m \le 6$ (see \cite{DM2017}, \cite{DM2018}). If $Y$ is a $2$-uniform tiling on the plane then $m \le 9$ (see \cite{DM2020}). If $Y$ is a $3$-uniform tiling on the plane then $m \le 15$ (see \cite{DM2020}). Here, we prove the following. 

\begin{theorem} \label{theo1} Let $X$ be a map on the torus and $\eta \colon Y \to X $ be a covering map. Let the vertices of $X$ form $m$ ${Aut}(X)$-orbits. Let $K_i$ for $1 \le i \le 65$ (in Example \ref{exam:plane1}) denote $k$-vertex-homogeneous tilings on the plane for some $k \in \{4, 5, 6, 7\}$.\\
\smallskip
{\rm (1)} If $Y = K_{1}, K_{2}, K_{4}, K_{8}, K_{9}, K_{14}, K_{20}, K_{24}, K_{26}, K_{31}, K_{32}, K_{33}$, then, $m \le 6$.\\
\smallskip
{\rm (2)} If $Y = K_{6}, K_{7}$, then, $m = 4$.\\
\smallskip
{\rm (3)} If $Y = K_{3}, K_{30}$, then, $m \le 5$.\\
\smallskip
{\rm (4)} If $Y = K_{5}, K_{21}, K_{25}, K_{29}, K_{34}, K_{40}, K_{43}, K_{47}$, then, $m \le 7$.\\
\smallskip
{\rm (5)} If $Y = K_{22}, K_{35}, K_{39}, K_{41}, K_{42}, K_{44}$, then, $m \le 8$.\\
\smallskip
{\rm (6)} If $Y = K_{18}, K_{52}, K_{53}, K_{56}$, then, $m \le 9$.\\
\smallskip
{\rm (7)} If $Y = K_{13}, K_{54}$, then, $m \le 10$.\\
\smallskip
{\rm (8)} If $Y = K_{37}, K_{60}$, then, $m \le 11$.\\
\smallskip
{\rm (9)} If $Y = K_{19}, K_{65}$, then, $m \le 12$.\\
\smallskip
{\rm (10)} If $Y = K_{15}, K_{16}, K_{17}, K_{48}$, then, $m \le 13$.\\
\smallskip
{\rm (11)} If $Y = K_{11}, K_{12}, K_{28}$, then, $m \le 15$.\\
\smallskip
{\rm (12)} If $Y = K_{10}, K_{38}$, then, $m \le 16$.\\
\smallskip
{\rm (13)} If $Y = K_{27}, K_{46}$, then, $m \le 18$.\\
\smallskip
{\rm (14)} If $Y = K_{45}$, then, $m \le 19$.\\
\smallskip
{\rm (15)} If $Y = K_{23}, K_{57}$, then, $m \le 21$.\\
\smallskip
{\rm (16)} If $Y = K_{36}, K_{63}$, then, $m \le 22$.\\
\smallskip
{\rm (17)} If $Y = K_{49}$, then, $m \le 24$.\\
\smallskip
{\rm (18)} If $Y = K_{51}$, then, $m \le 25$.\\
\smallskip
{\rm (19)} If $Y = K_{58}$, then, $m \le 26$.\\
\smallskip
{\rm (20)} If $Y = K_{55}$, then, $m \le 27$.\\
\smallskip
{\rm (21)} If $Y = K_{64}$, then, $m \le 28$.\\
\smallskip
{\rm (22)} If $Y = K_{50}$, then, $m \le 30$.\\
\smallskip
{\rm (23)} If $Y = K_{59}, K_{61}, K_{62}$, then, $m \le 33$.\\
\smallskip
{\rm (24)} The bounds in {\rm (1)} - {\rm (23)} are sharp.
\end{theorem}

\section{$k$-vertex-homogeneous tilings on the plane, $k \ge 4$}\label{fig:kuniform}
\begin{example} \label{exam:plane1}
{\rm $K_i$ $(1\le i \le 33)$ are $4$-vertex-homogeneous, $K_i$ $(34\le i \le 48)$ are $5$-vertex-homogeneous, $K_i$ $(49\le i \le 58)$ are $6$-vertex-homogeneous and $K_i$ $(59\le i \le 65)$ are $7$-vertex-homogeneous tilings on the plane. We need these tilings in Theorems \ref{theo1}.}
\end{example}
\begin{figure}[H]
    \centering
    \includegraphics[height=2.9cm, width= 2.9cm]{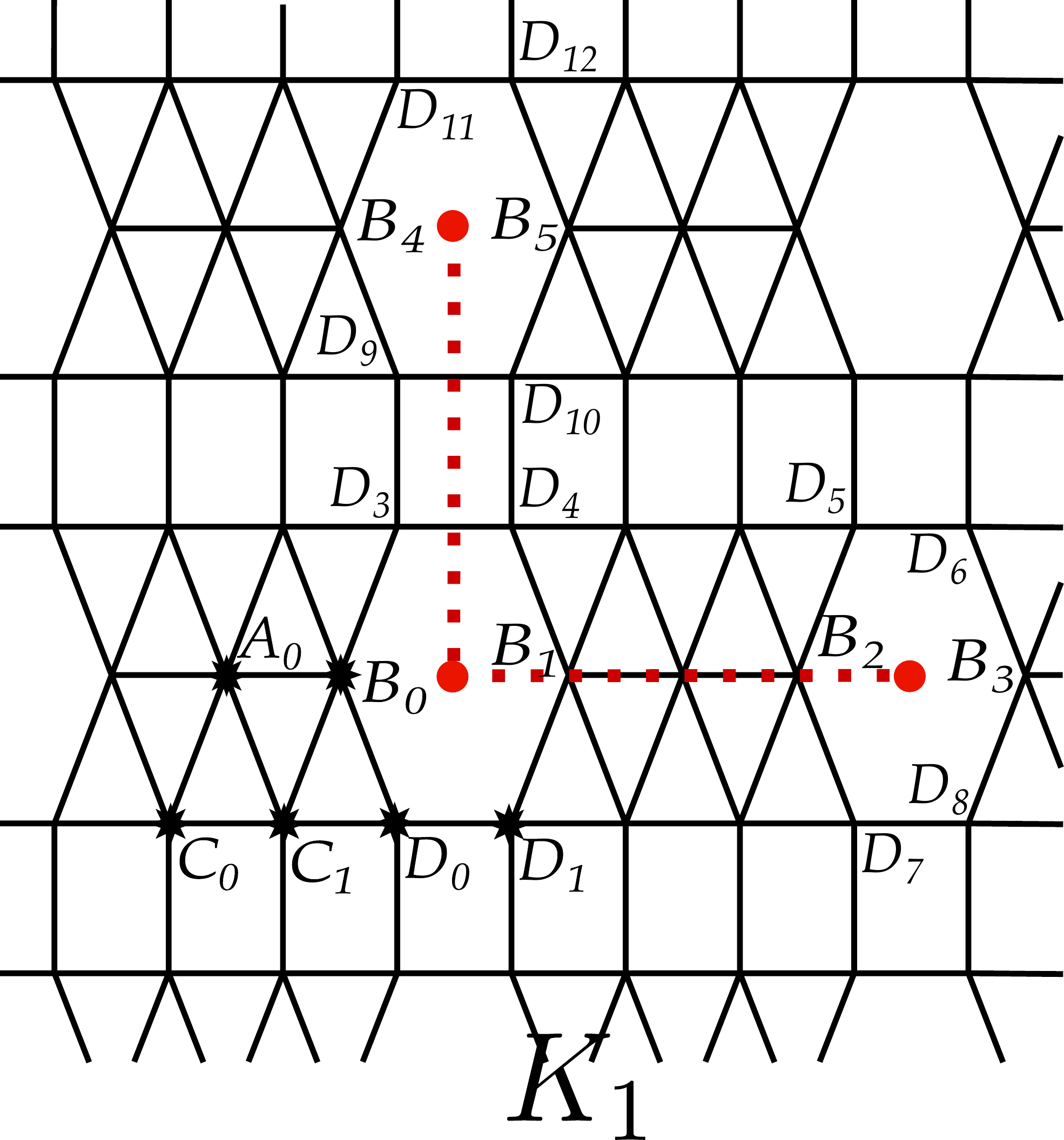}
    \includegraphics[height=2.9cm, width= 2.9cm]{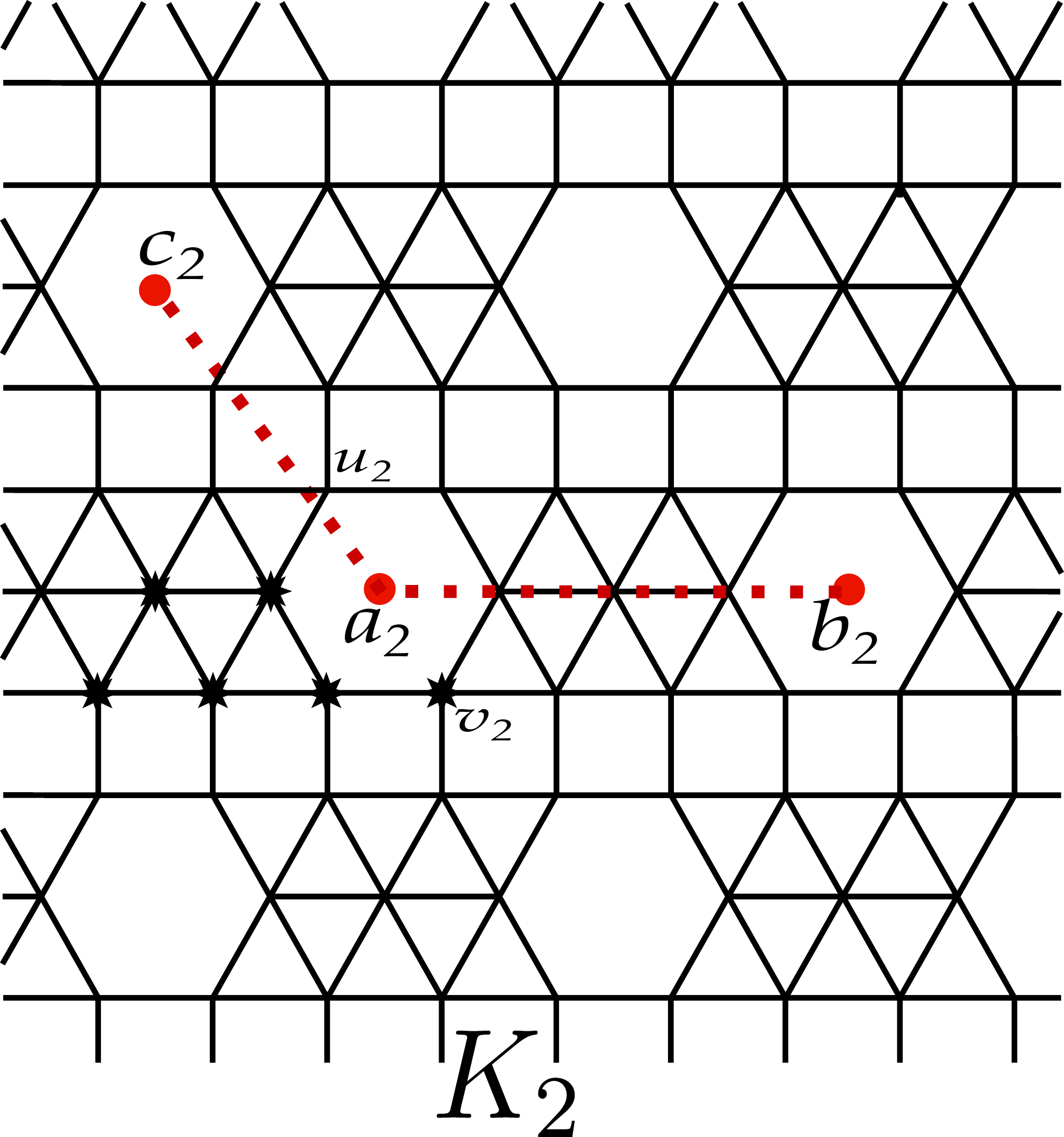}
    \includegraphics[height=2.9cm, width= 2.9cm]{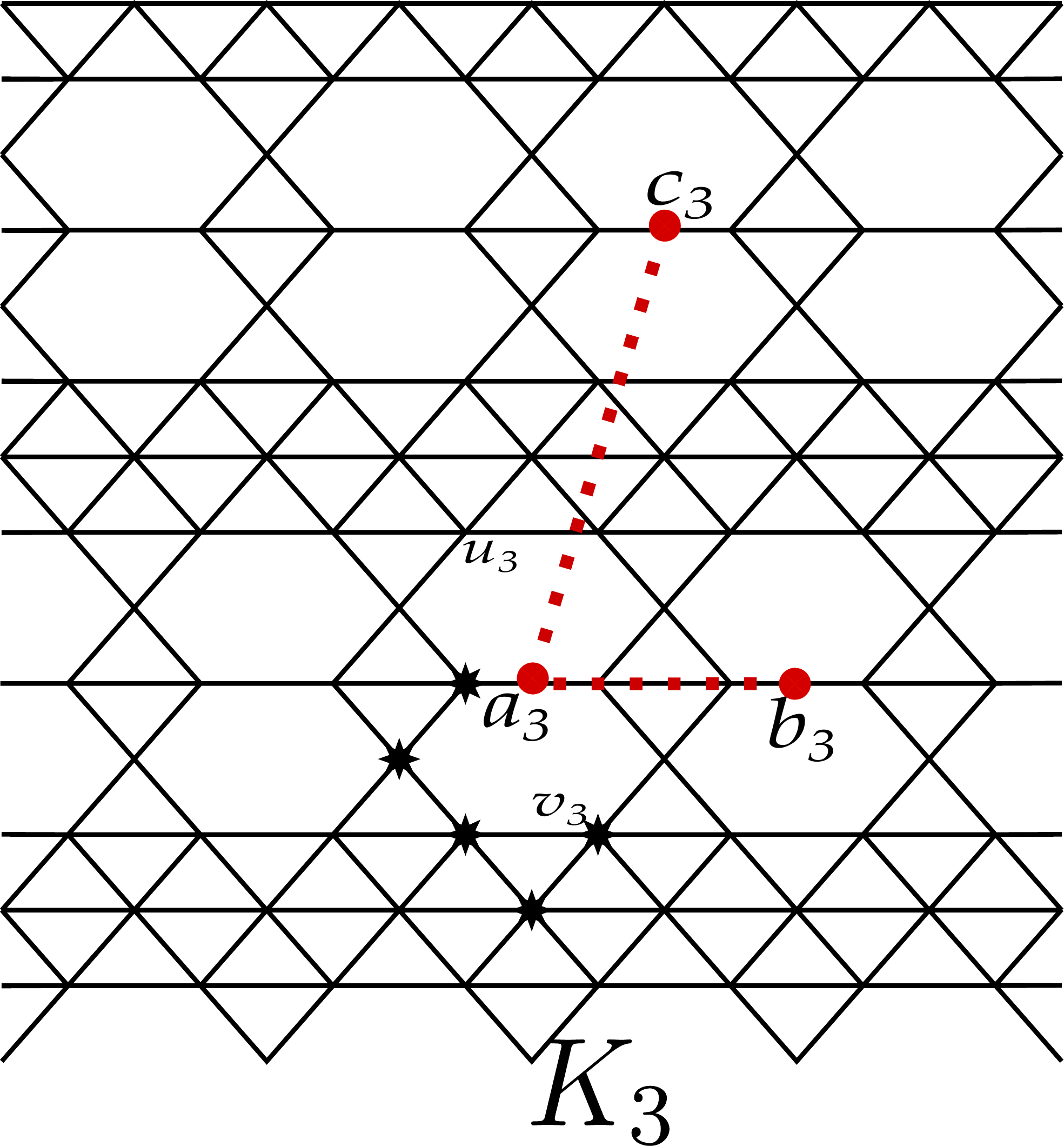}
    \includegraphics[height=2.9cm, width= 2.9cm]{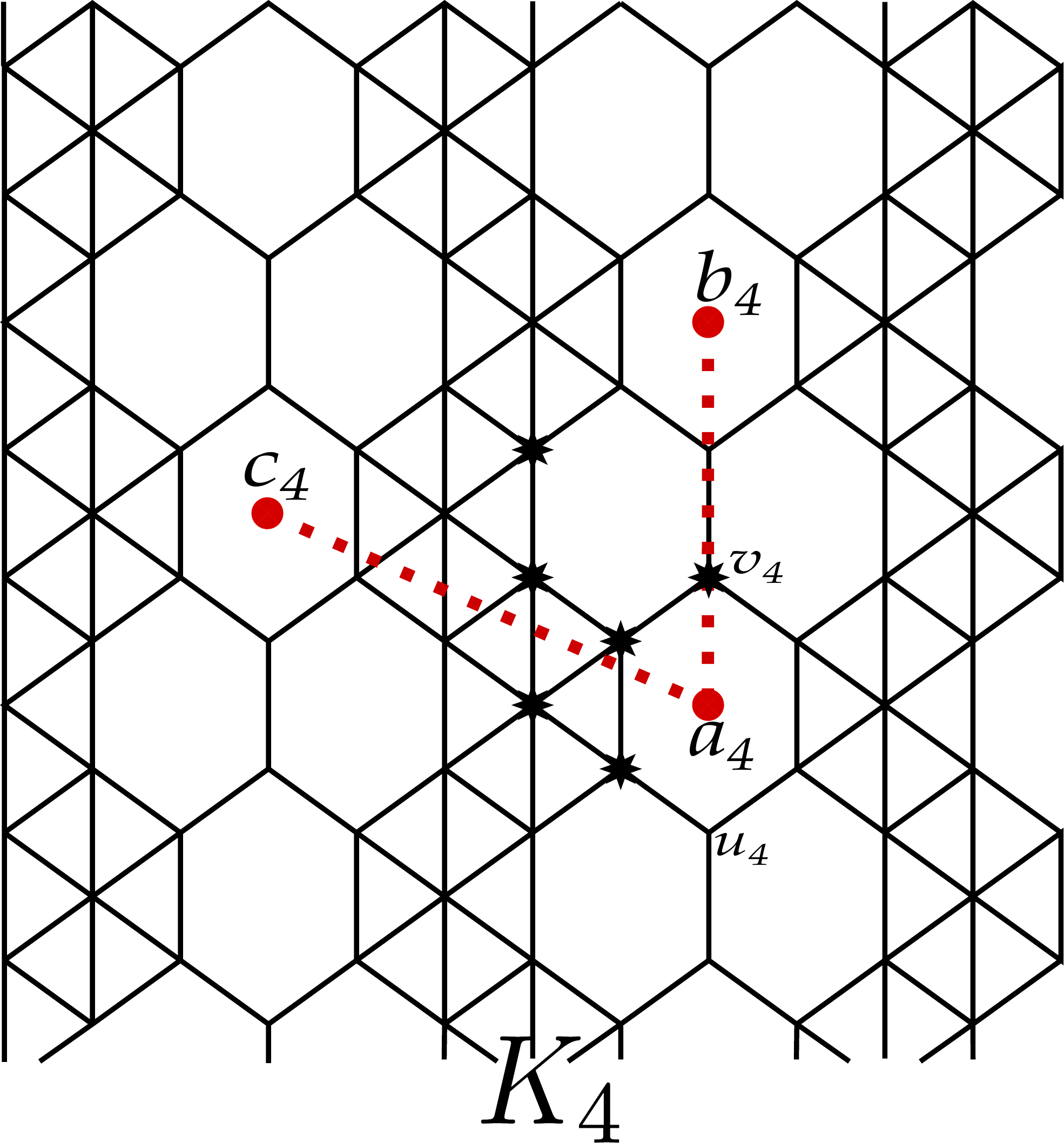}
    \includegraphics[height=2.9cm, width= 2.9cm]{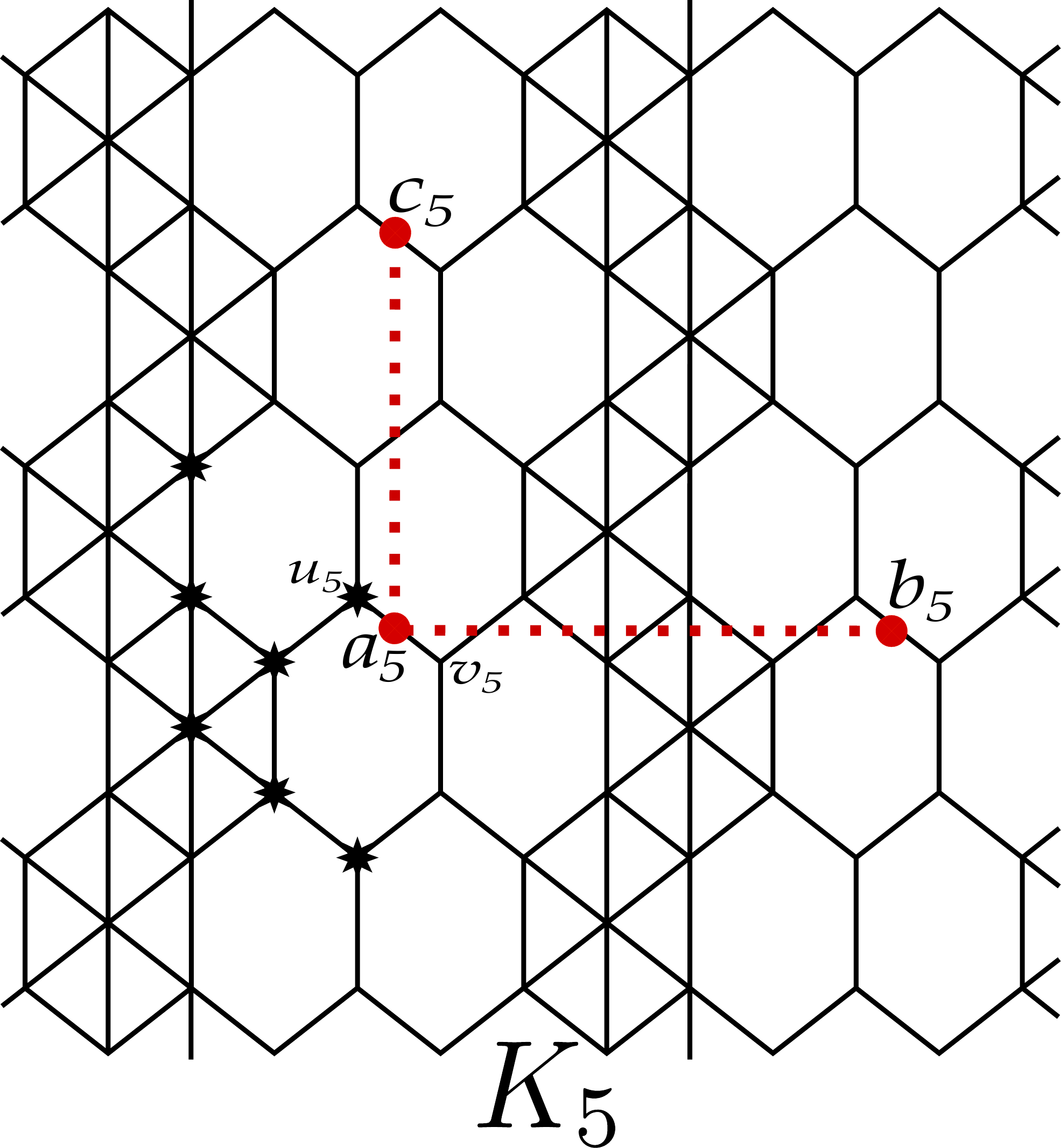}
    \end{figure}
    \vspace{-7mm}
\begin{figure}[H]
    \centering
    \includegraphics[height=2.9cm, width= 2.9cm]{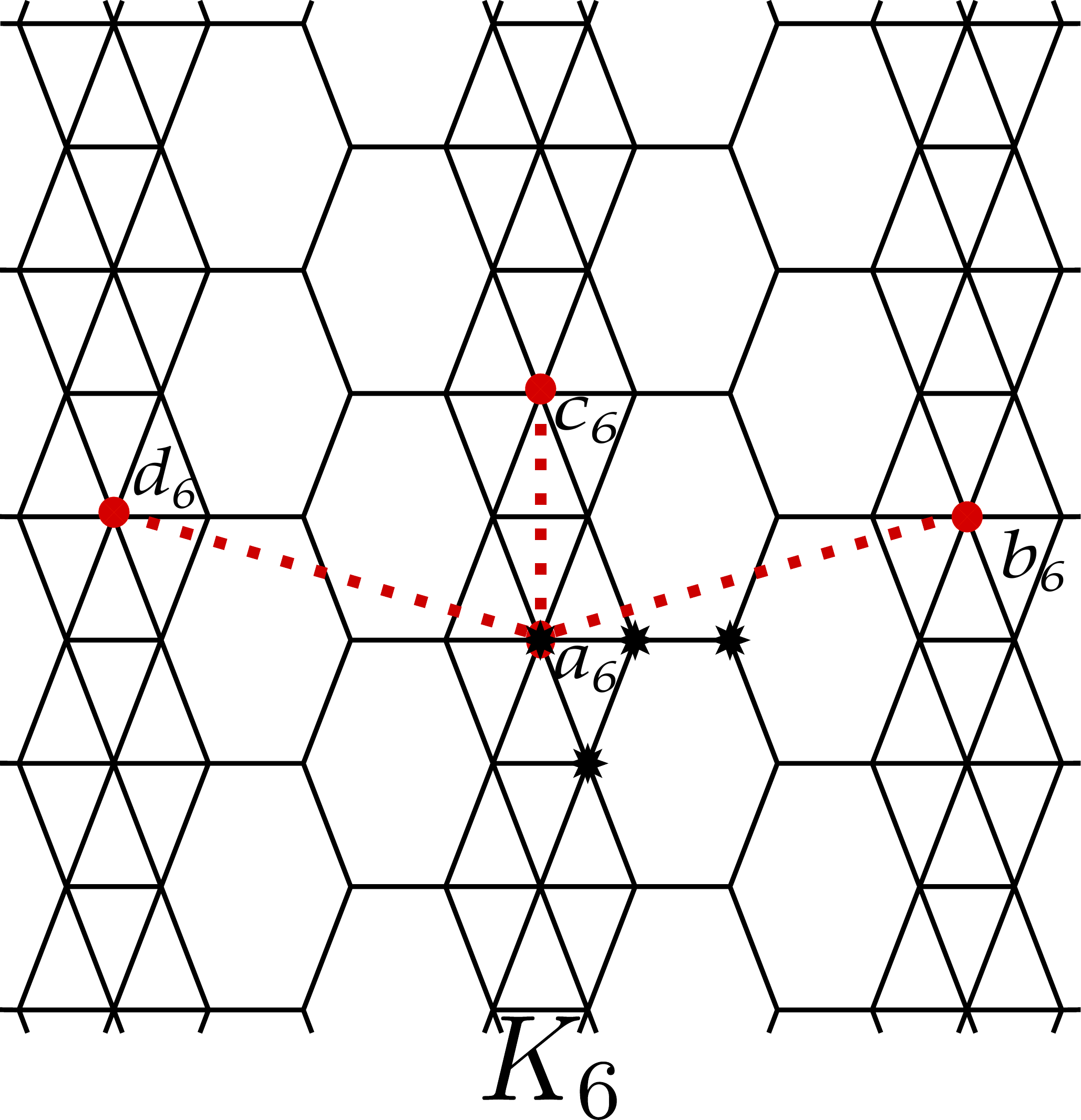}
    \includegraphics[height=2.9cm, width= 2.9cm]{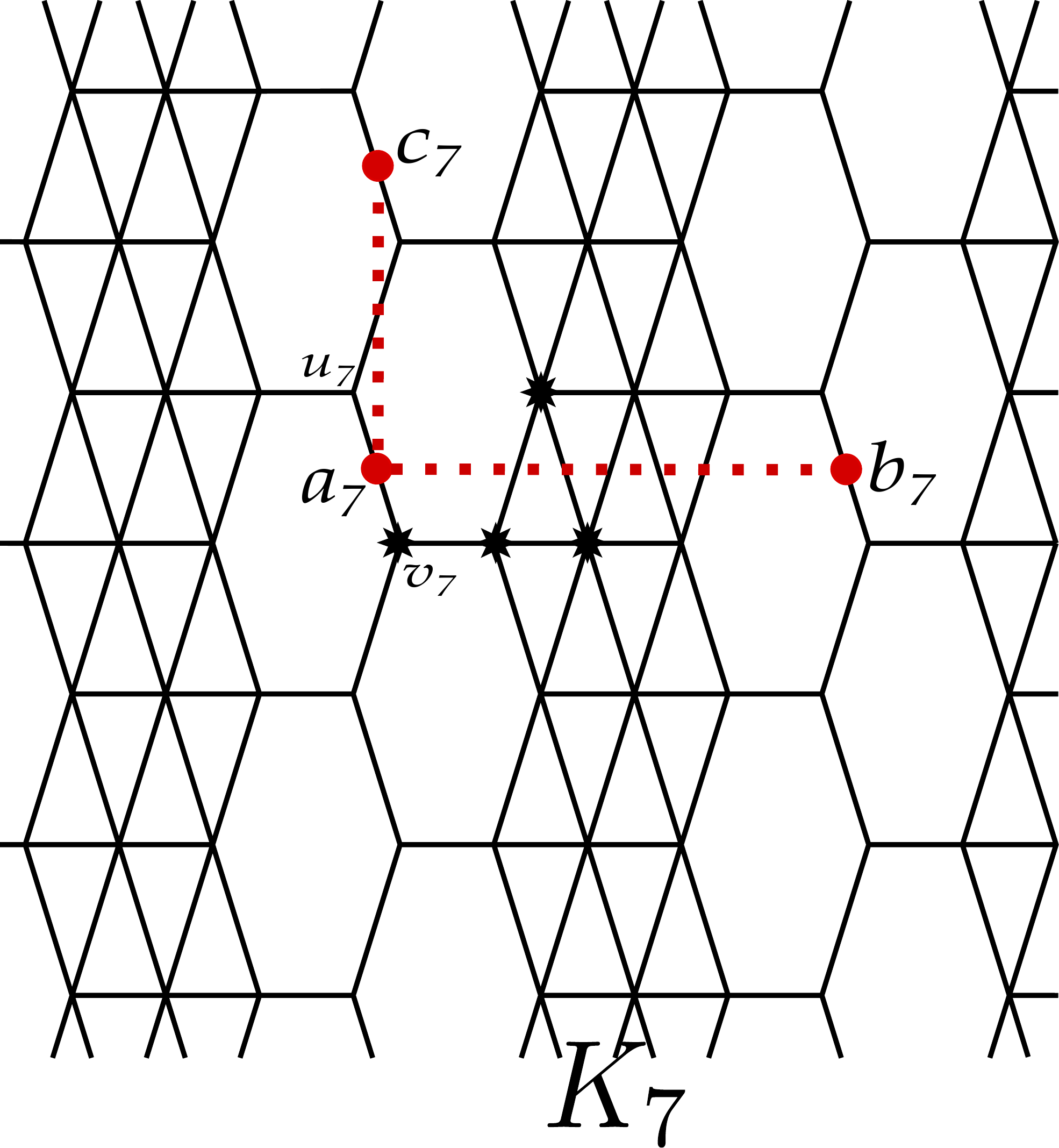}
    \includegraphics[height=2.9cm, width= 2.9cm]{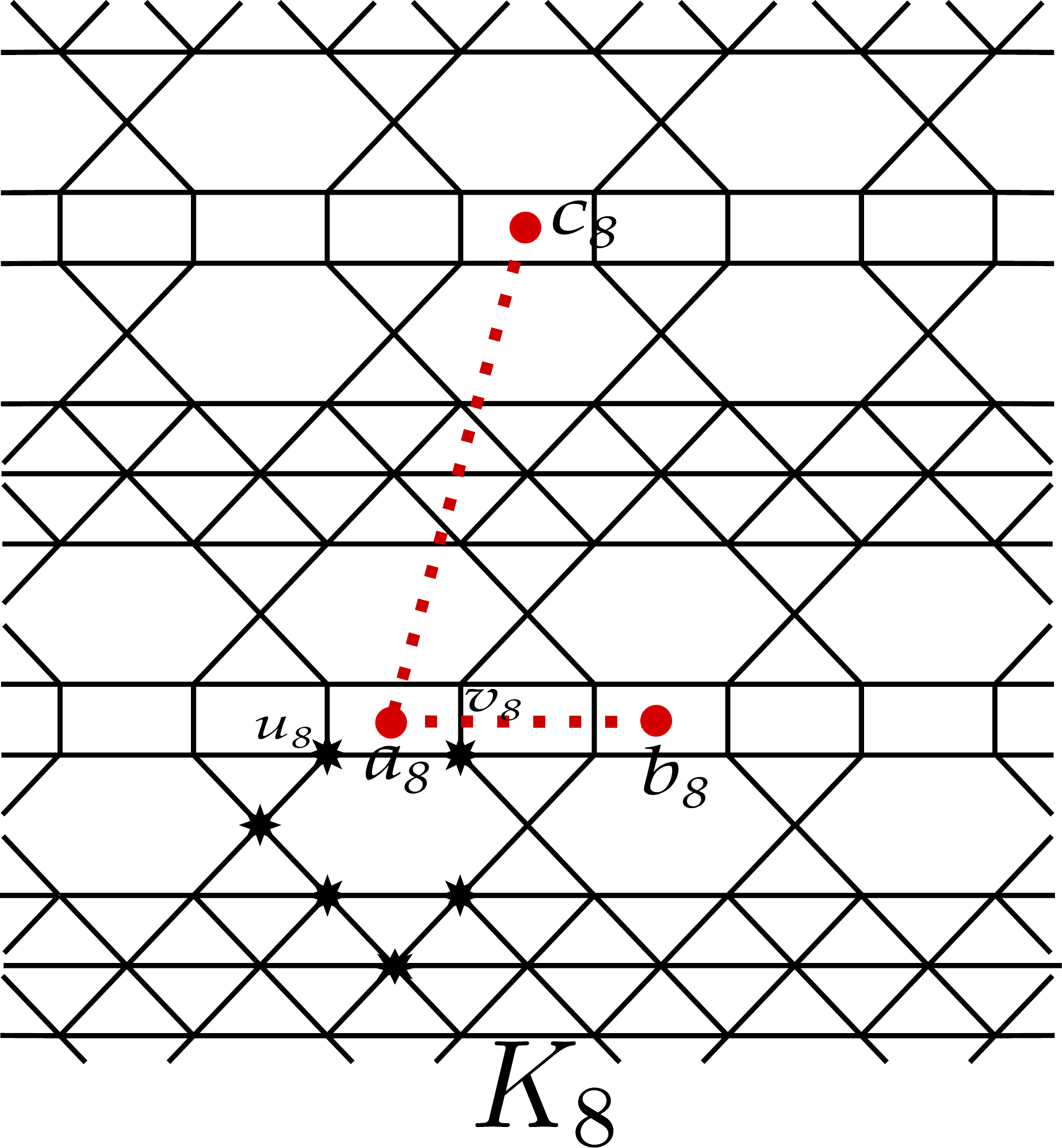}
    \includegraphics[height=2.9cm, width= 2.9cm]{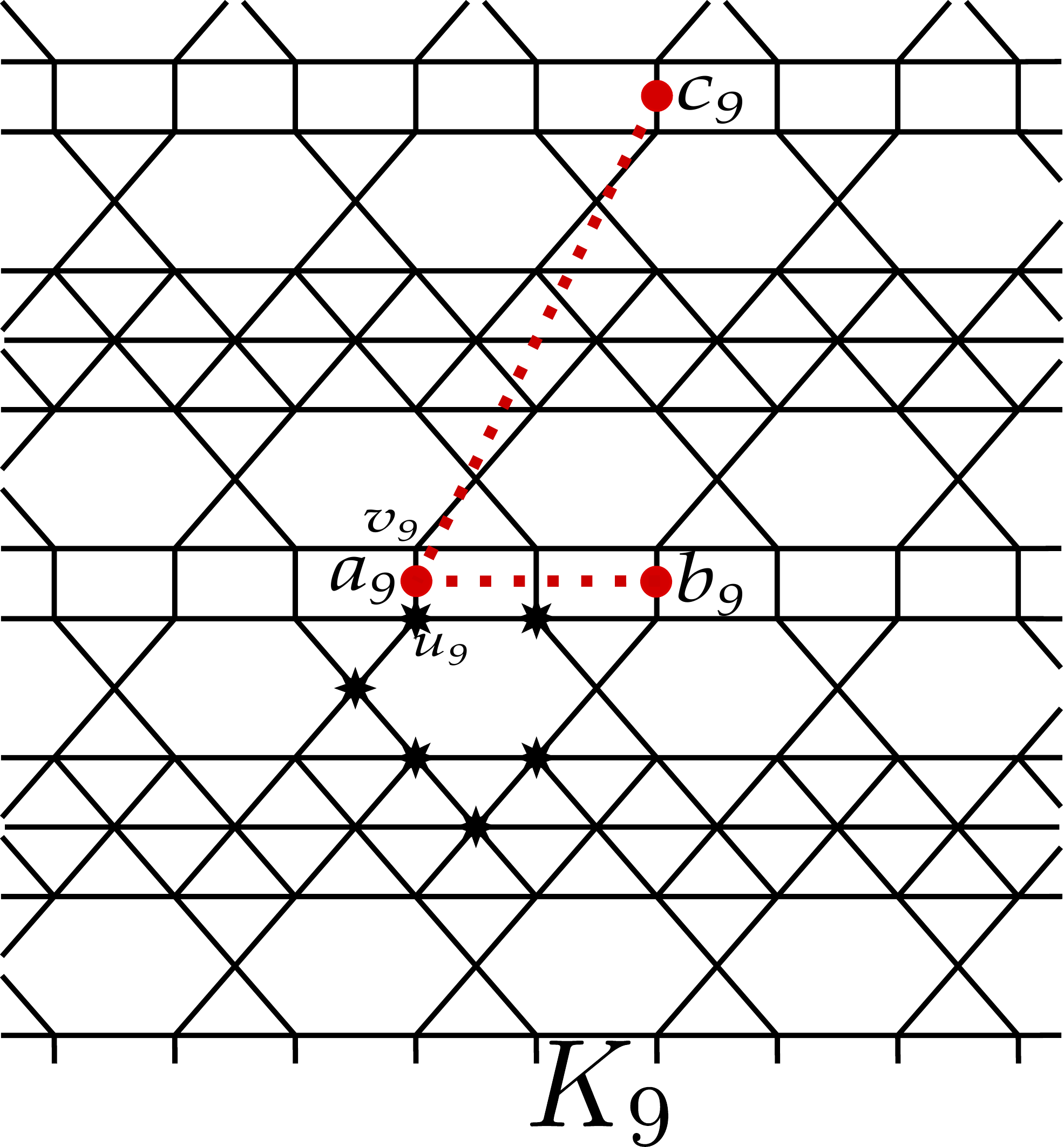}
    \includegraphics[height=2.9cm, width= 2.9cm]{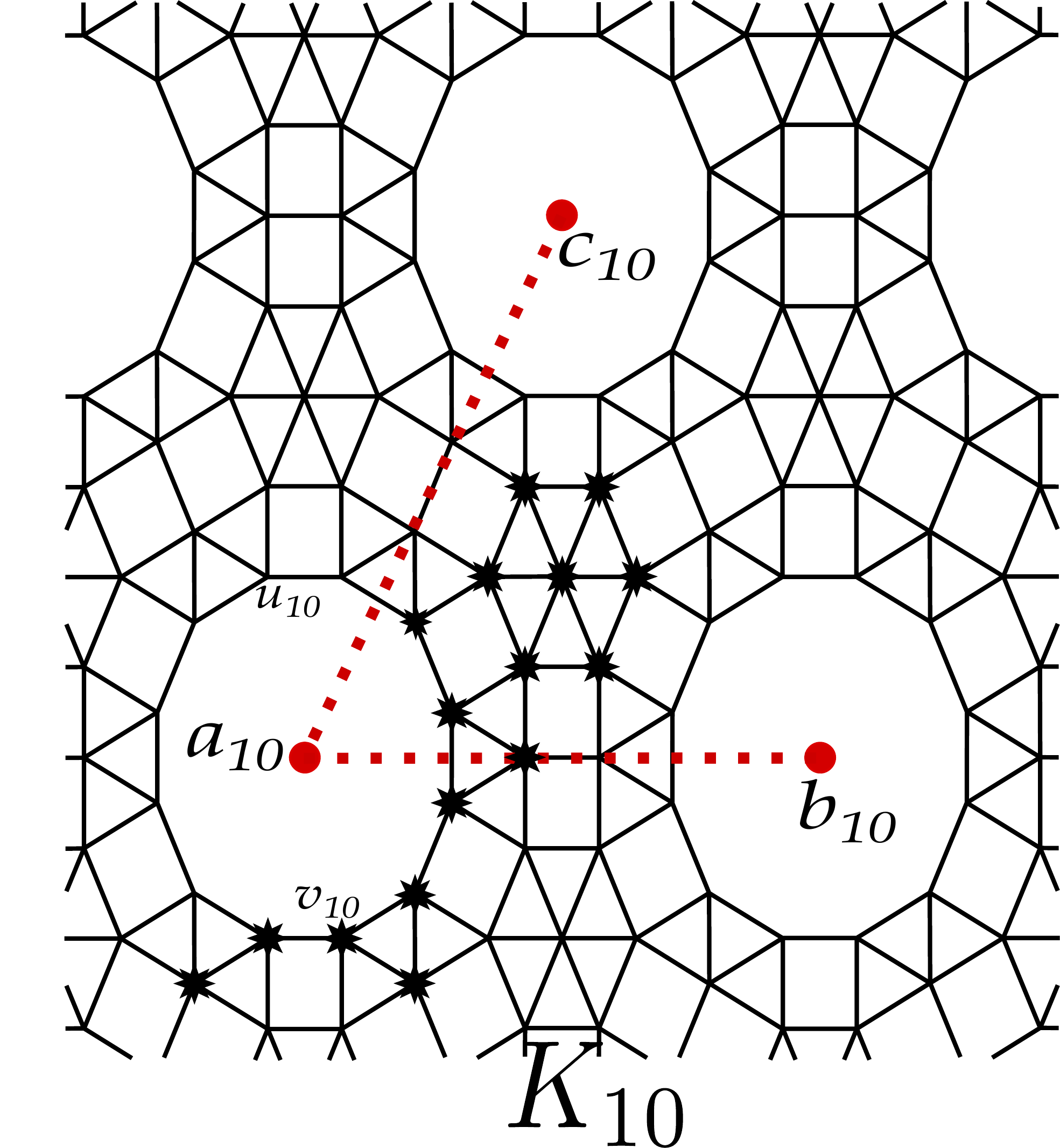}
     \vspace{-7mm}
\end{figure}
\begin{figure}[H]
    \centering
    \includegraphics[height=2.9cm, width= 2.9cm]{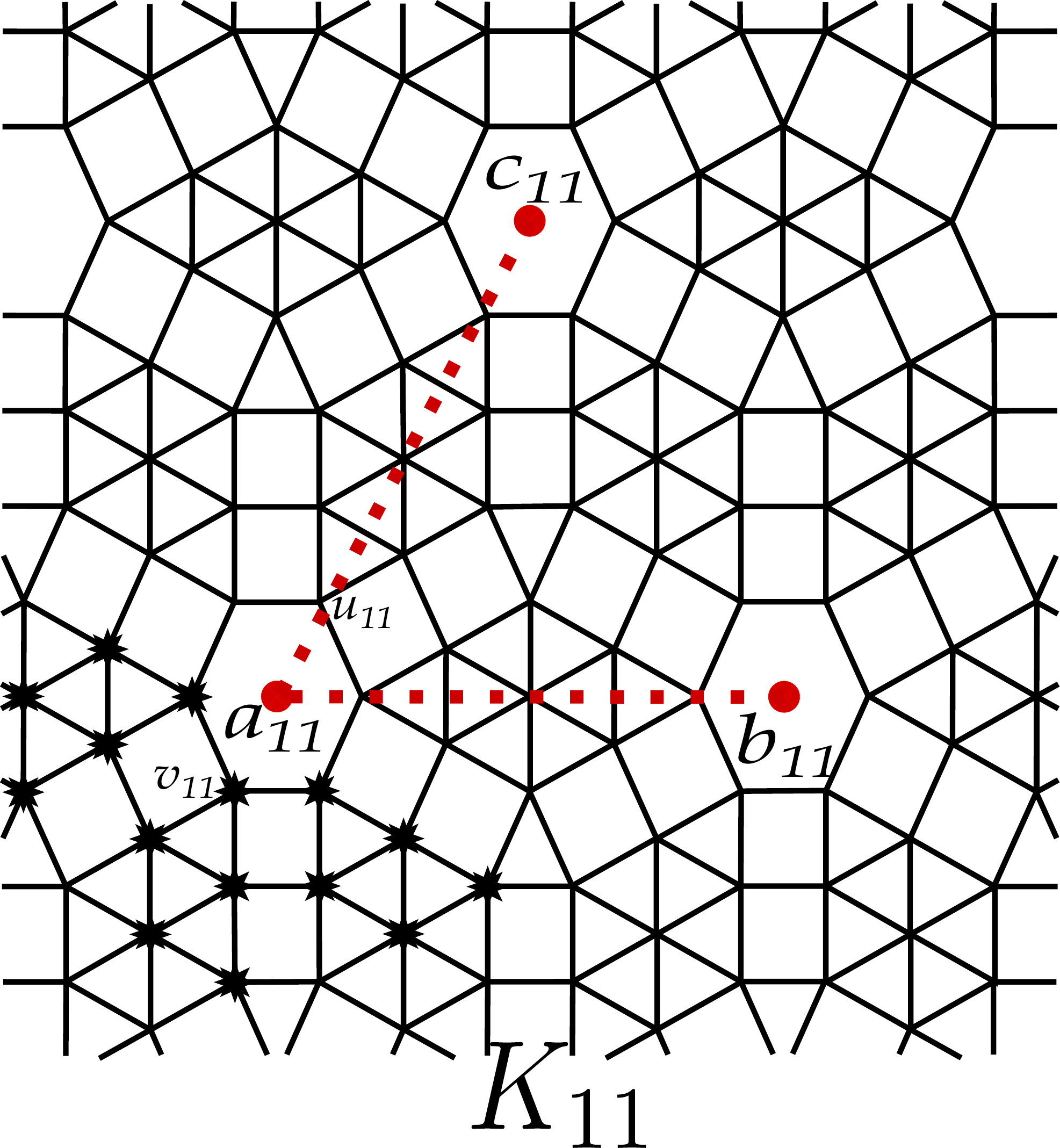}
    \includegraphics[height=2.9cm, width= 2.9cm]{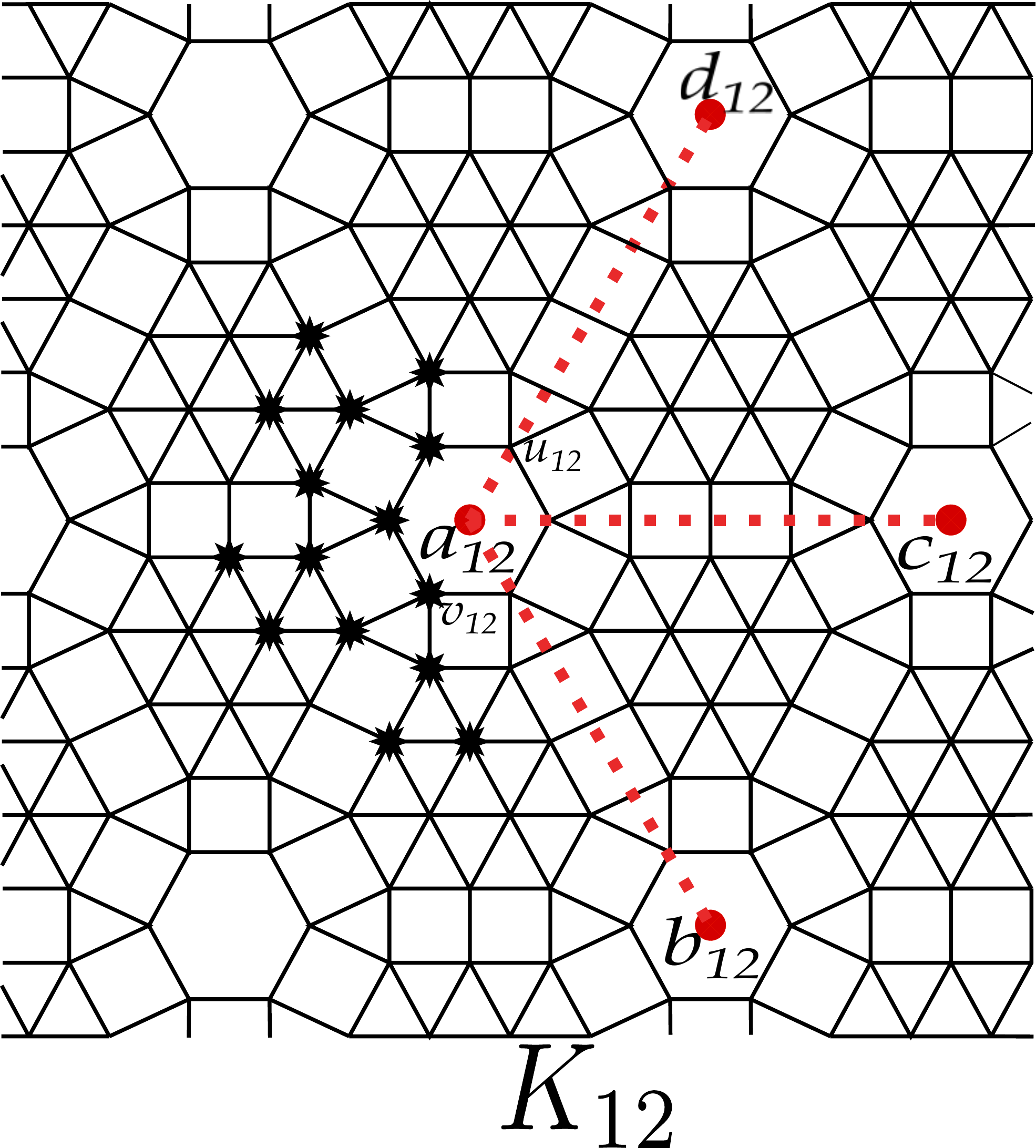}
    \includegraphics[height=2.9cm, width= 2.9cm]{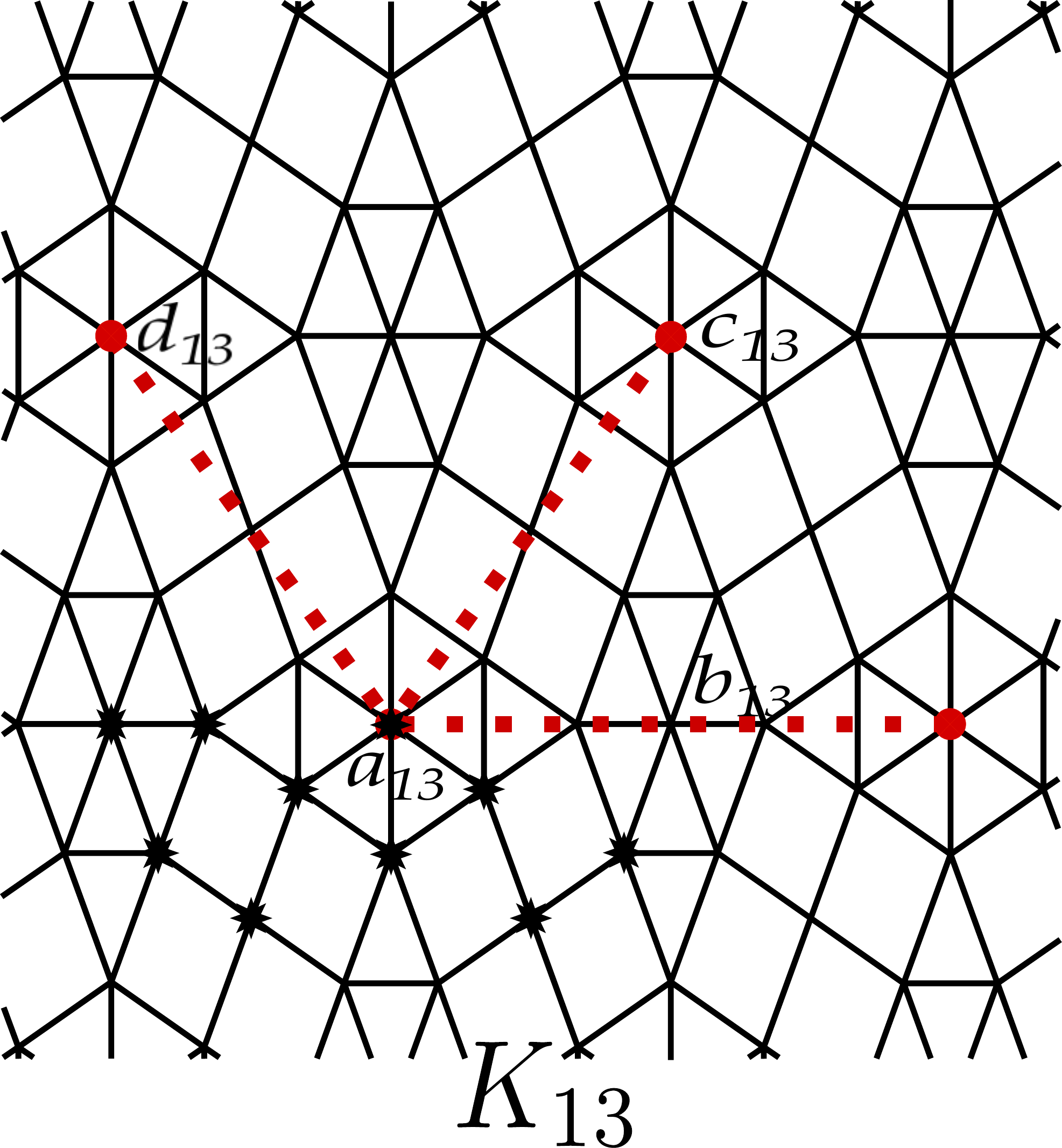}
    \includegraphics[height=2.9cm, width= 2.9cm]{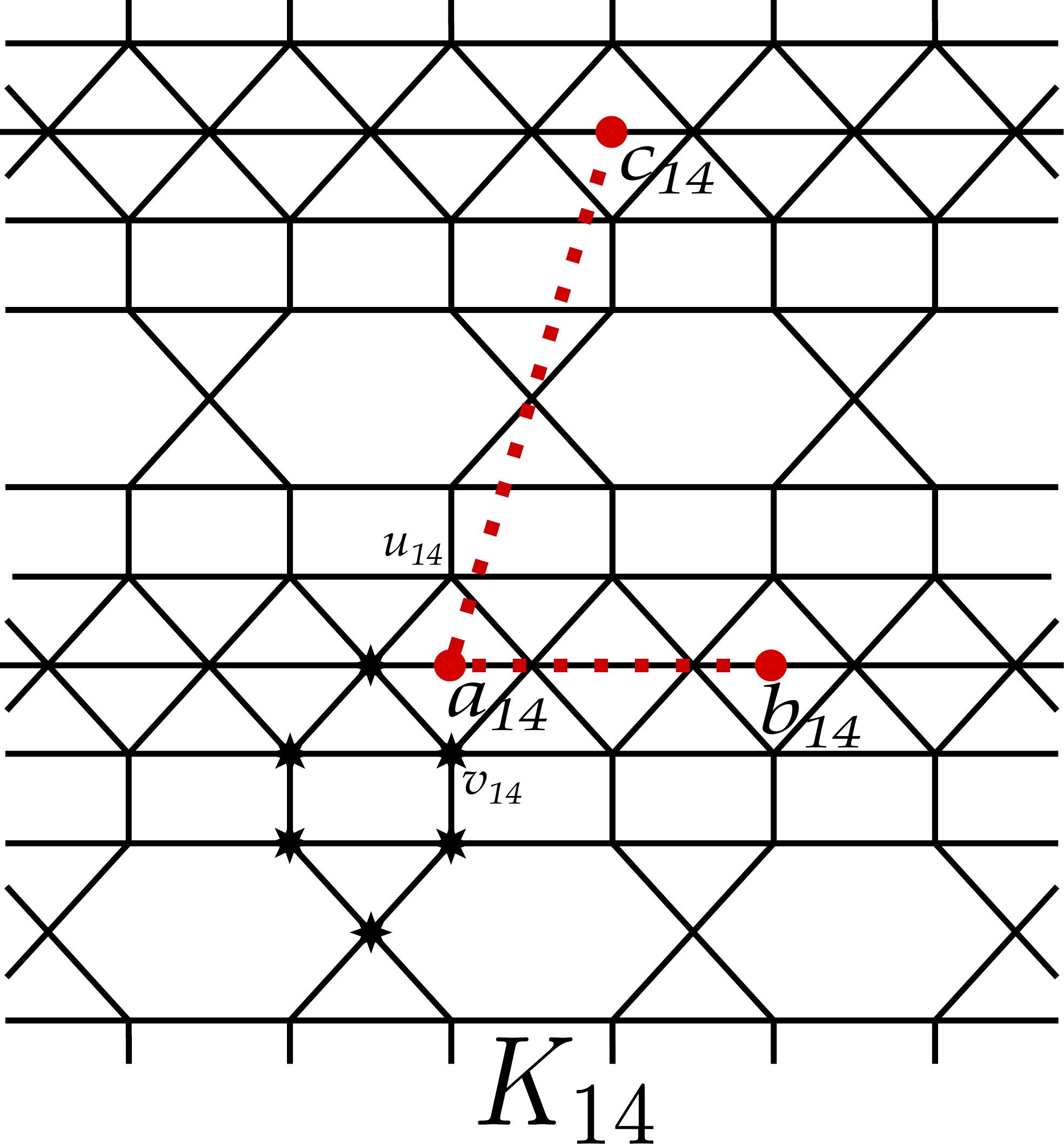}
    \includegraphics[height=2.9cm, width= 2.9cm]{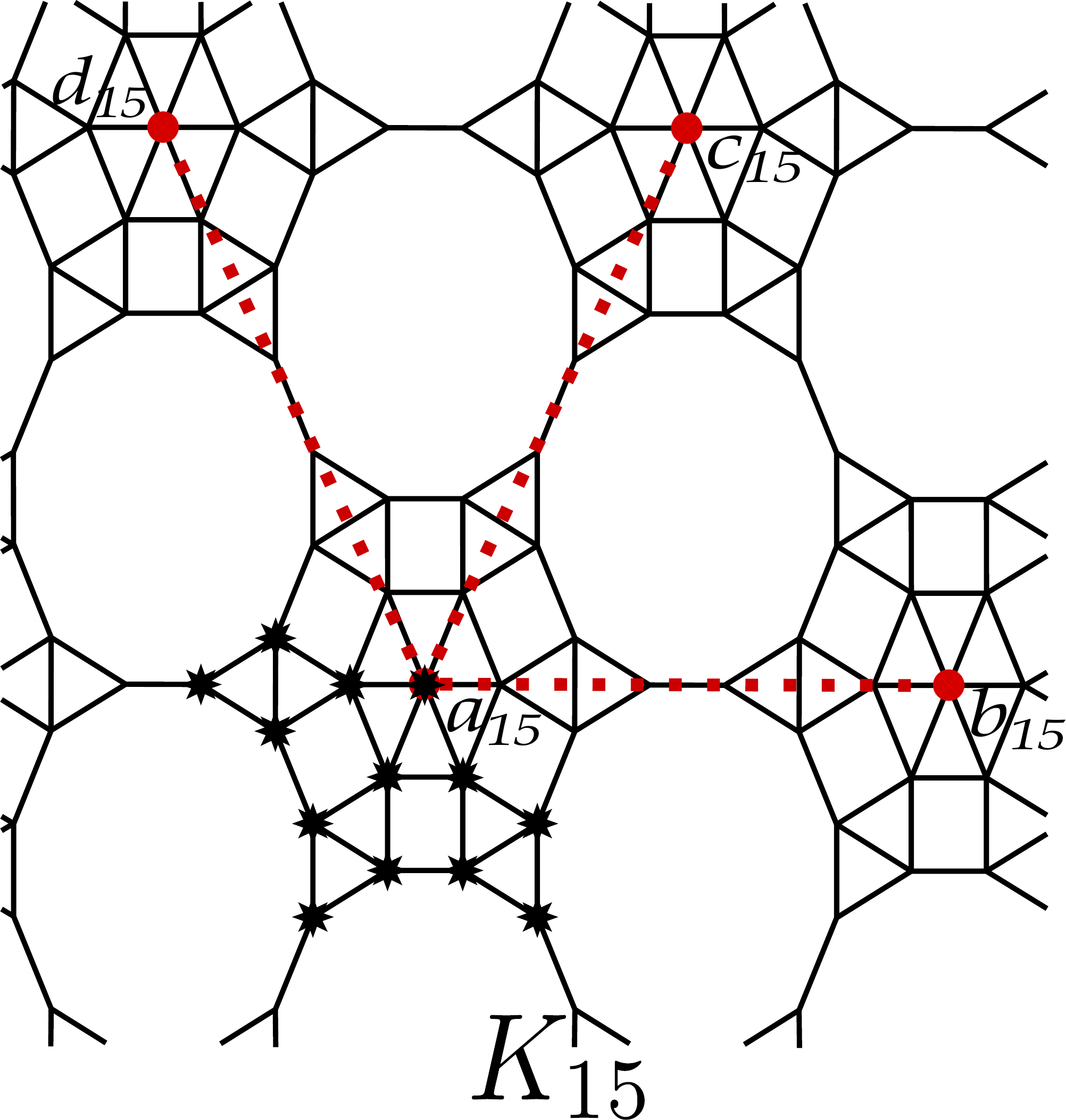}
     \vspace{-7mm}
\end{figure}
\begin{figure}[H]
    \centering
    \includegraphics[height=2.9cm, width= 2.9cm]{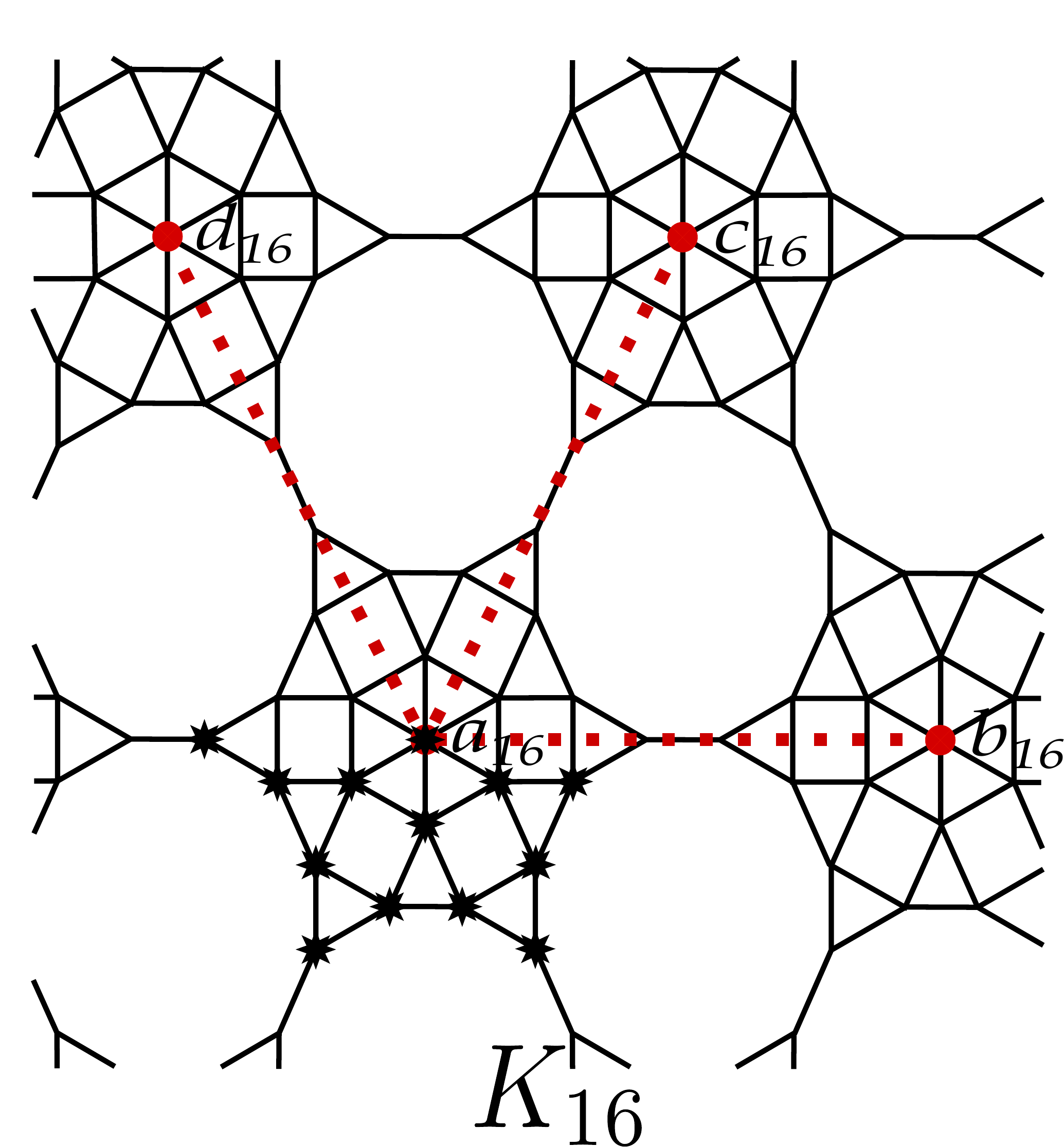}
    \includegraphics[height=2.9cm, width= 2.9cm]{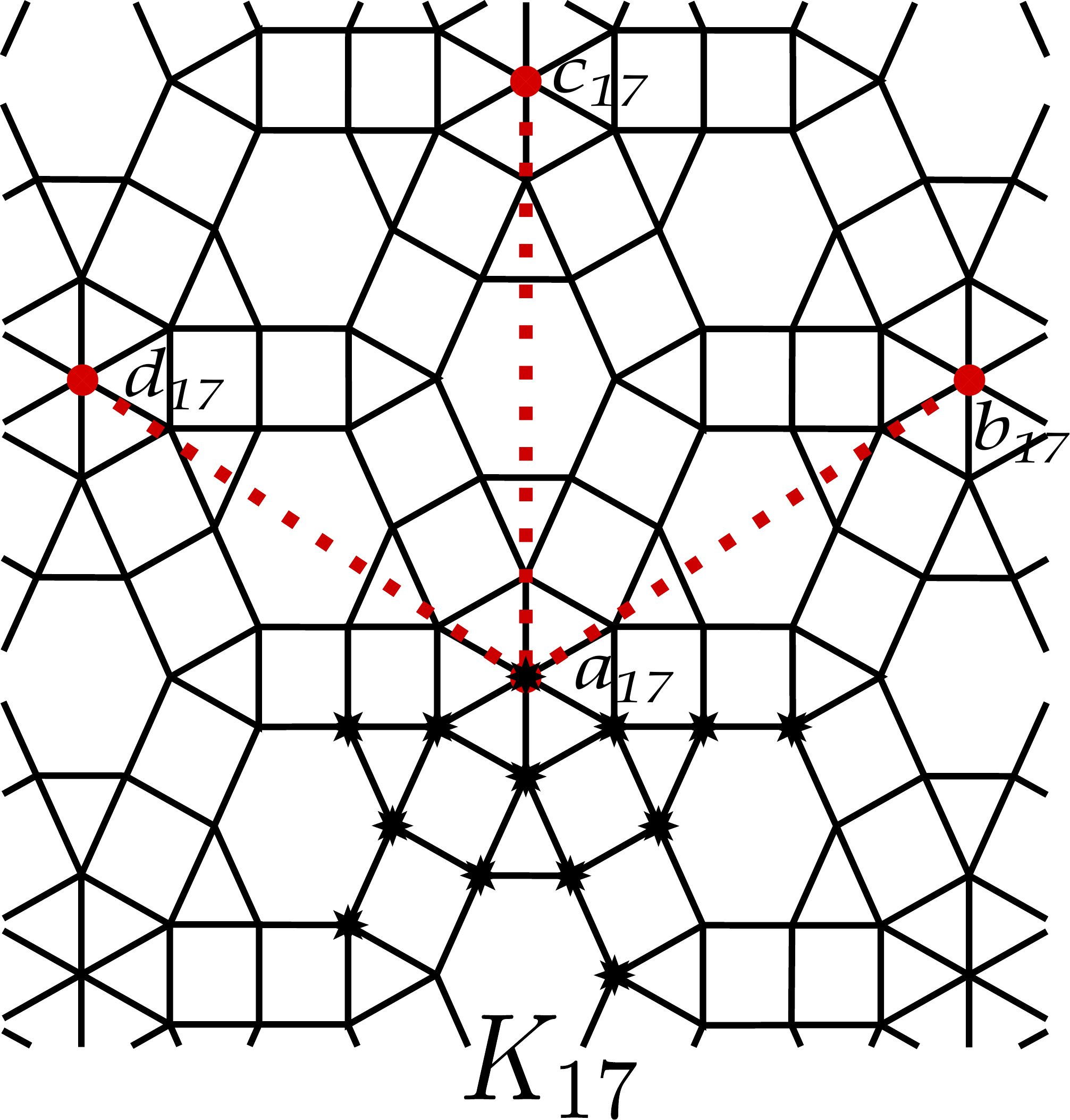}
    \includegraphics[height=2.9cm, width= 2.9cm]{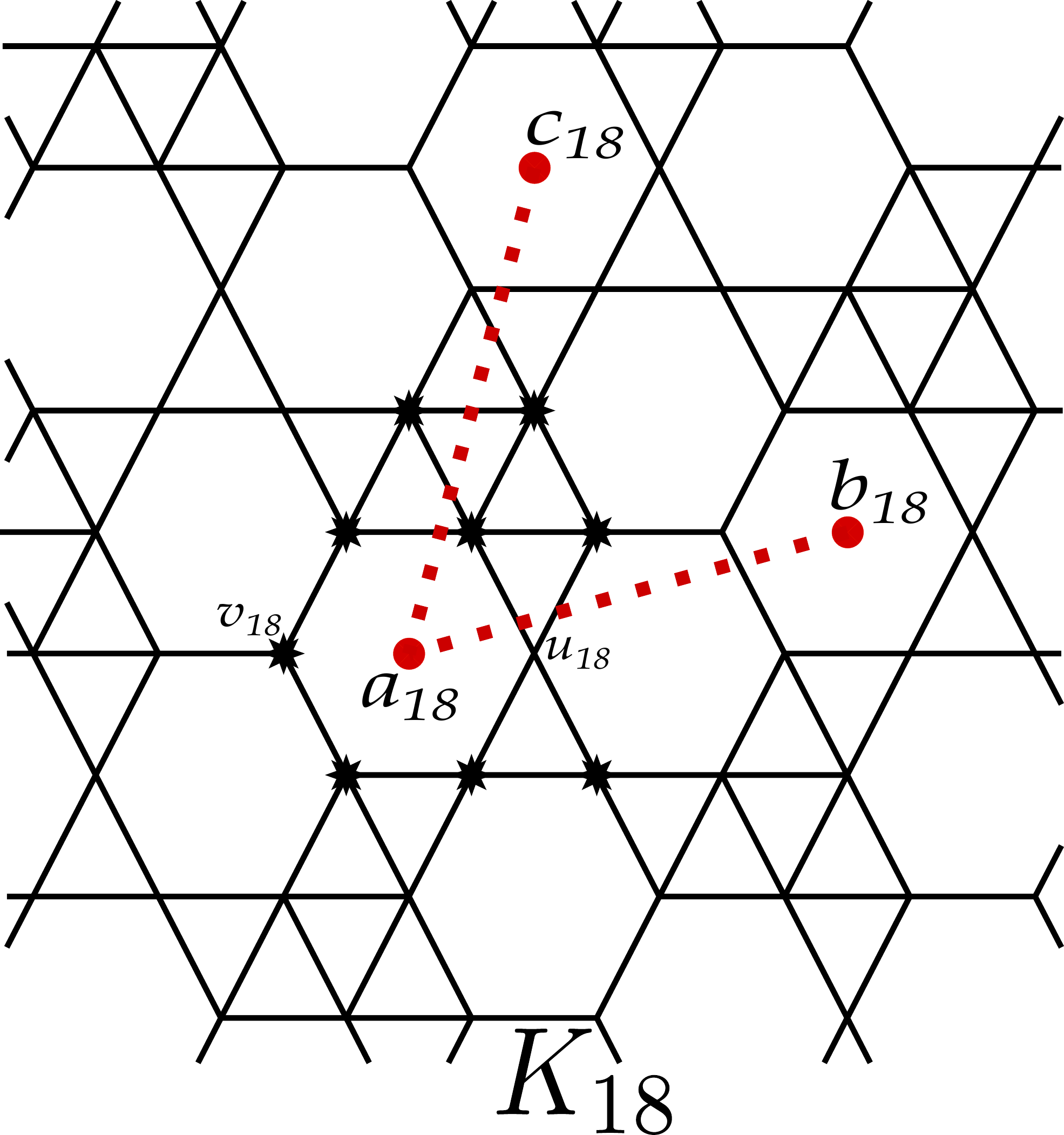}
    \includegraphics[height=2.9cm, width= 2.9cm]{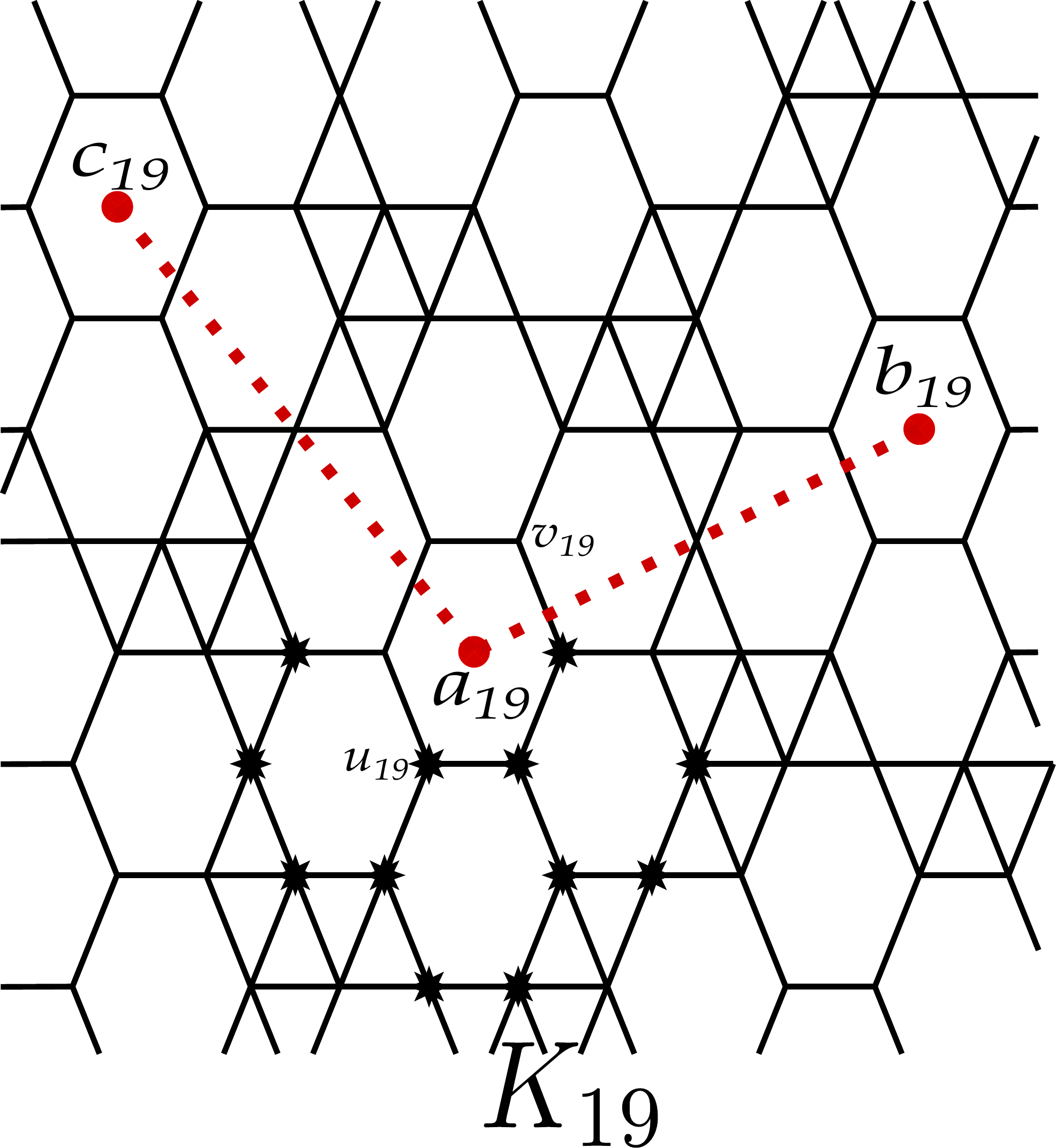}
    \includegraphics[height=2.9cm, width= 2.9cm]{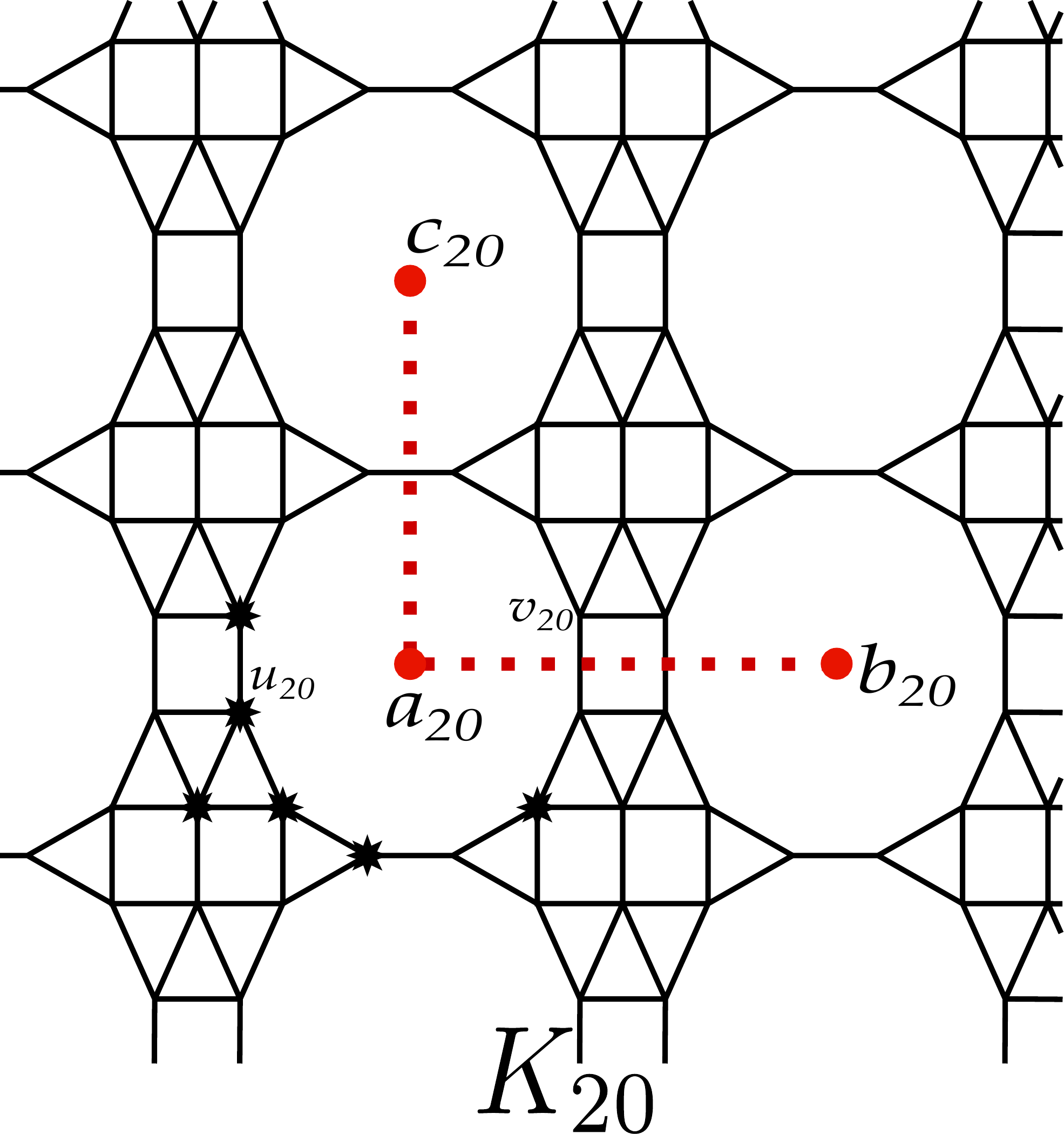}
     \vspace{-7mm}
\end{figure}
\begin{figure}[H]
    \centering
    \includegraphics[height=2.9cm, width= 2.9cm]{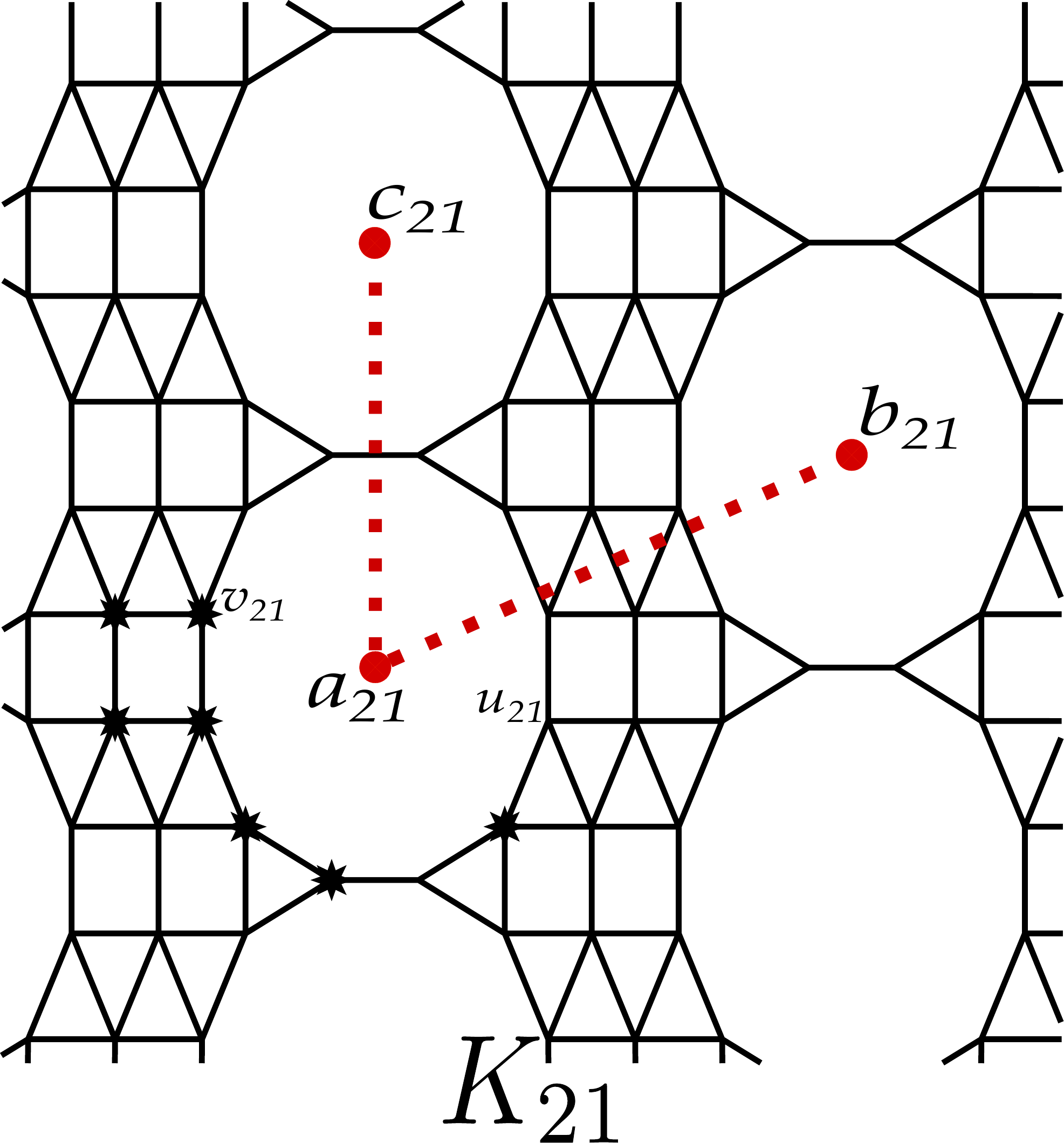}
    \includegraphics[height=2.9cm, width= 2.9cm]{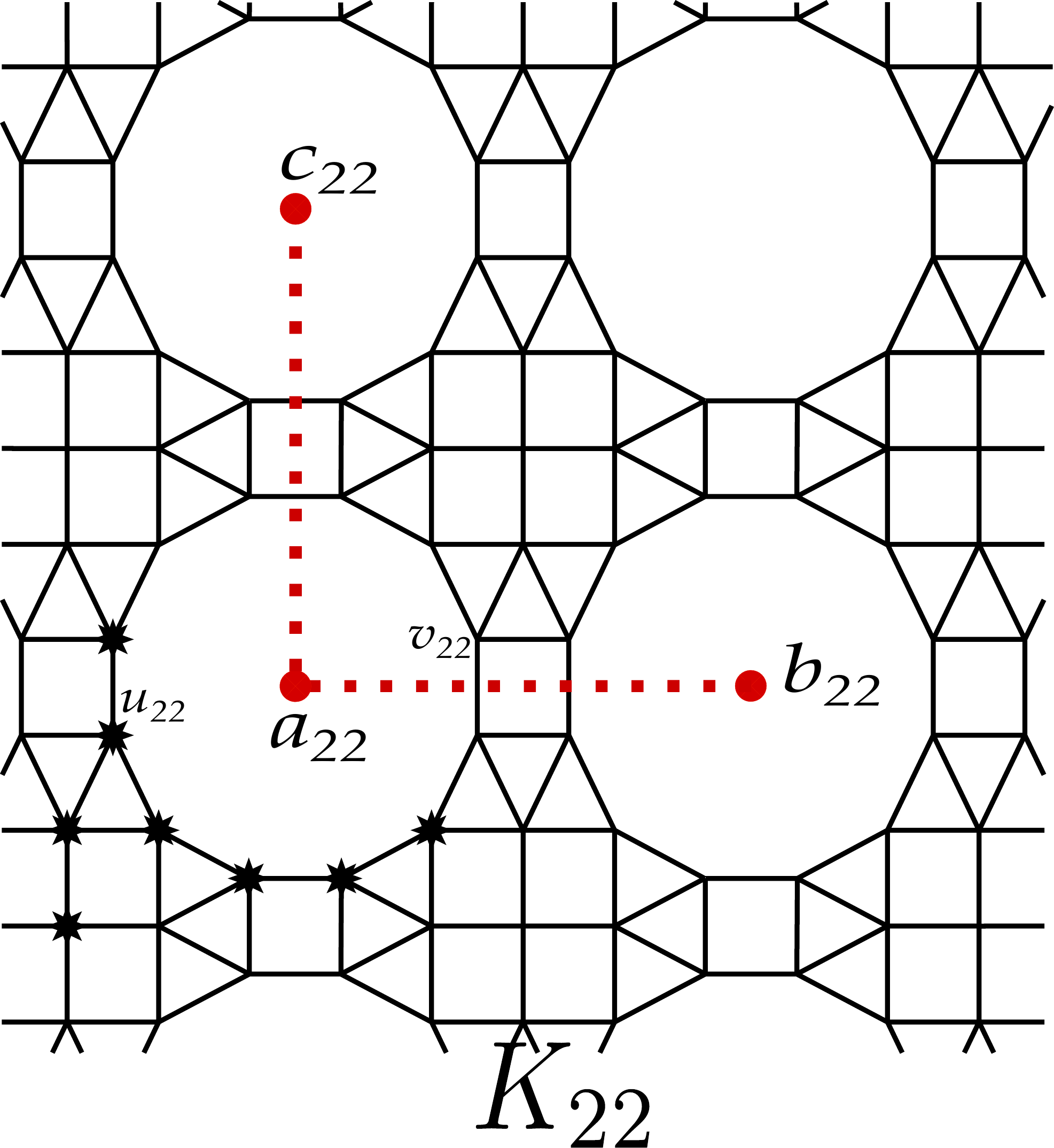}
    \includegraphics[height=2.9cm, width= 2.9cm]{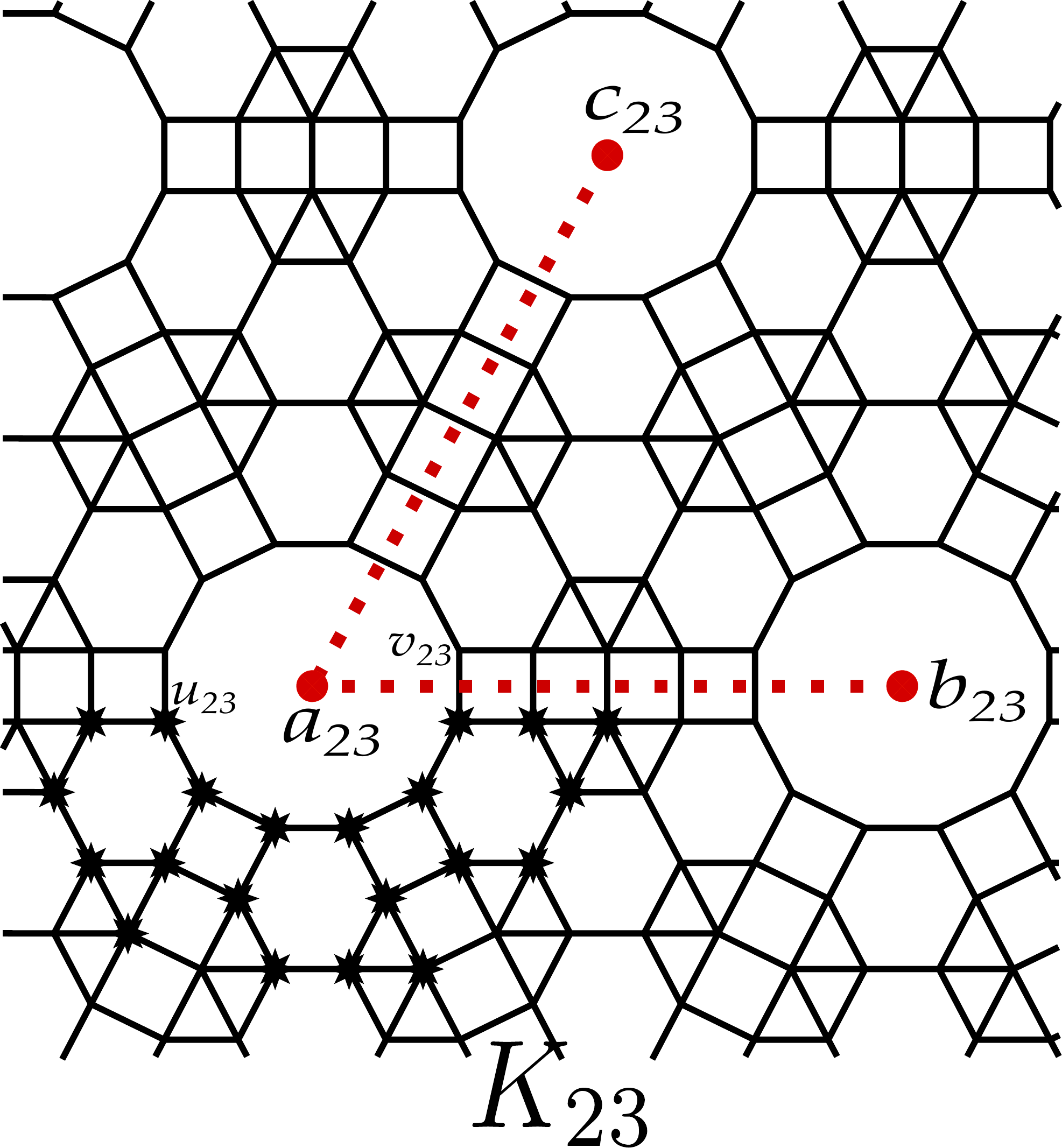}
    \includegraphics[height=2.9cm, width= 2.9cm]{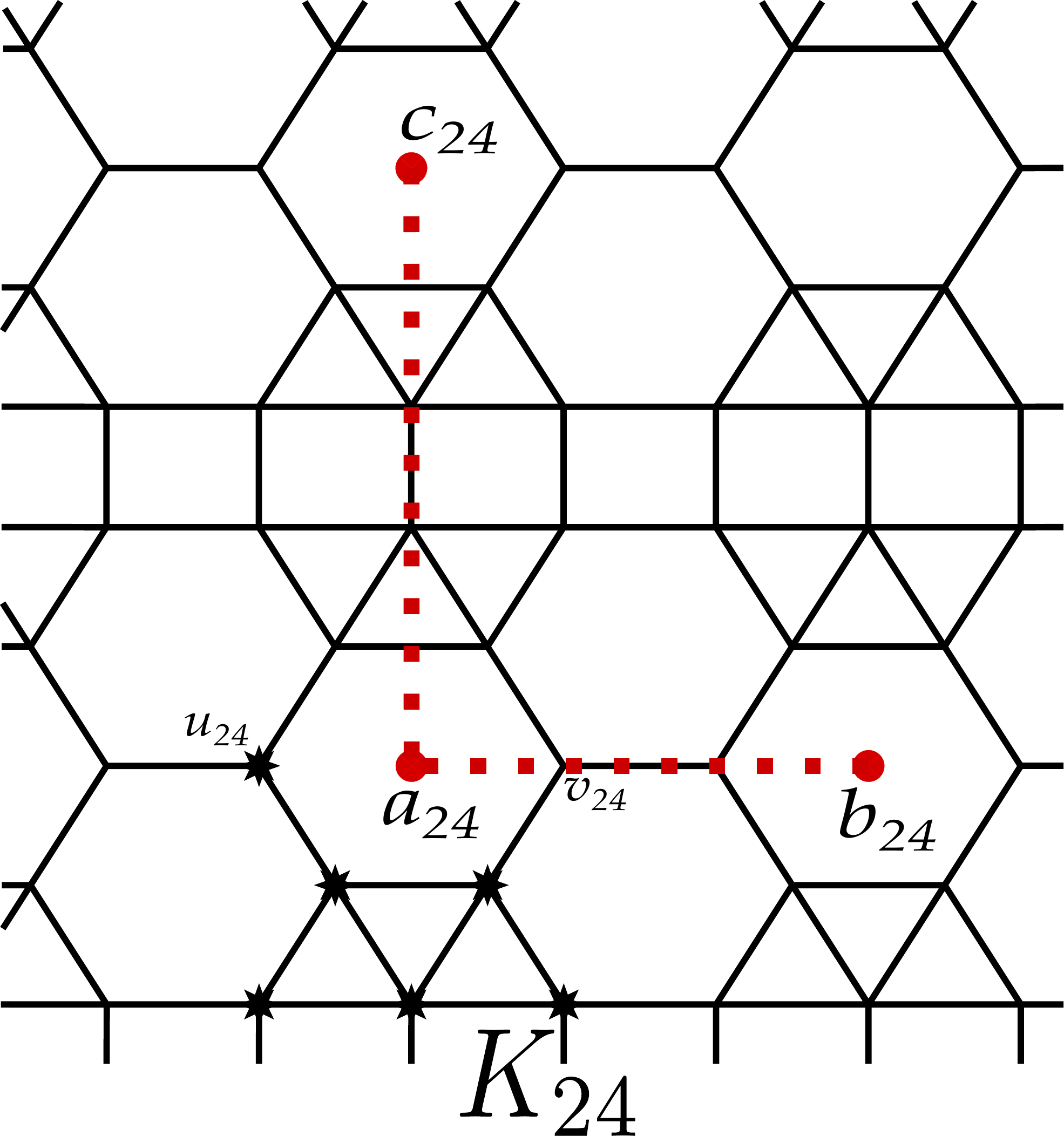}
    \includegraphics[height=2.9cm, width= 2.9cm]{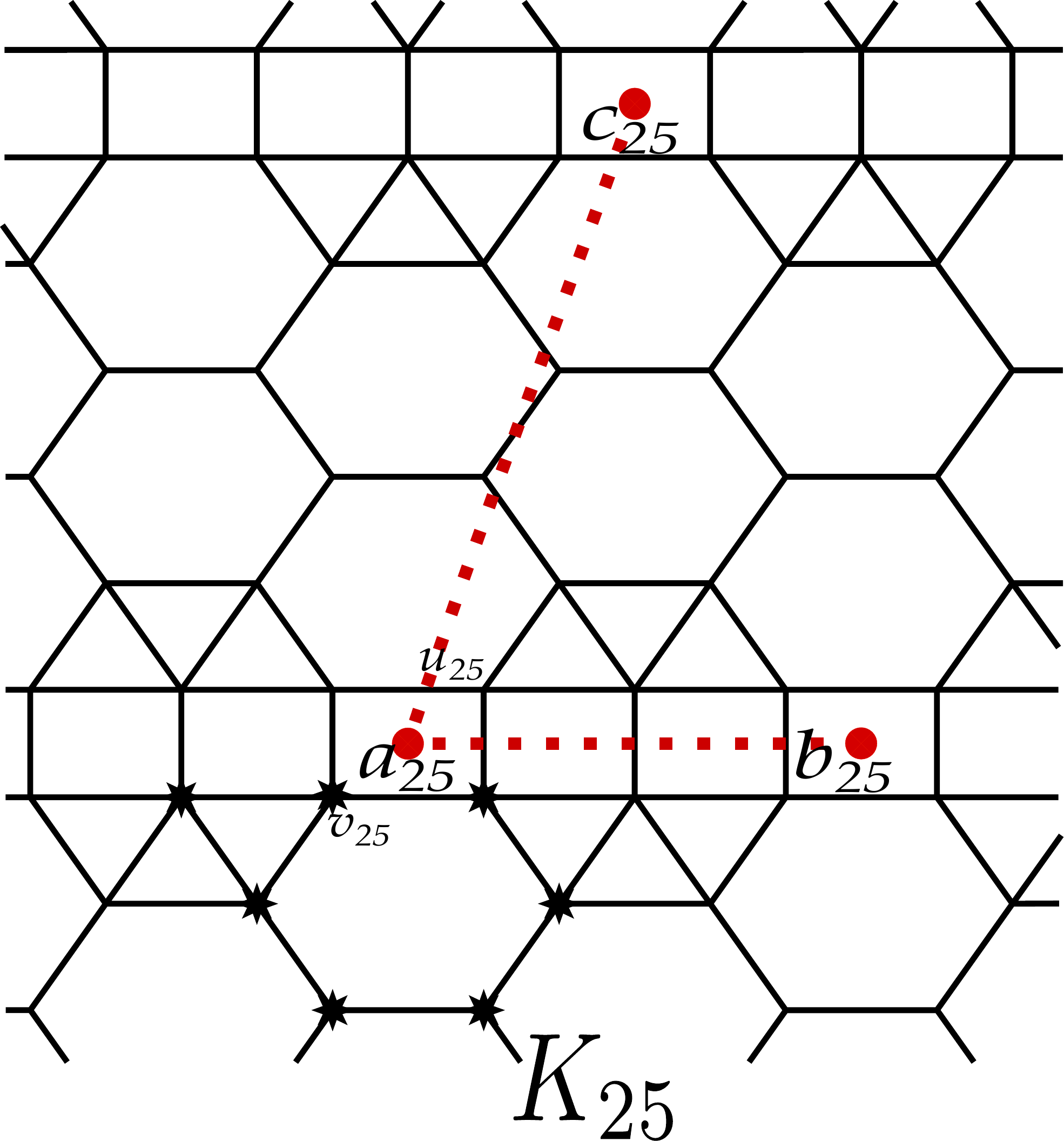}
     \vspace{-7mm}
\end{figure}
\begin{figure}[H]
    \centering
    \includegraphics[height=2.9cm, width= 2.9cm]{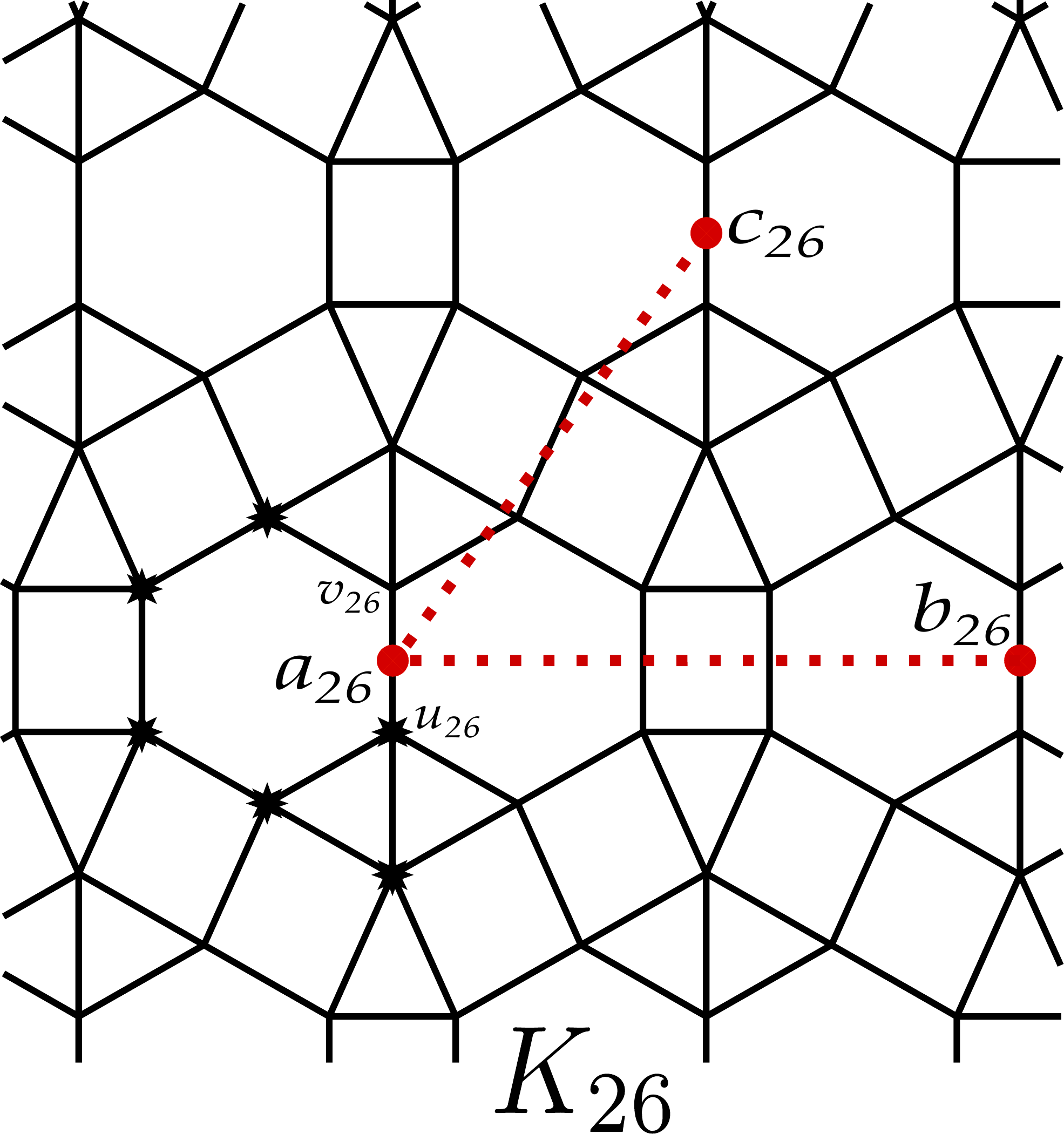}
    \includegraphics[height=2.9cm, width= 2.9cm]{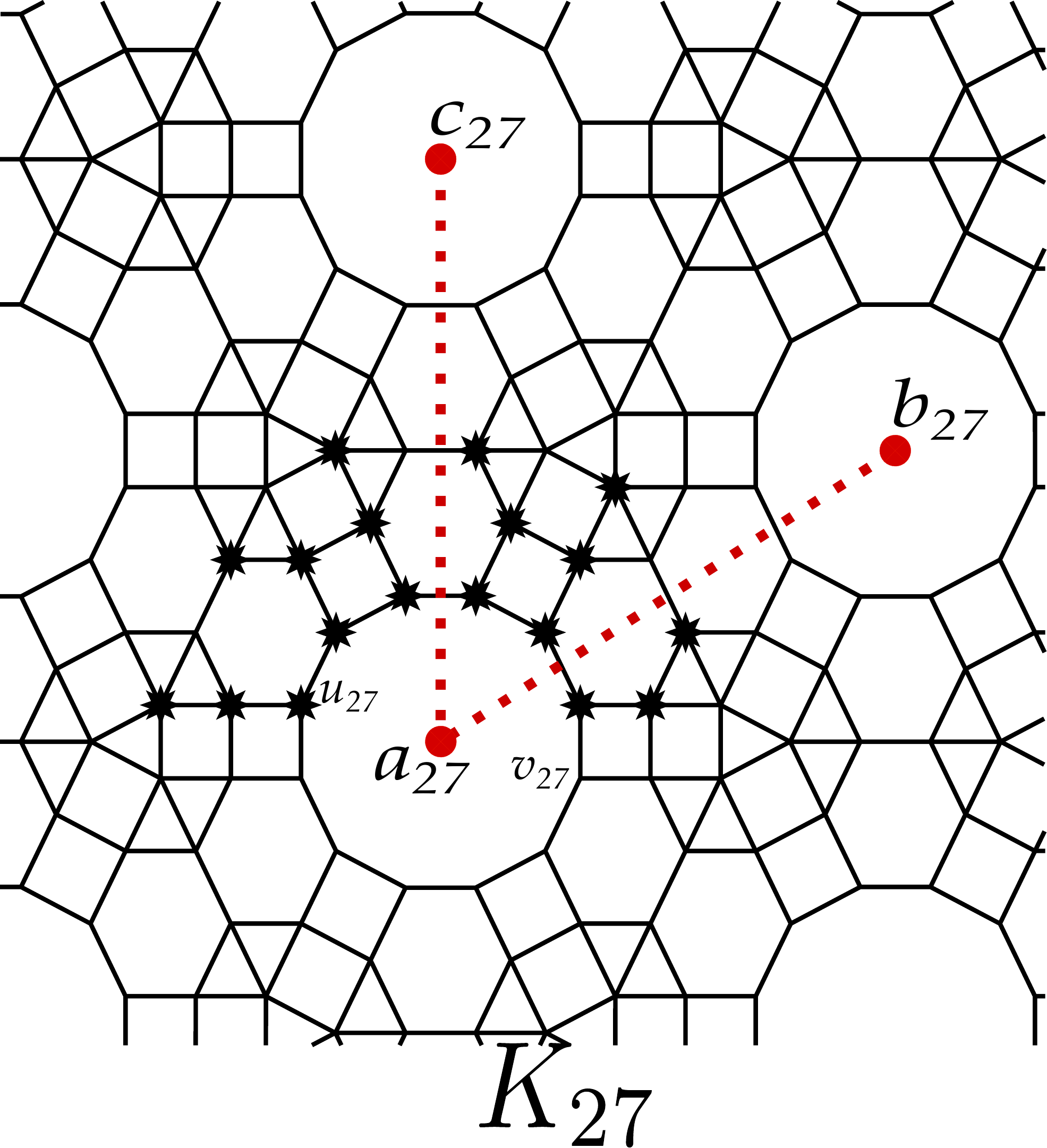}
    \includegraphics[height=2.9cm, width= 2.9cm]{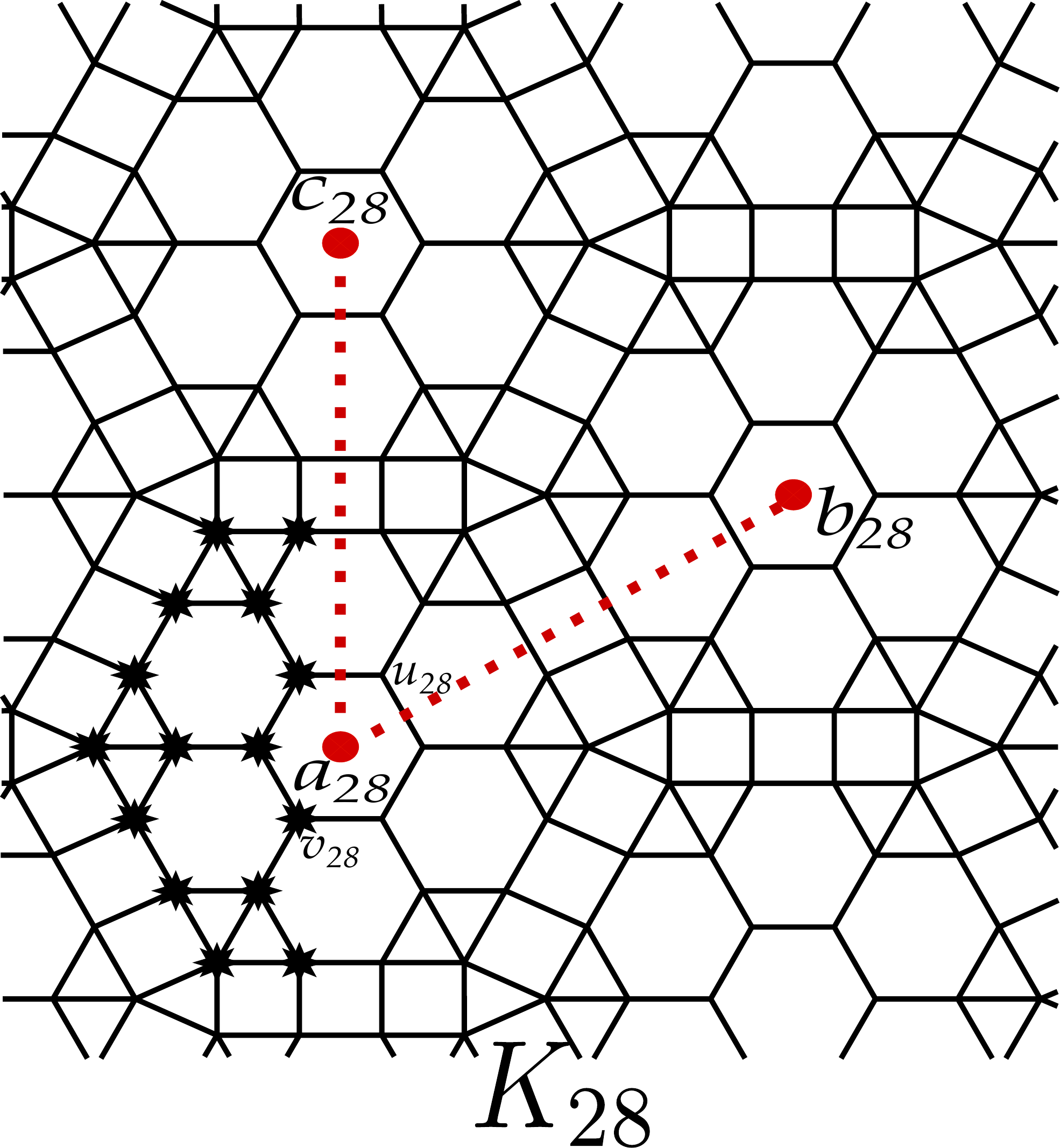}
    \includegraphics[height=2.9cm, width= 2.9cm]{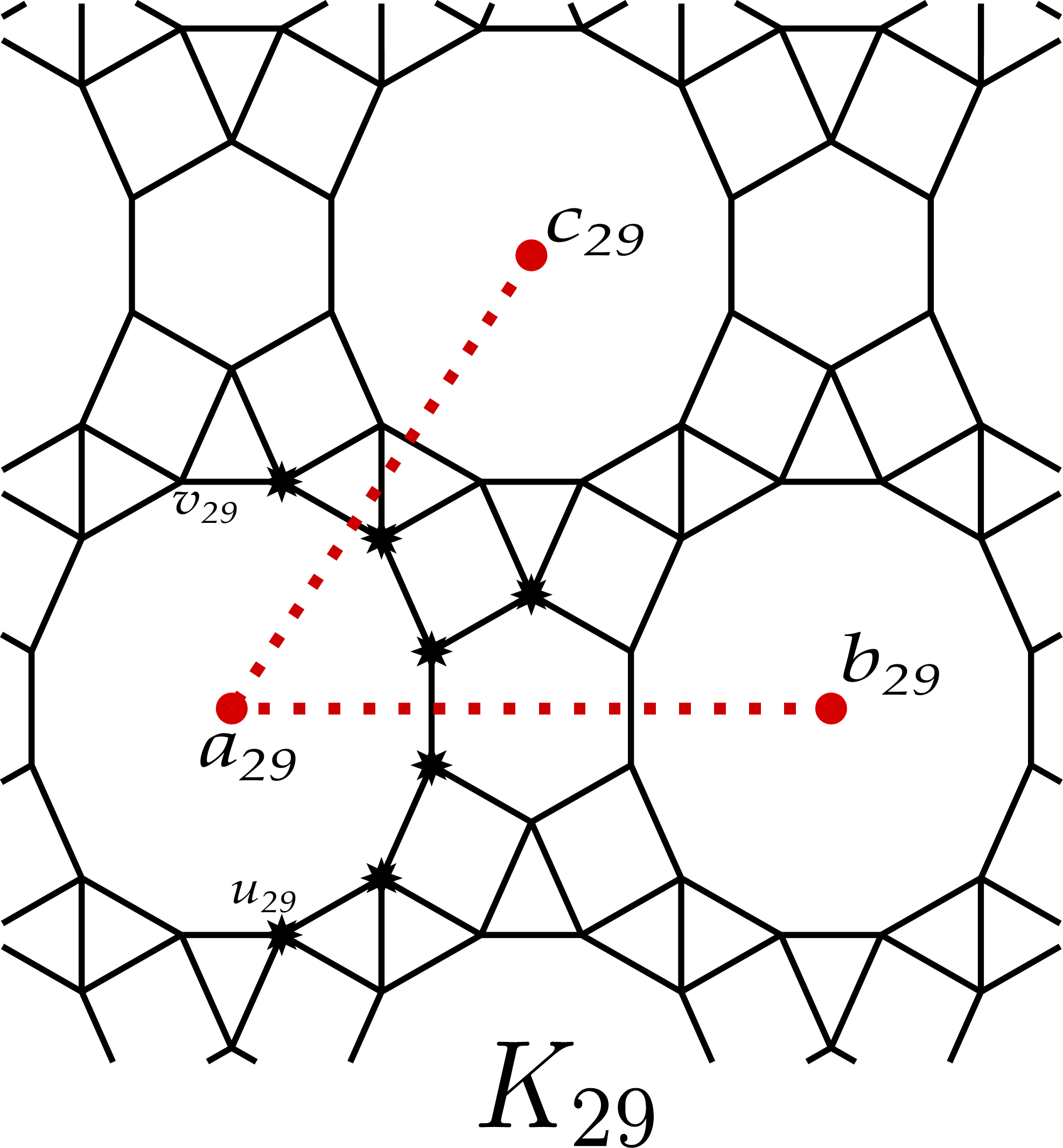}
    \includegraphics[height=2.9cm, width= 2.9cm]{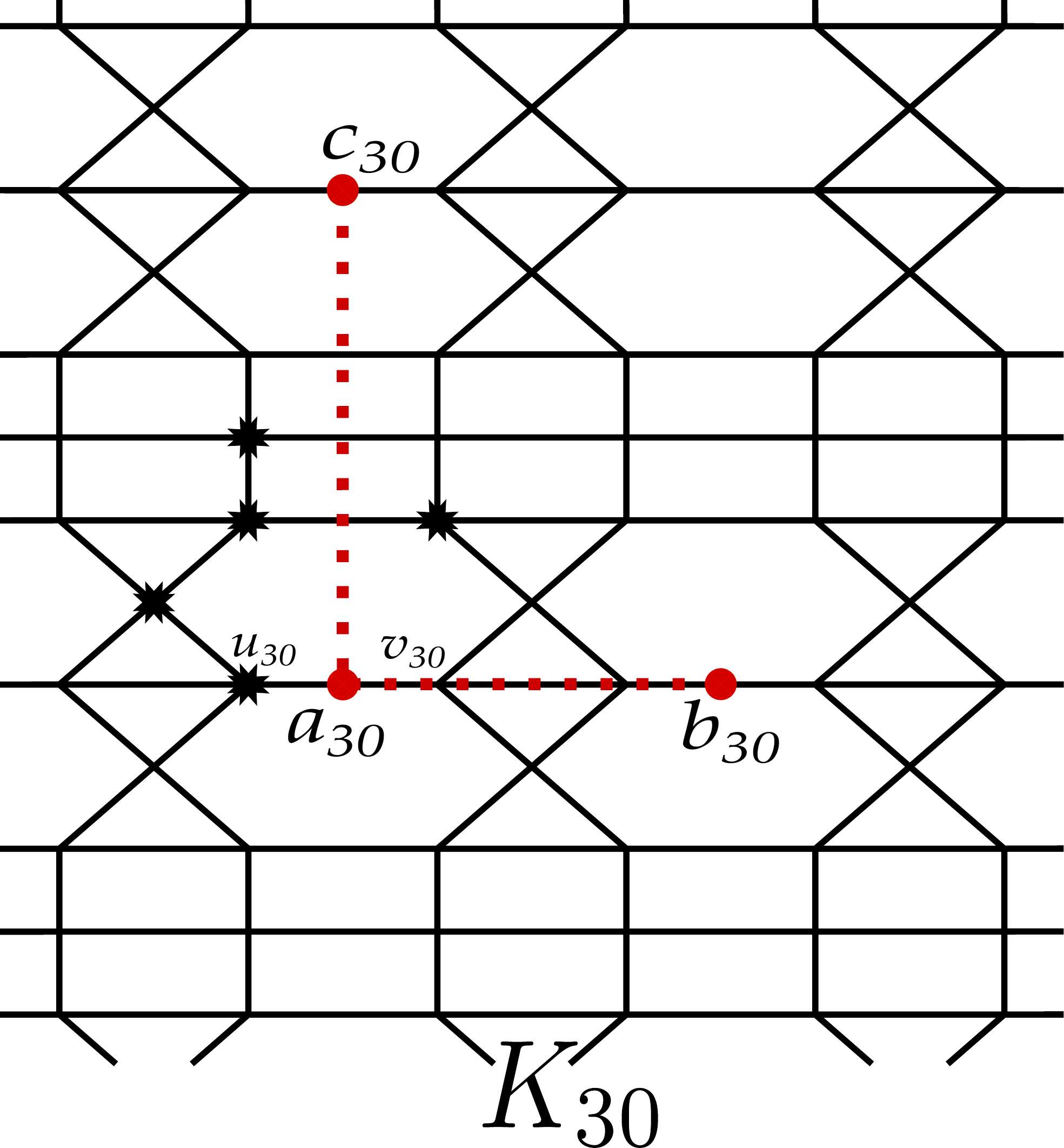}
     \vspace{-7mm}
\end{figure}
\begin{figure}[H]
    \centering
    \includegraphics[height=2.9cm, width= 2.9cm]{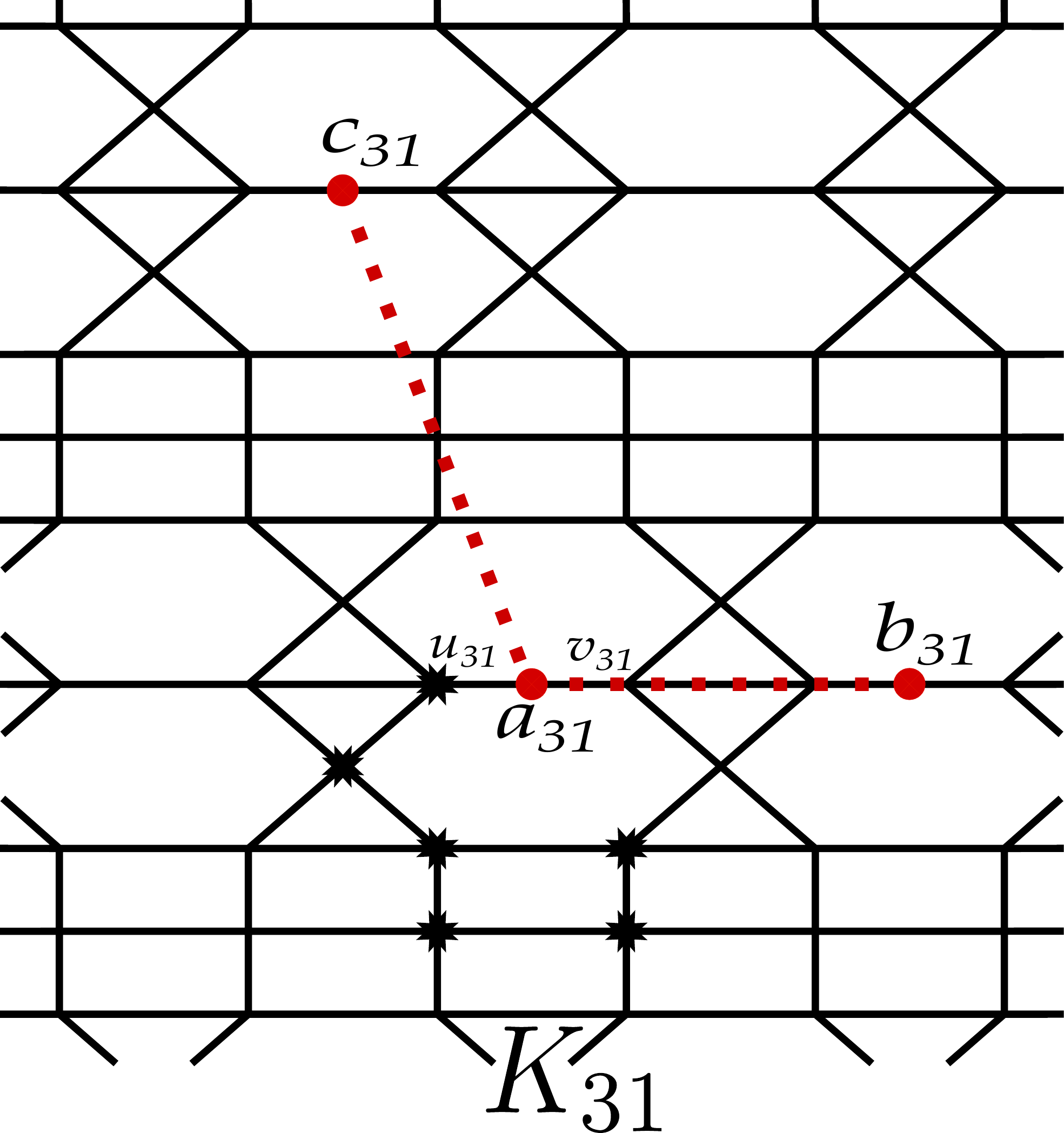}
    \includegraphics[height=2.9cm, width= 2.9cm]{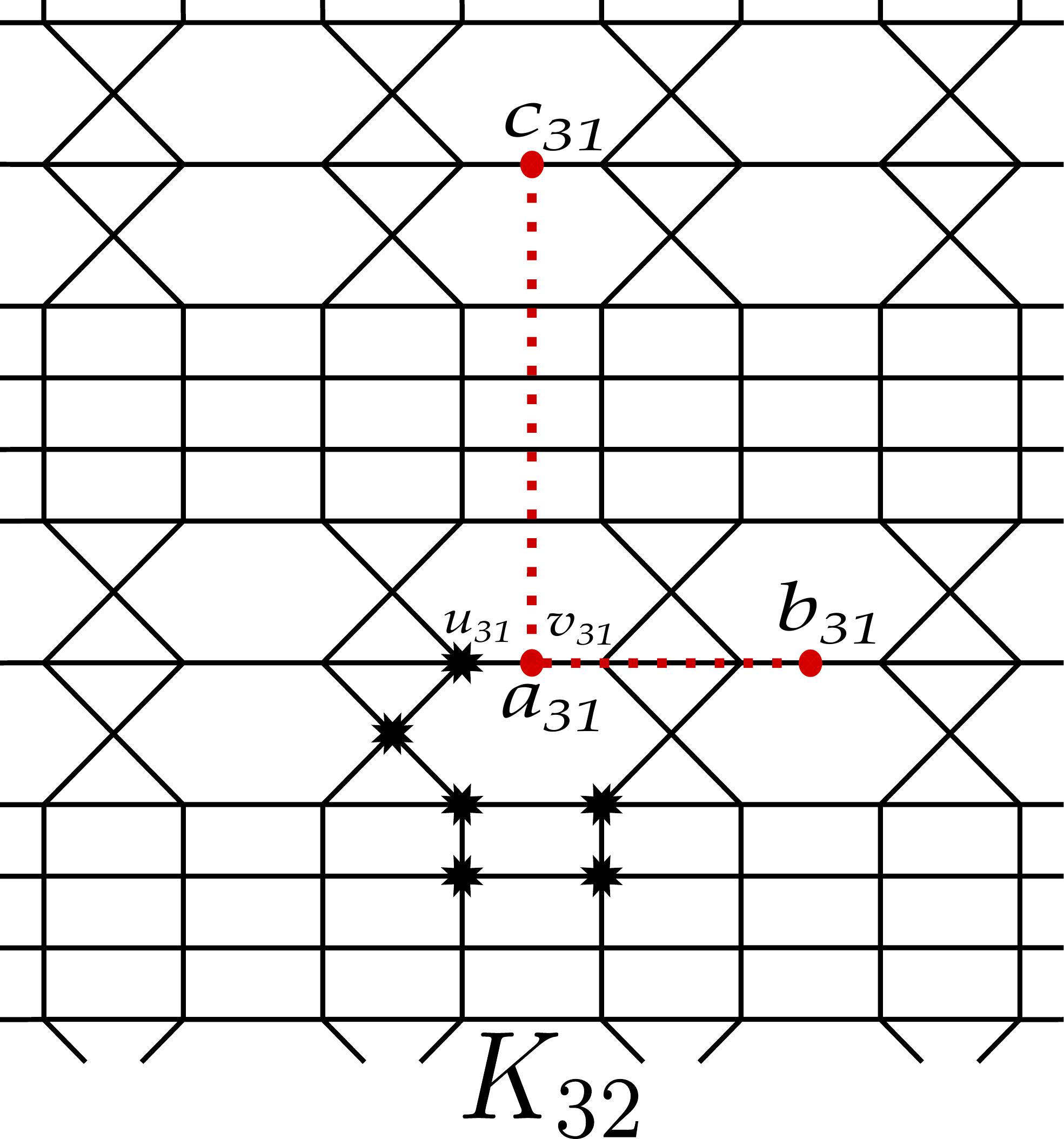}
    \includegraphics[height=2.9cm, width= 2.9cm]{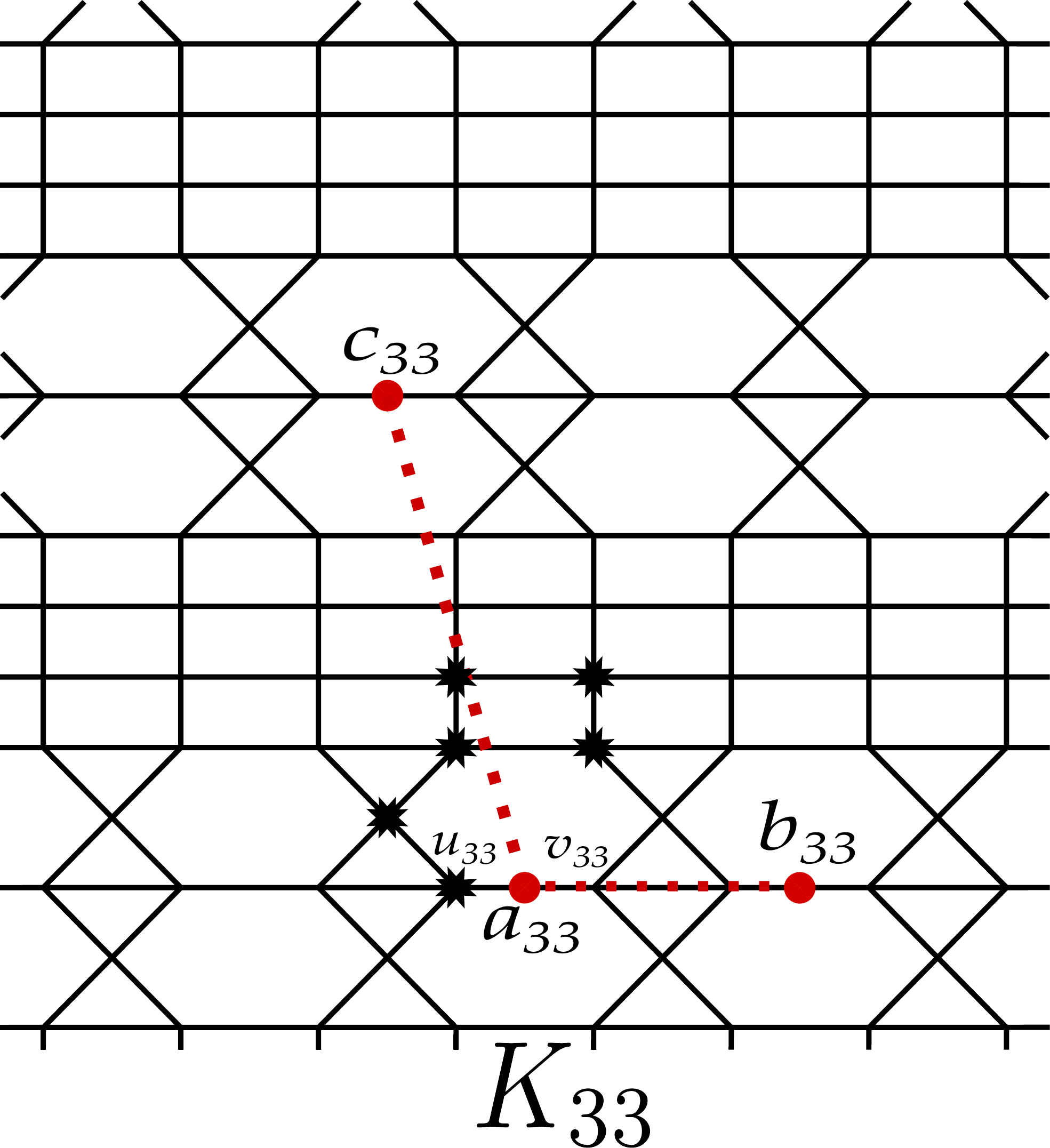}
    \includegraphics[height=2.9cm, width= 2.9cm]{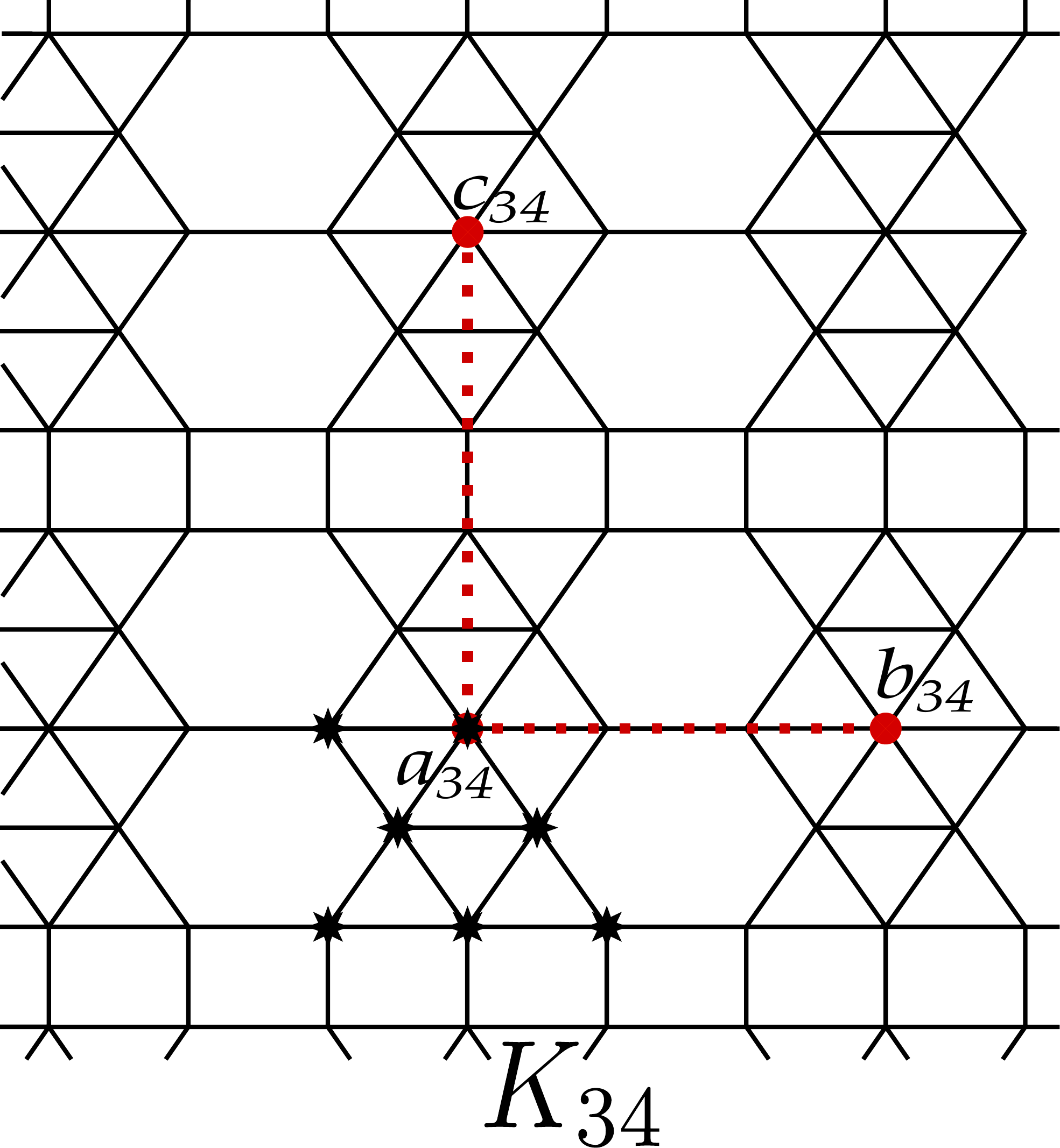}
    \includegraphics[height=2.9cm, width= 2.9cm]{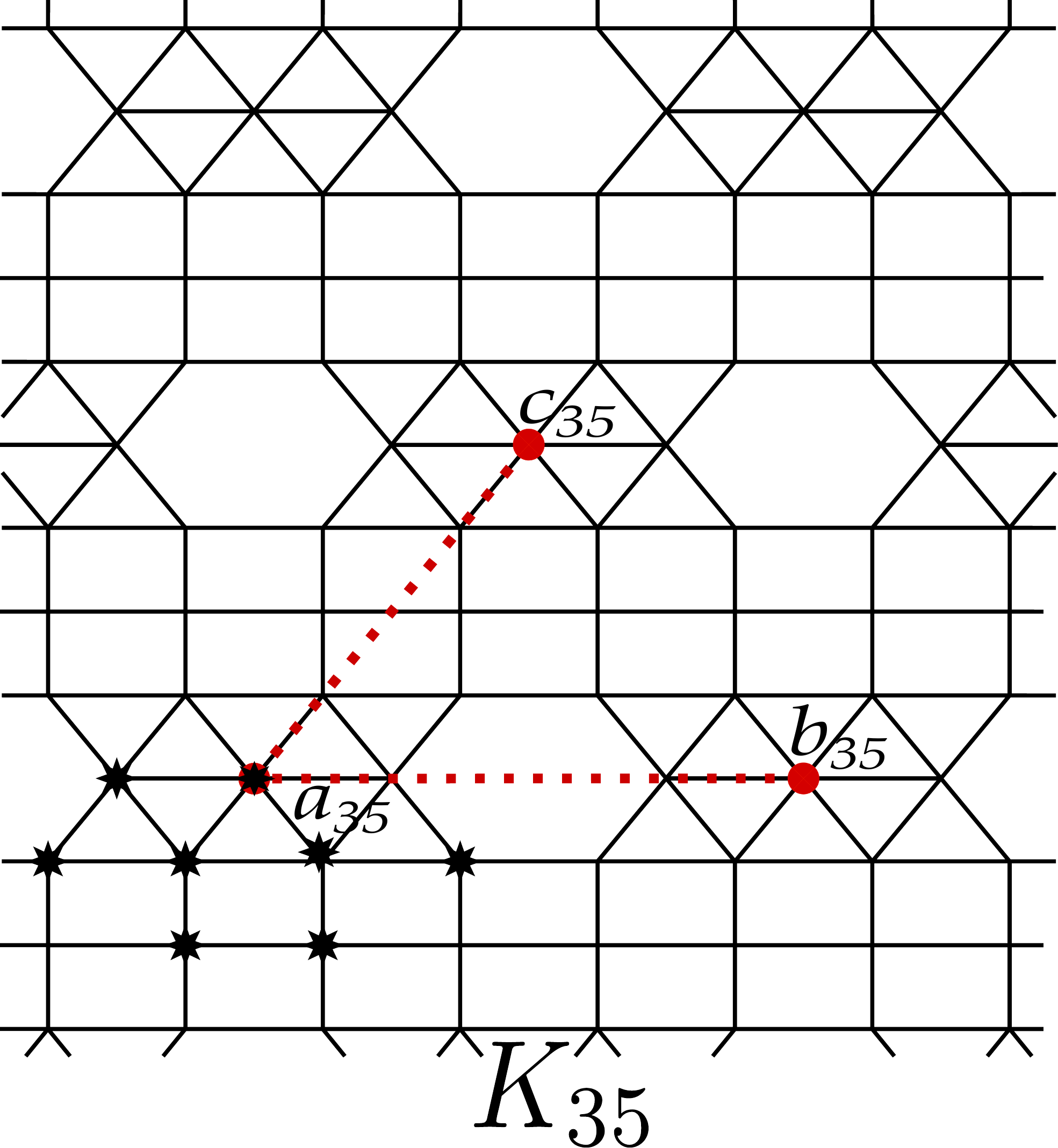}
 \vspace{-7mm}
\end{figure}
\begin{figure}[H]
    \centering
    \includegraphics[height=2.9cm, width= 2.9cm]{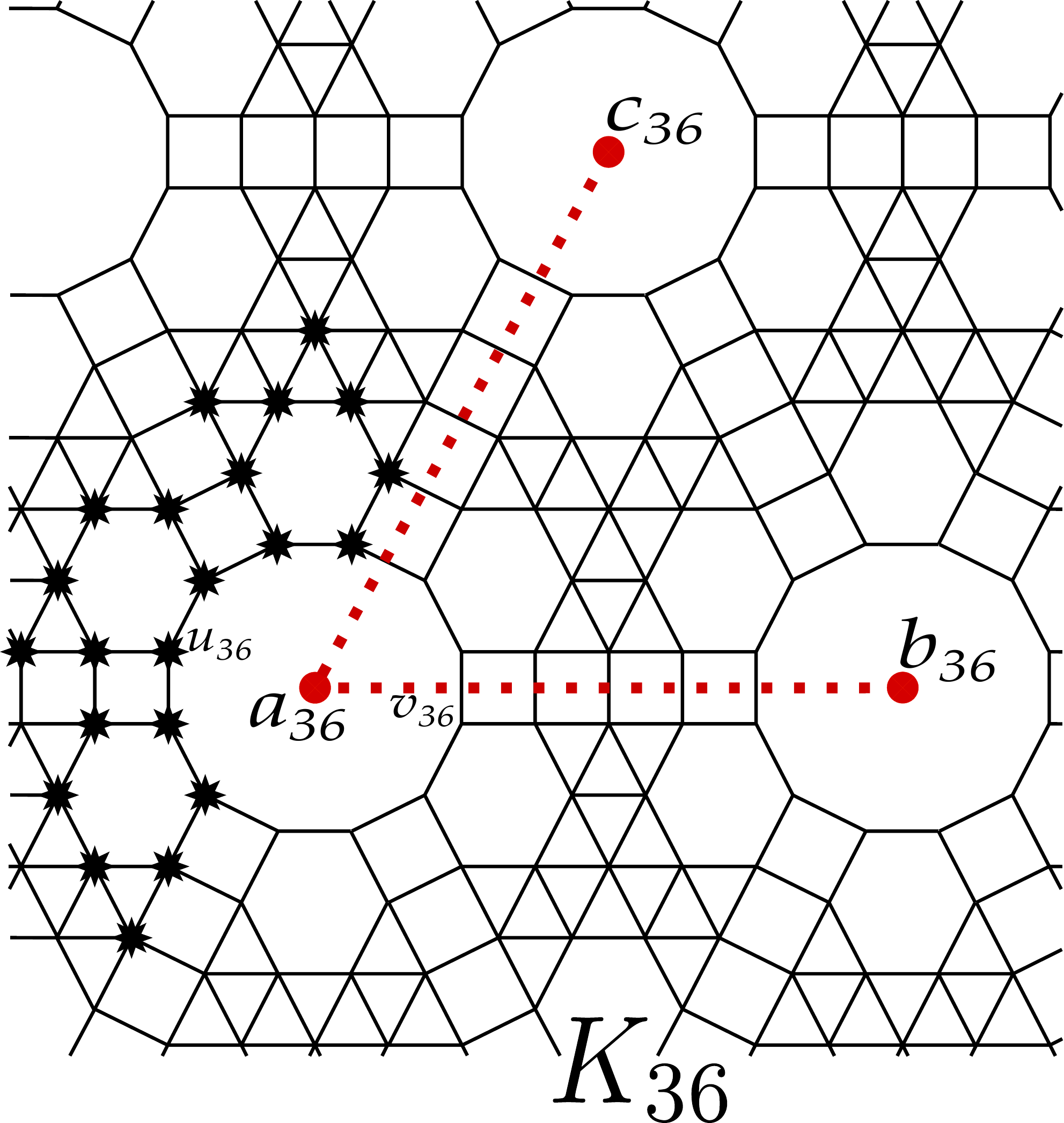}
    \includegraphics[height=2.9cm, width= 2.9cm]{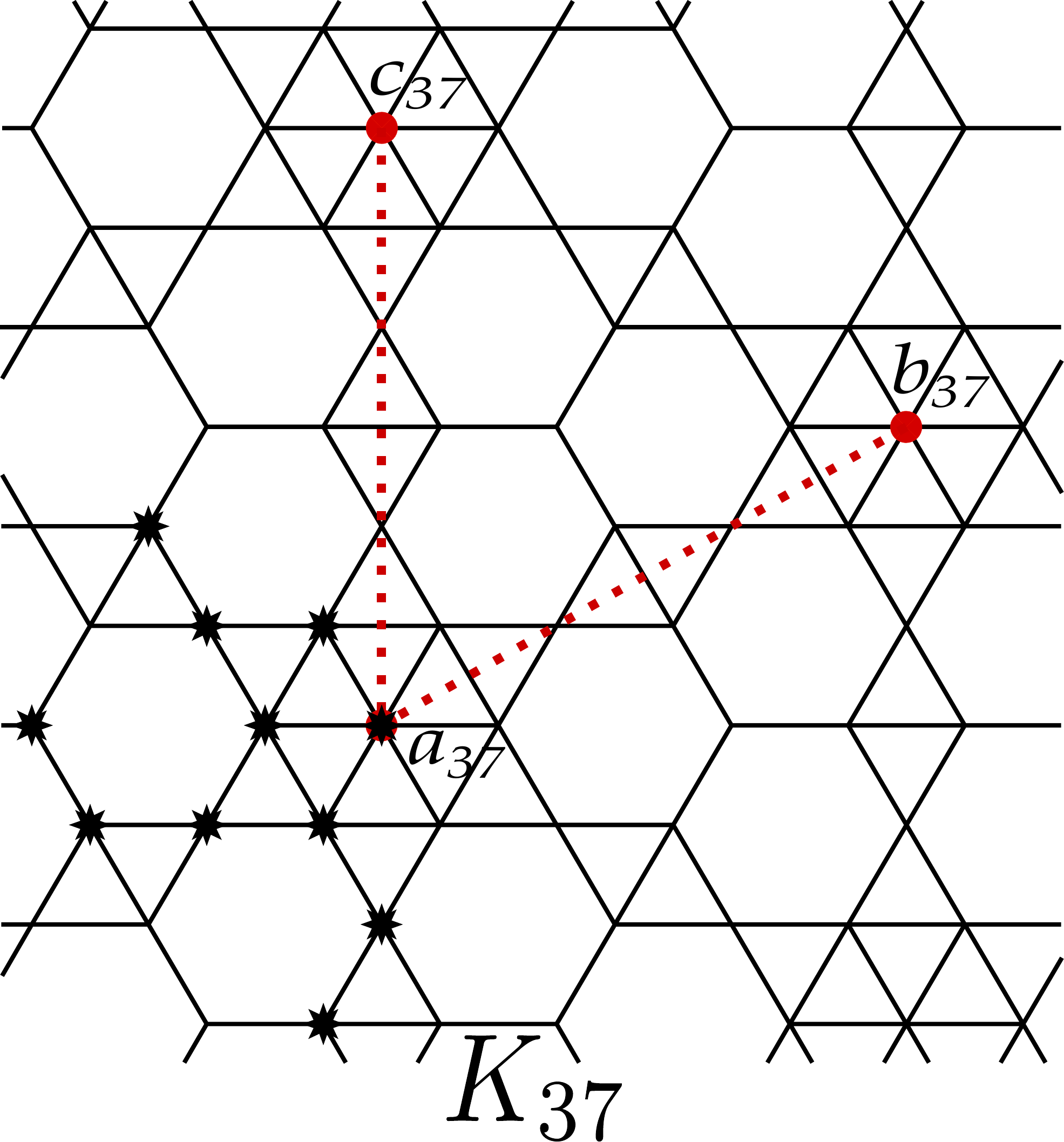}
    \includegraphics[height=2.9cm, width= 2.9cm]{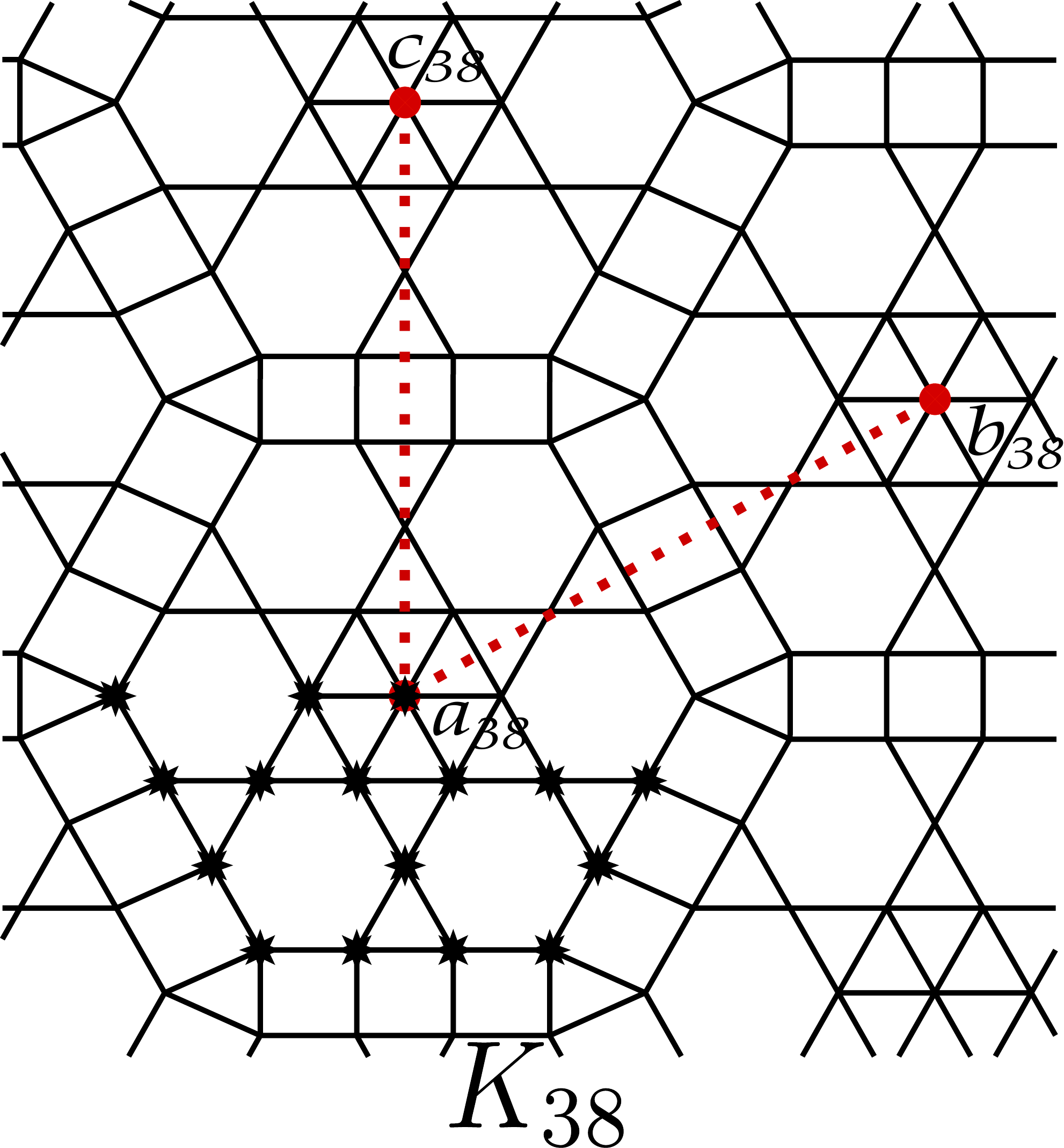}
    \includegraphics[height=2.9cm, width= 2.9cm]{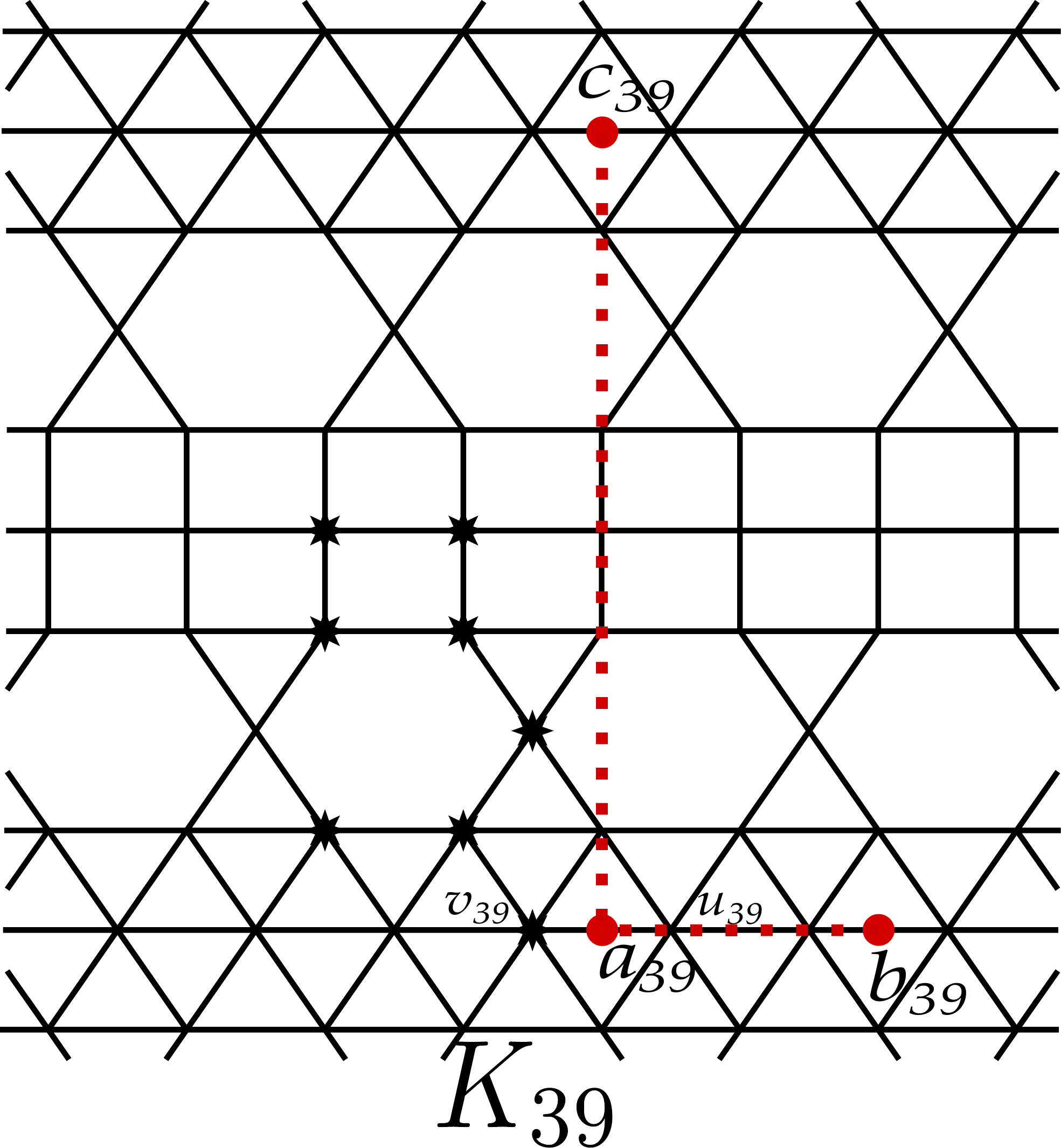}
    \includegraphics[height=2.9cm, width= 2.9cm]{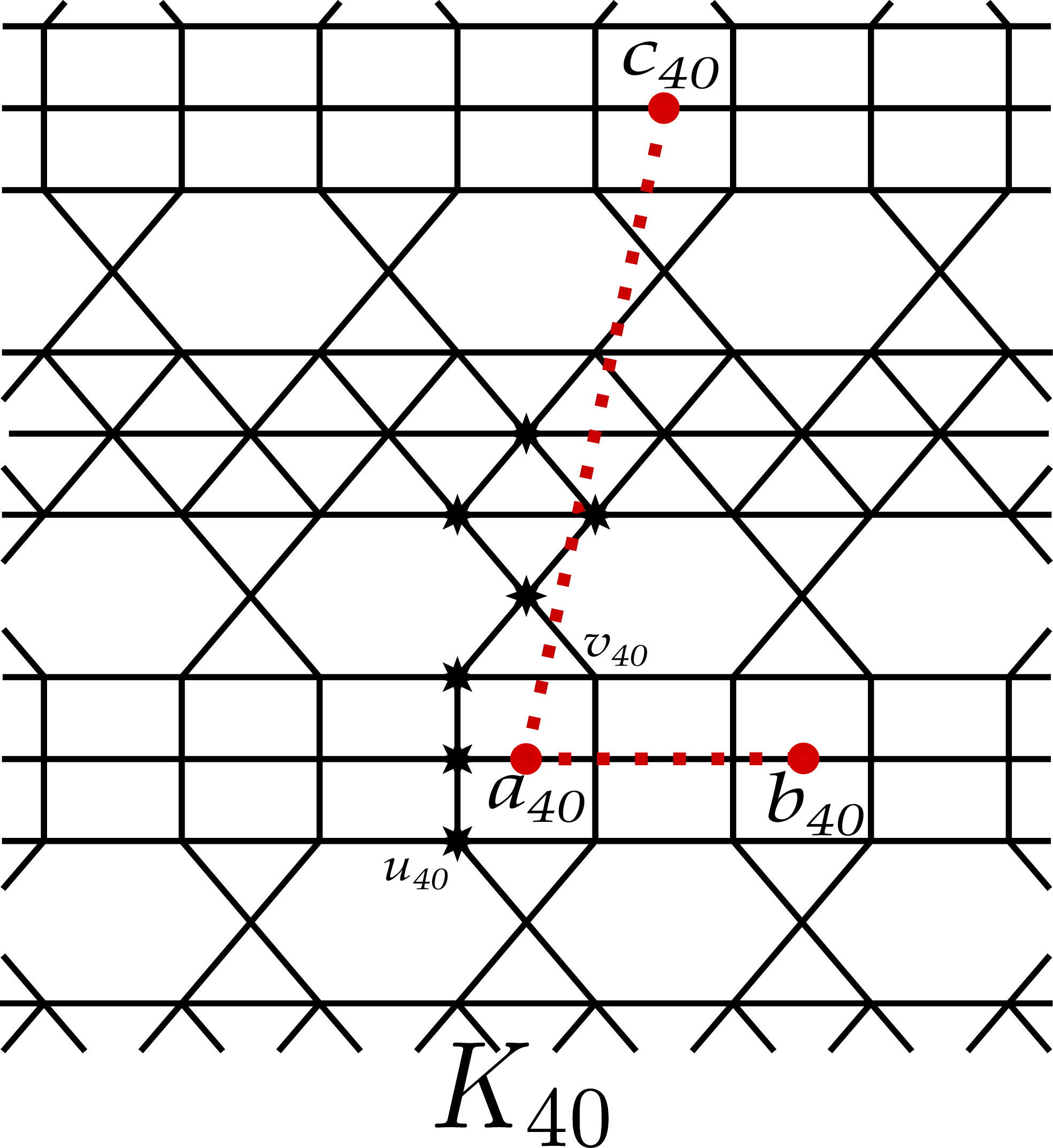}
    \end{figure}
     \vspace{-7mm}
\begin{figure}[H]
    \centering
    \includegraphics[height=2.9cm, width= 2.9cm]{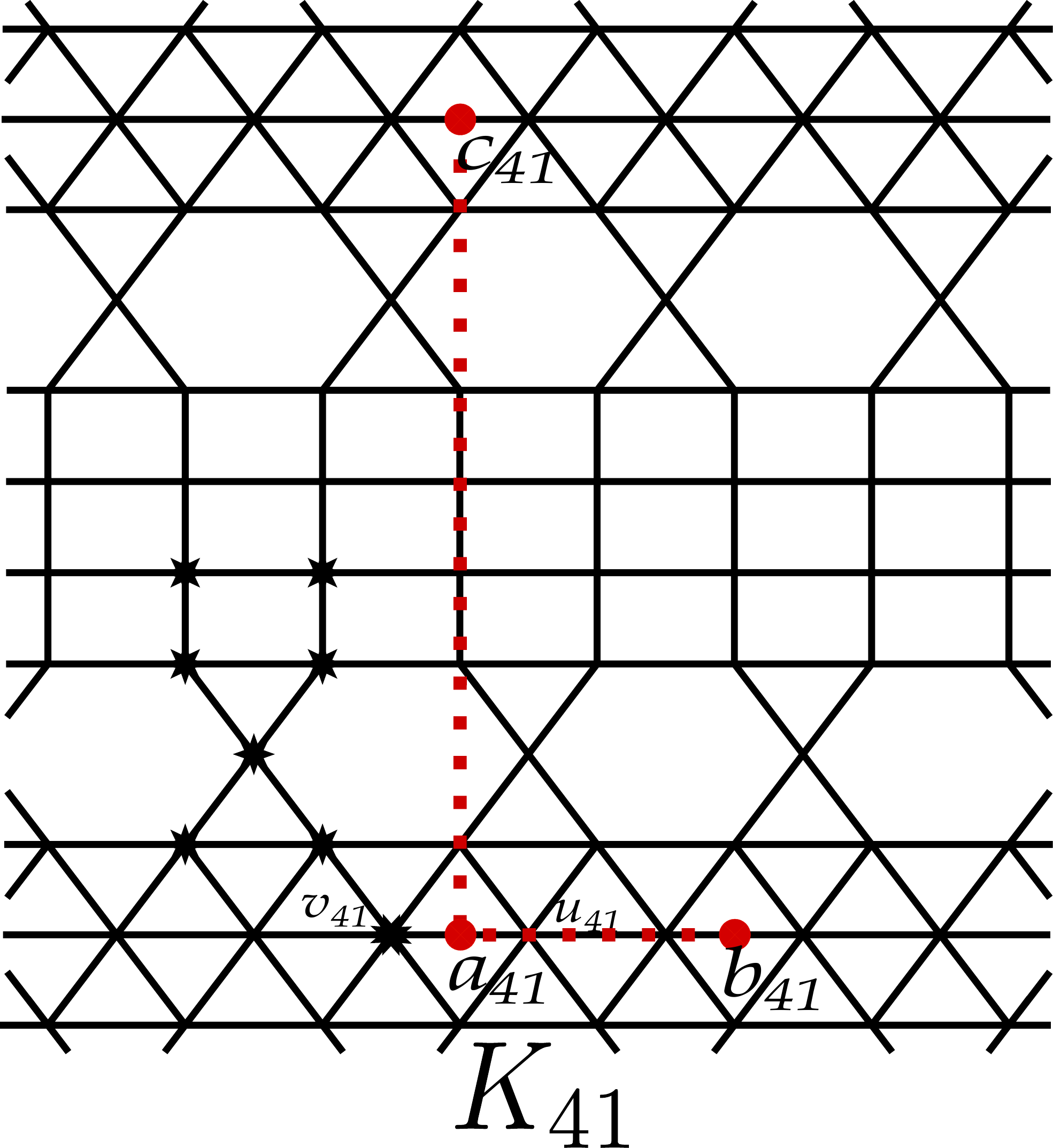}
    \includegraphics[height=2.9cm, width= 2.9cm]{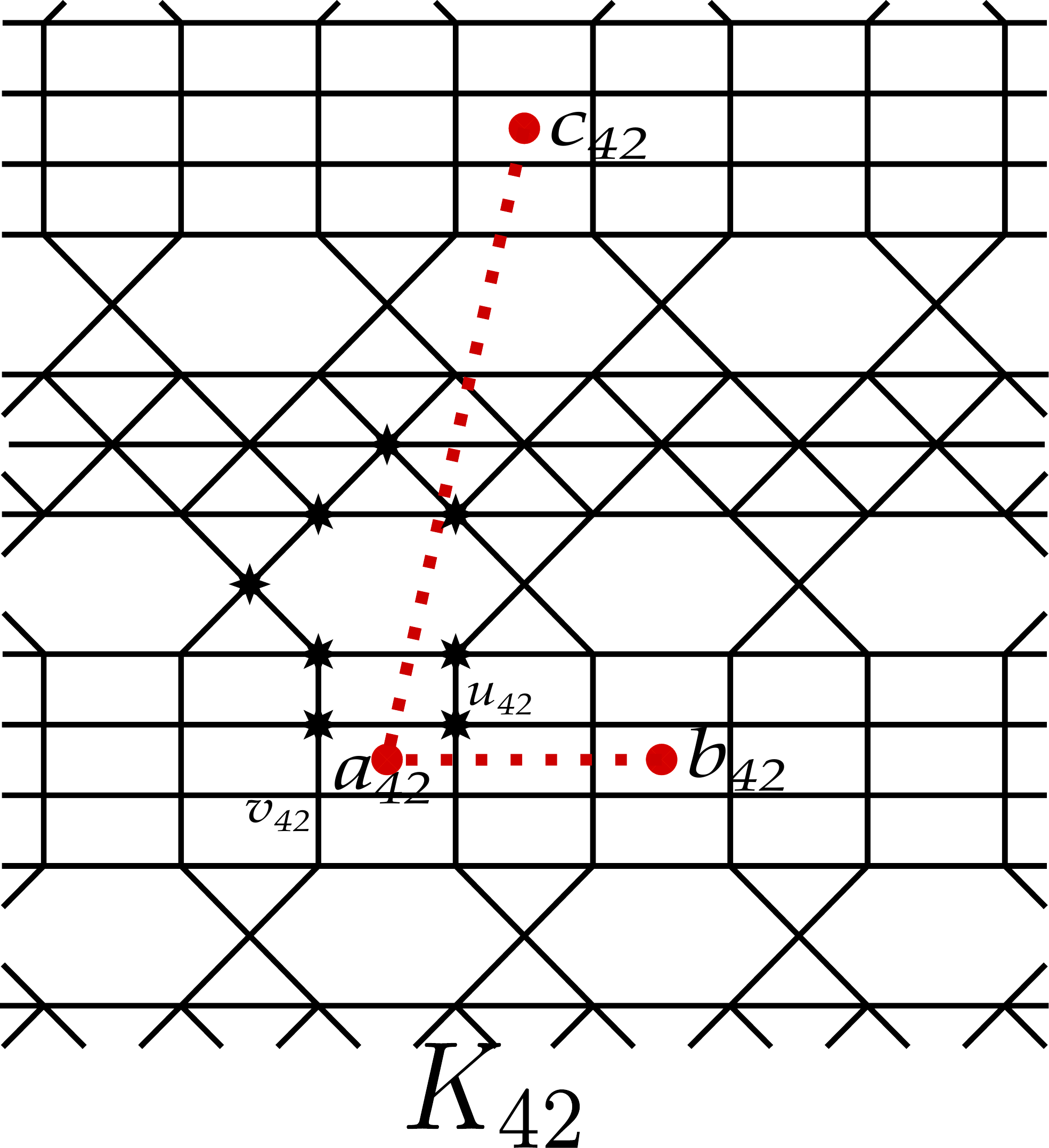}
    \includegraphics[height=2.9cm, width= 2.9cm]{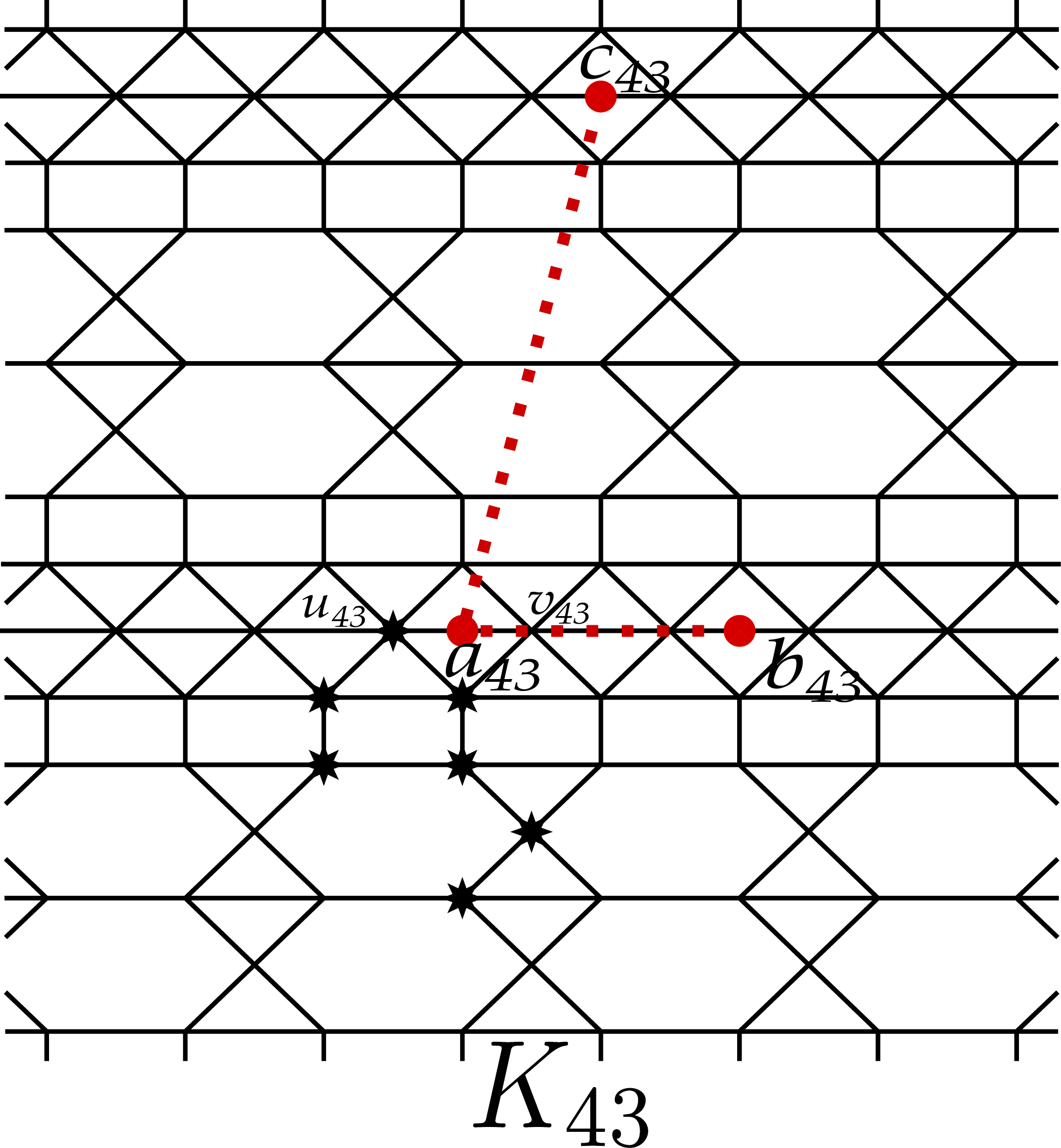}
    \includegraphics[height=2.9cm, width= 2.9cm]{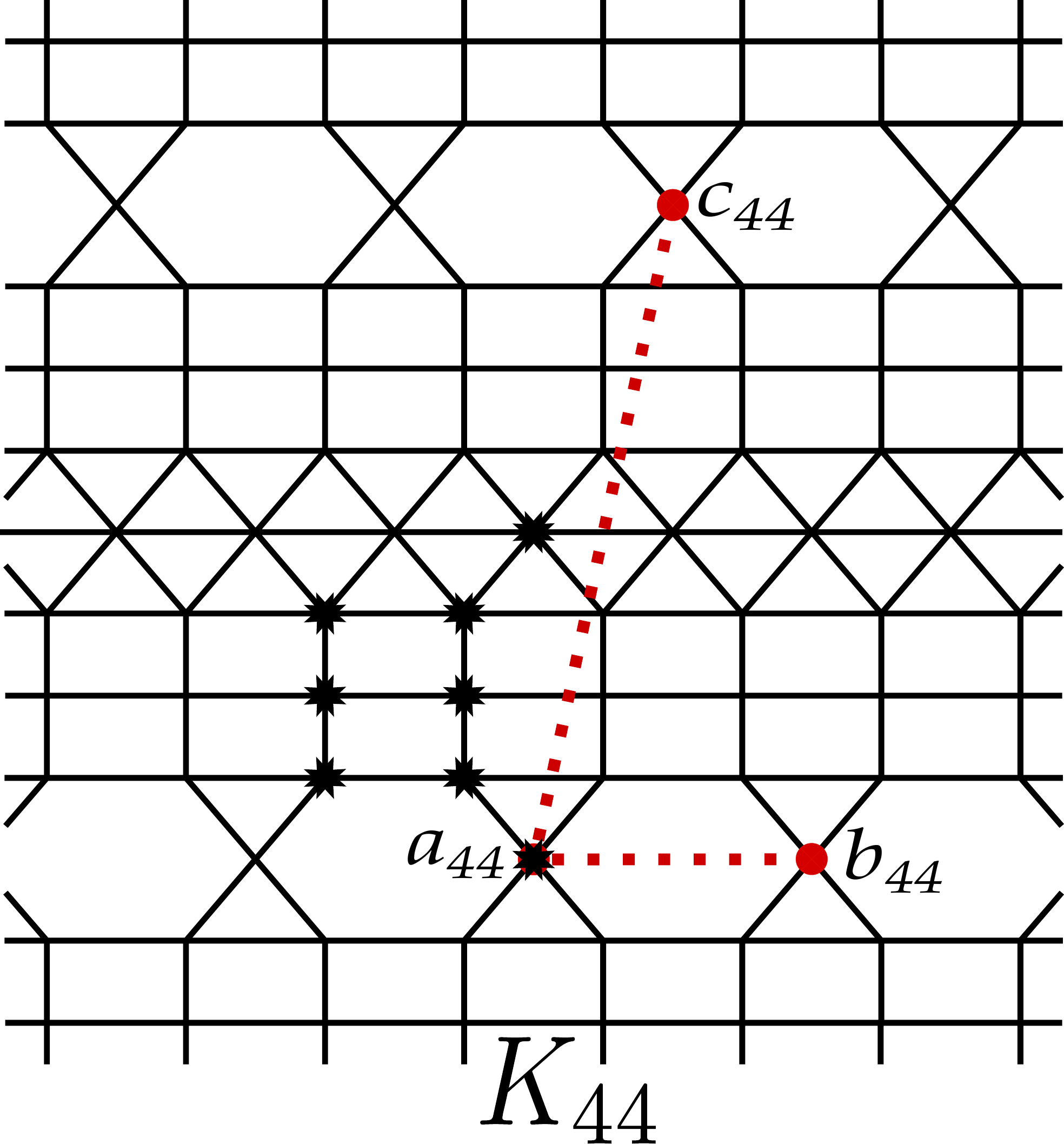}
    \includegraphics[height=2.9cm, width= 2.9cm]{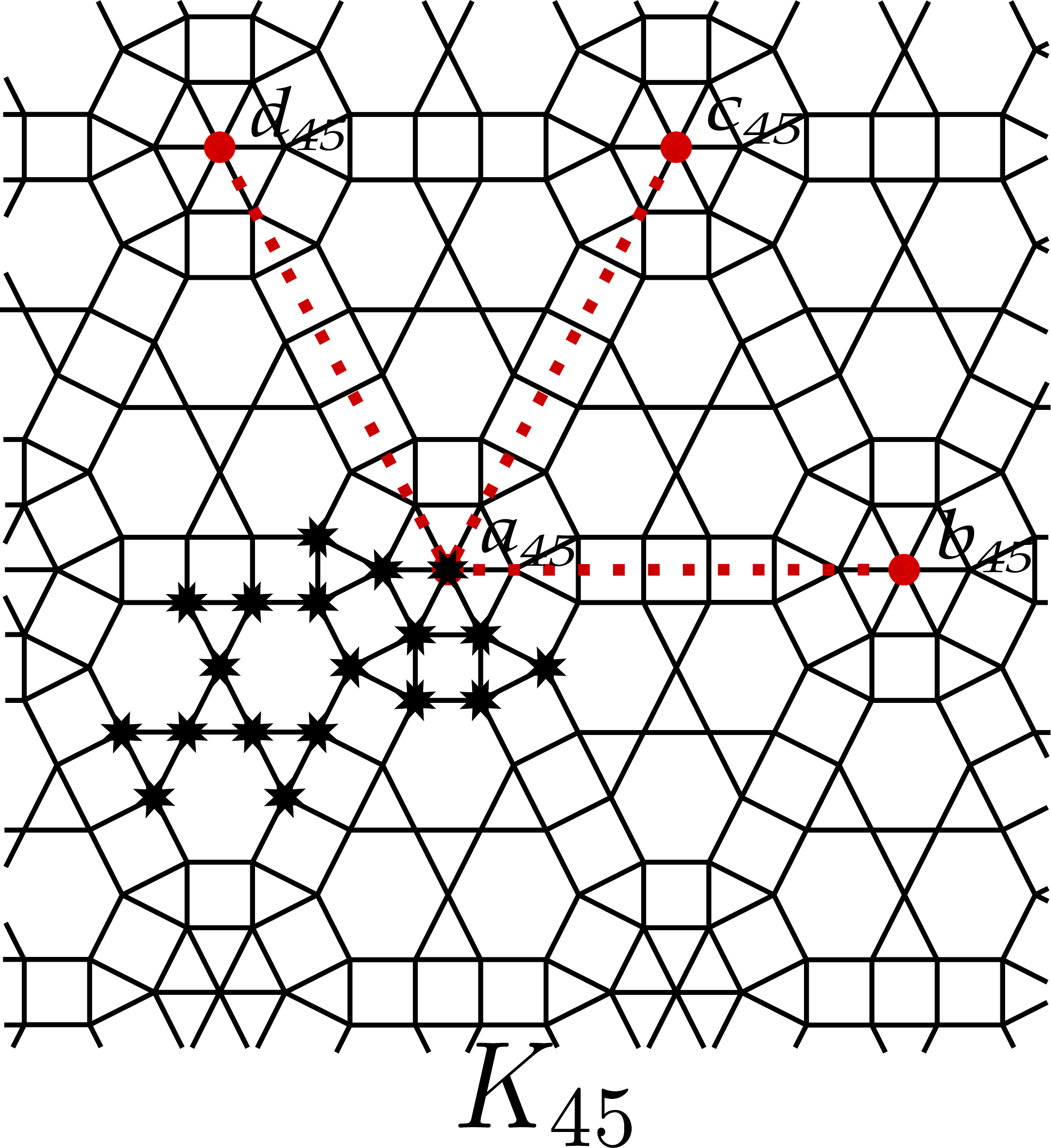}
    \end{figure}    
     \vspace{-7mm}
\begin{figure}[H]
    \centering
    \includegraphics[height=2.9cm, width= 2.9cm]{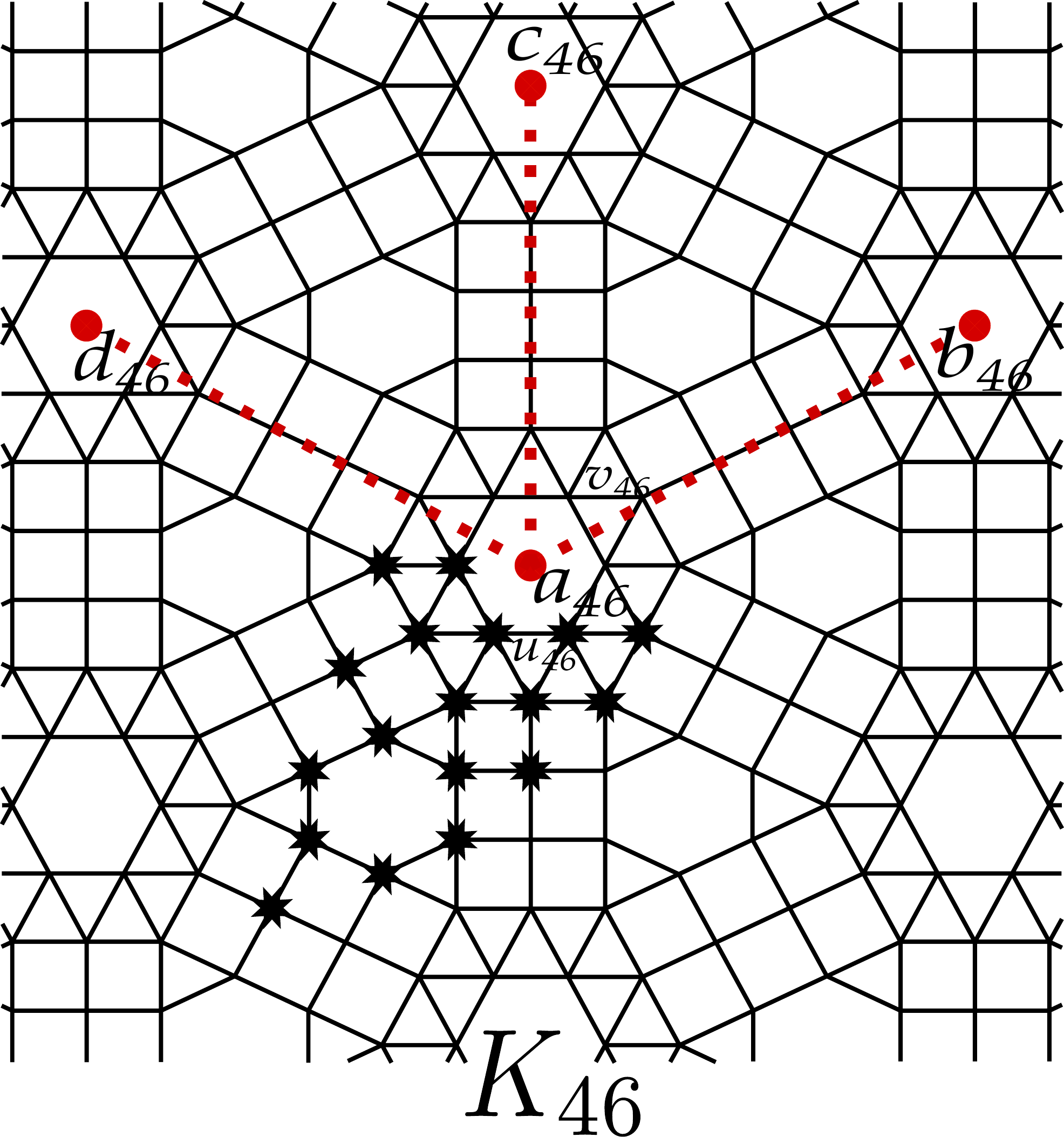}
    \includegraphics[height=2.9cm, width= 2.9cm]{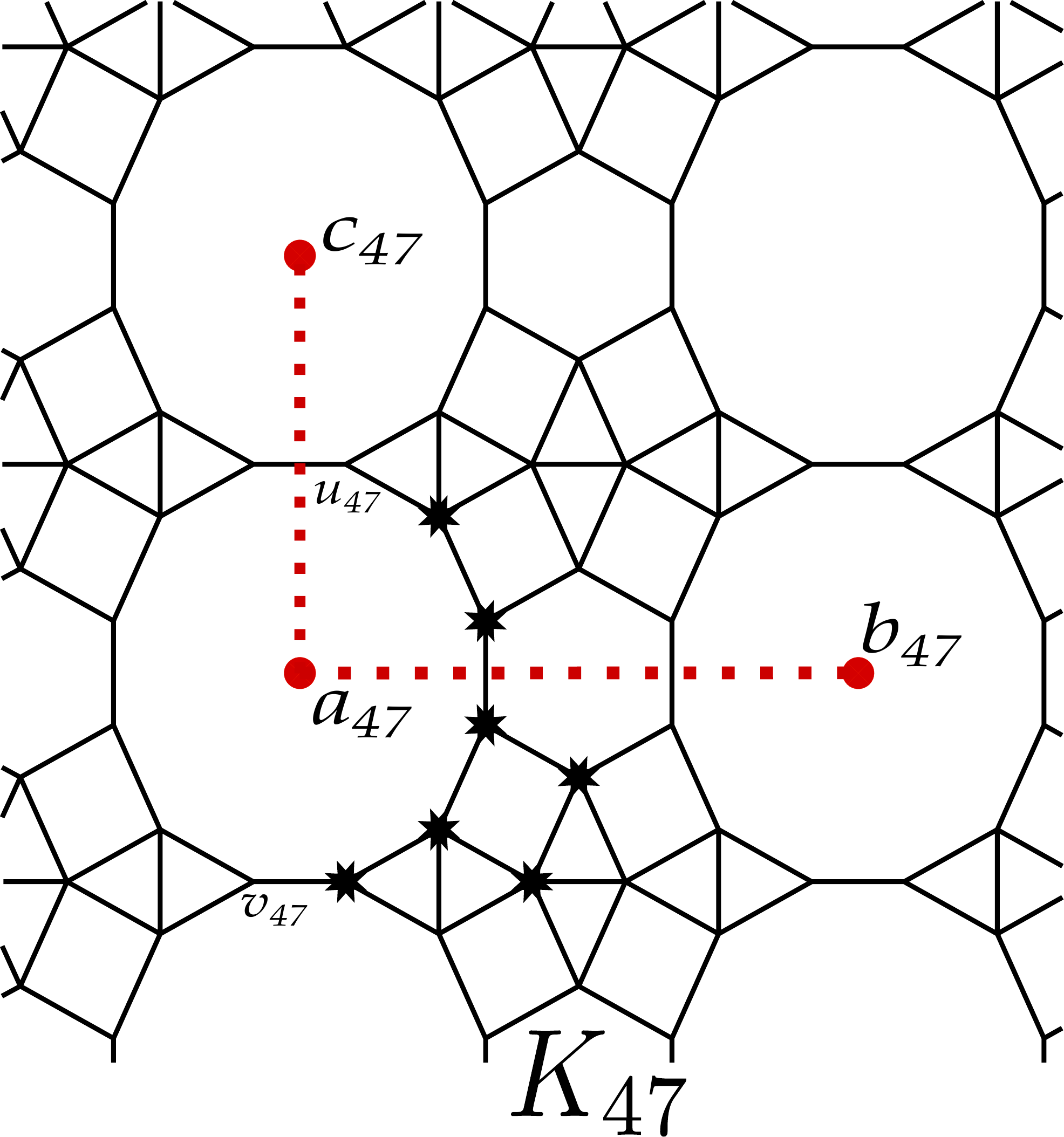}
    \includegraphics[height=2.9cm, width= 2.9cm]{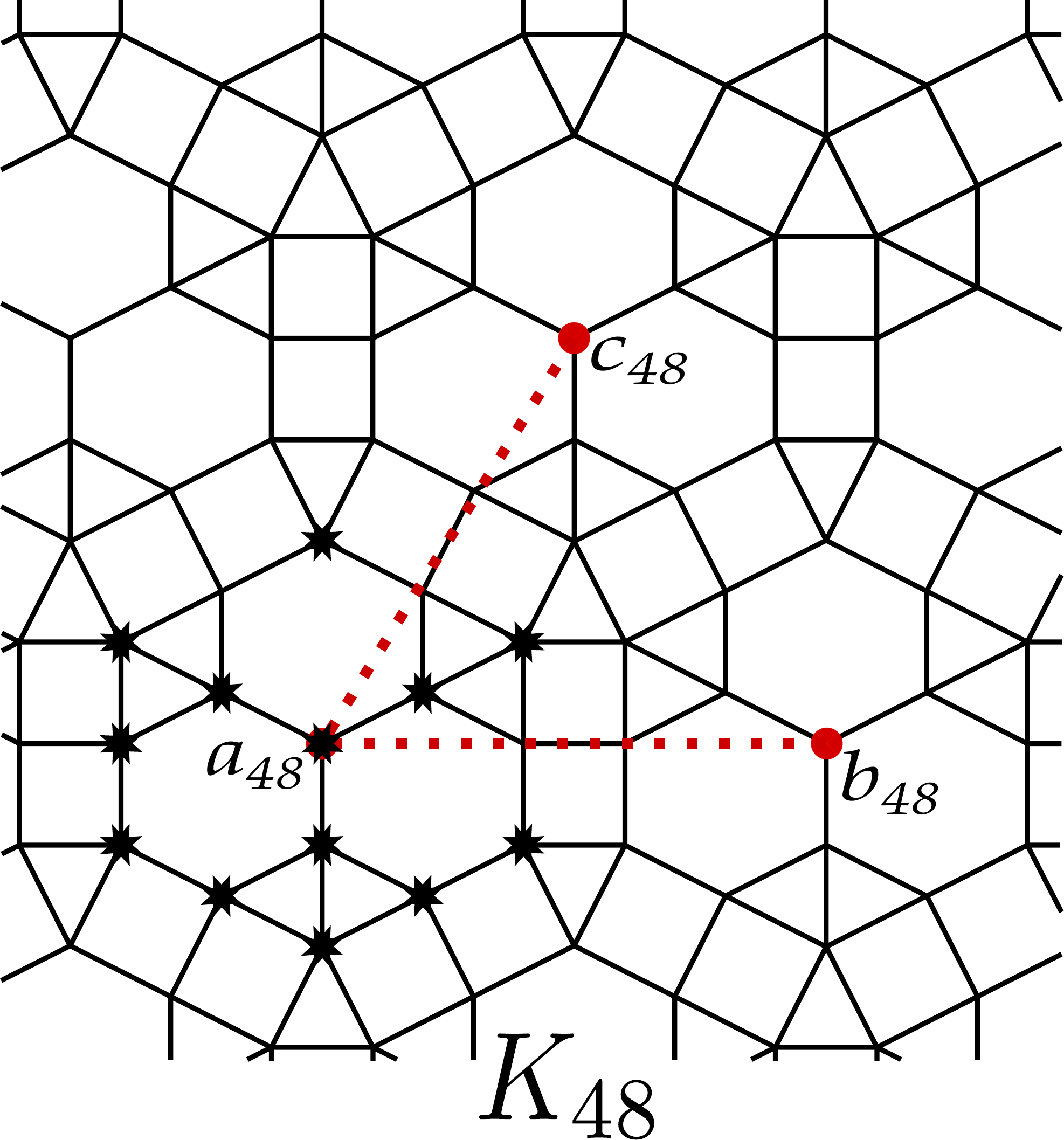}
    \includegraphics[height=2.9cm, width= 2.9cm]{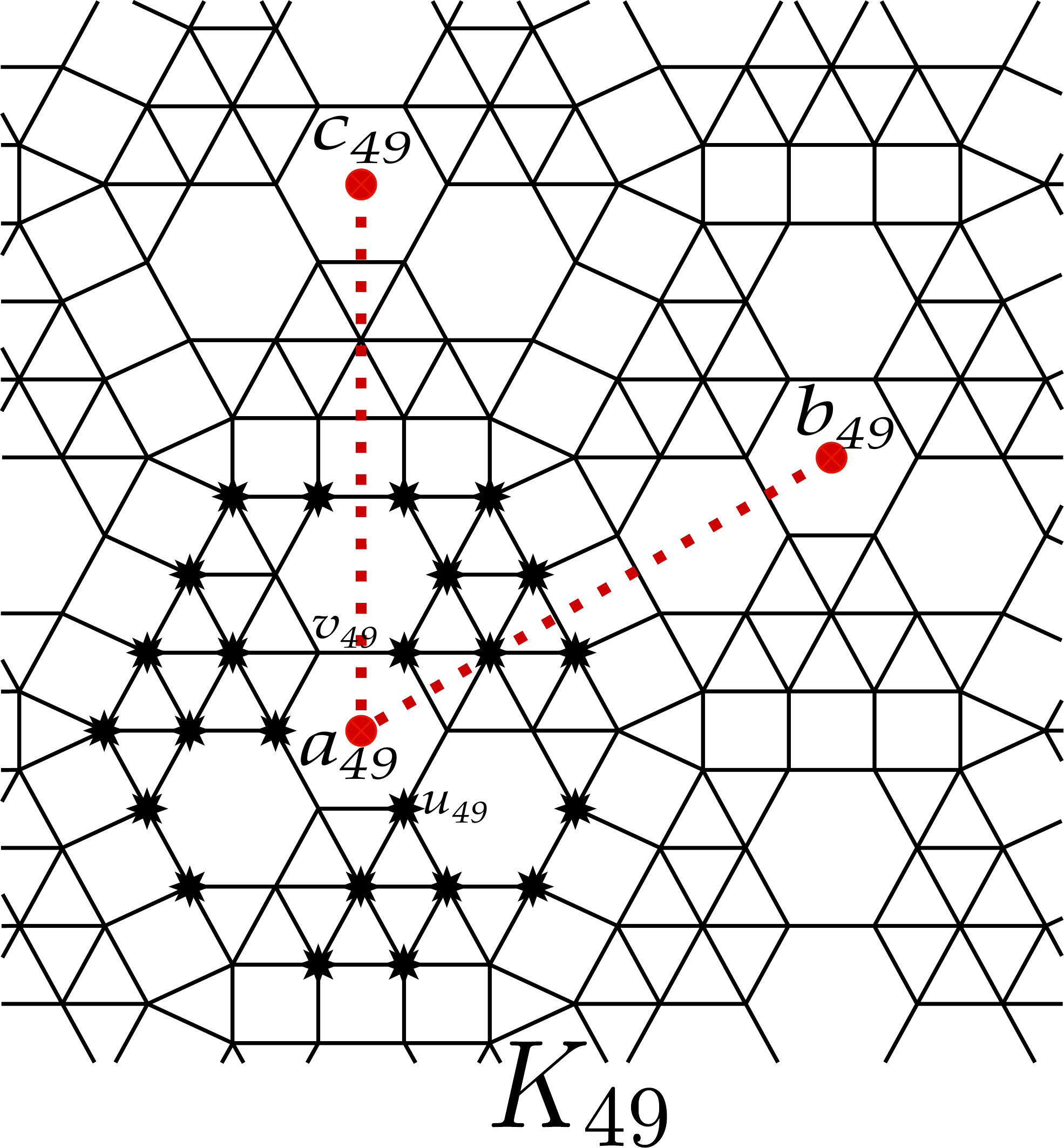}
    \includegraphics[height=2.9cm, width= 2.9cm]{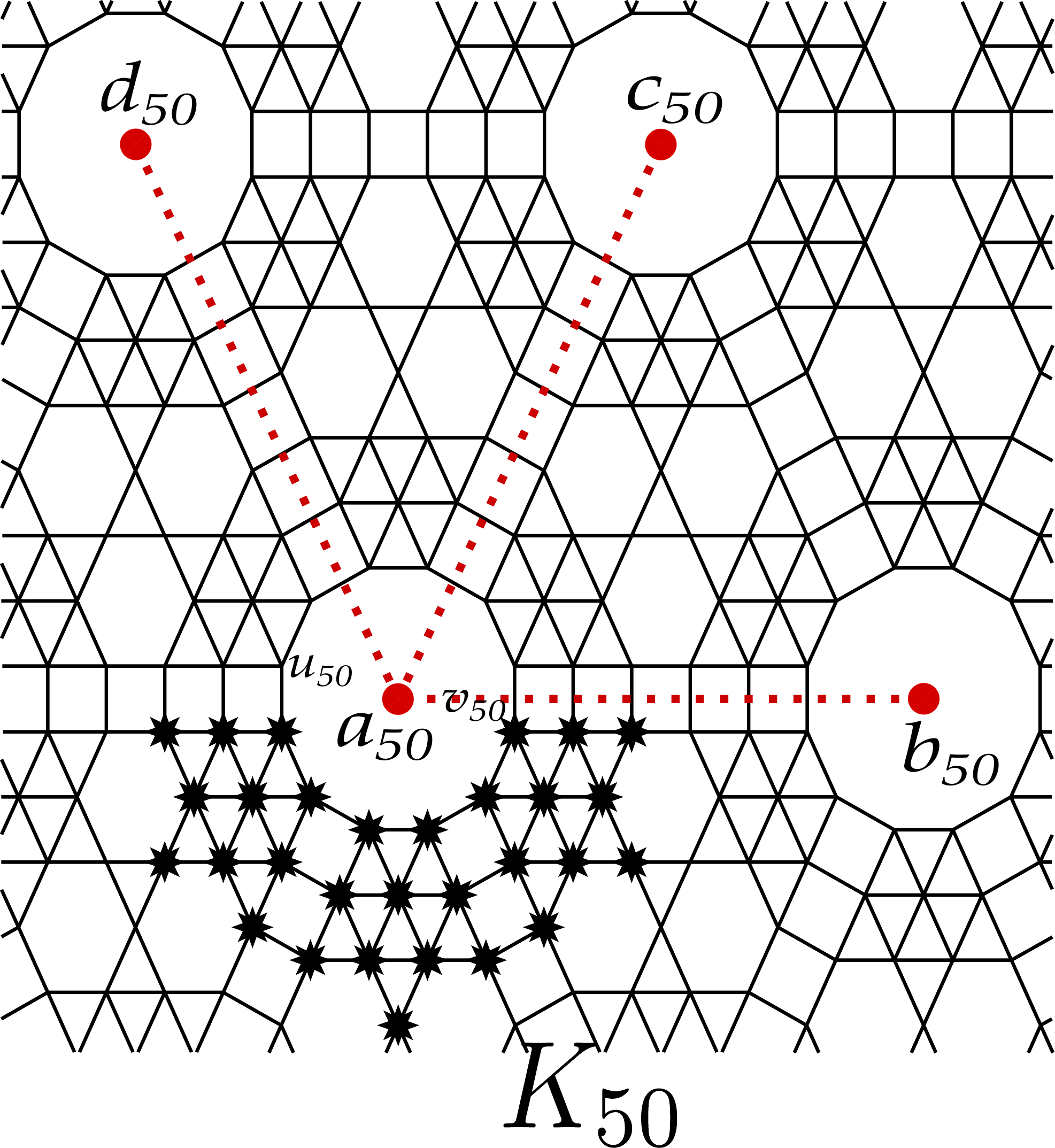}
    \end{figure}
     \vspace{-7mm}
\begin{figure}[H]
    \centering
    \includegraphics[height=2.9cm, width= 2.9cm]{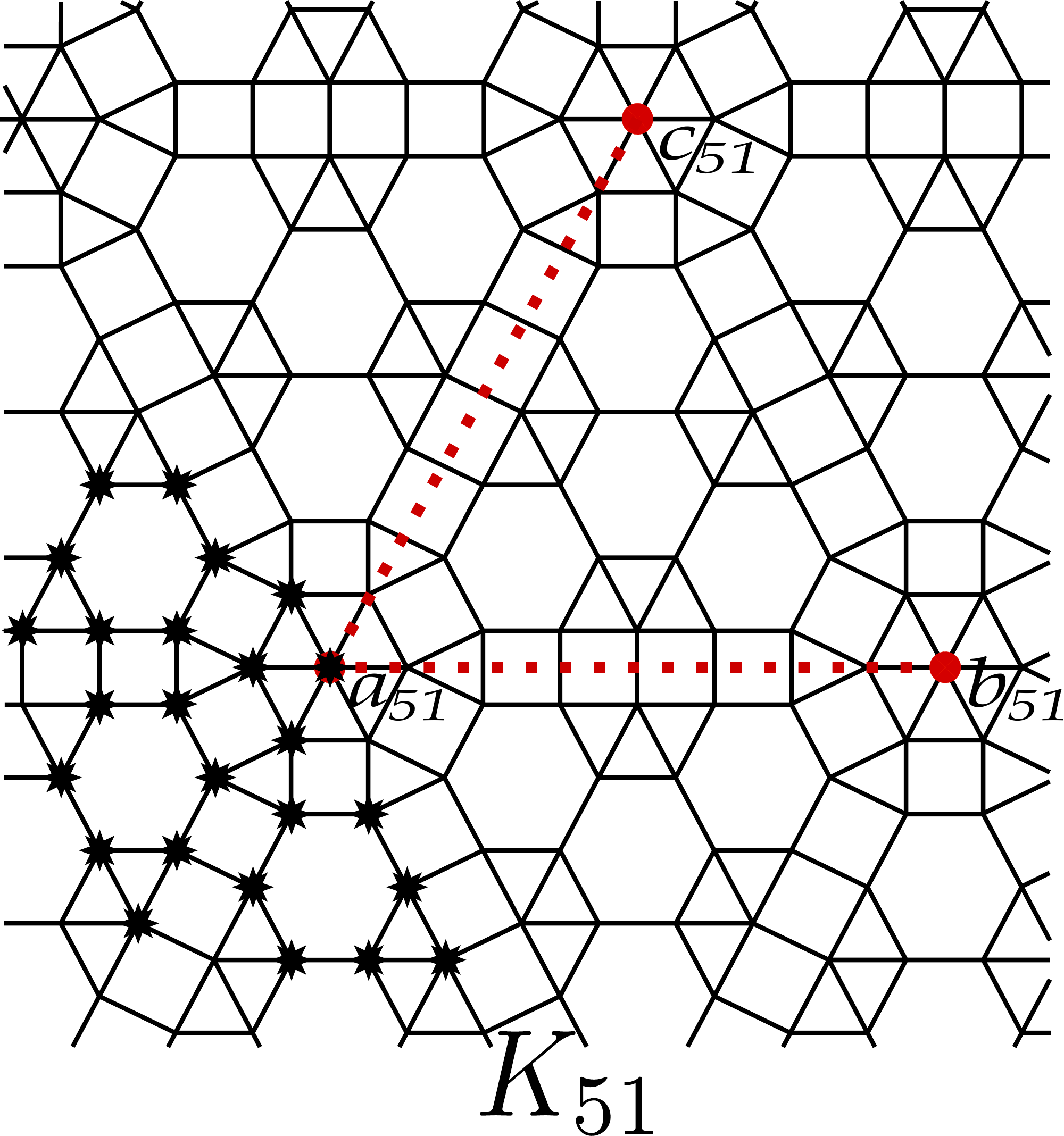}
    \includegraphics[height=2.9cm, width= 2.9cm]{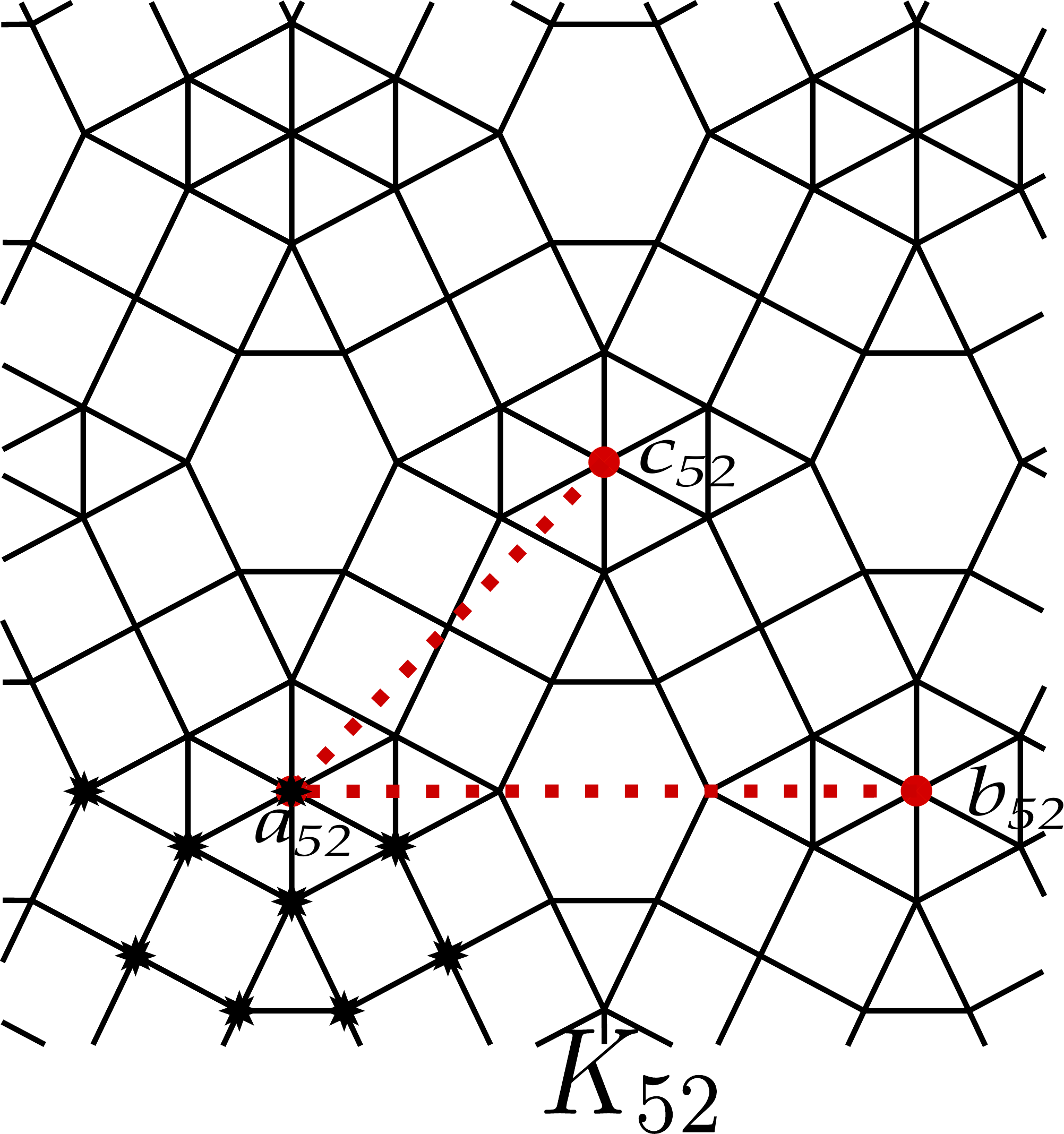}
    \includegraphics[height=2.9cm, width= 2.9cm]{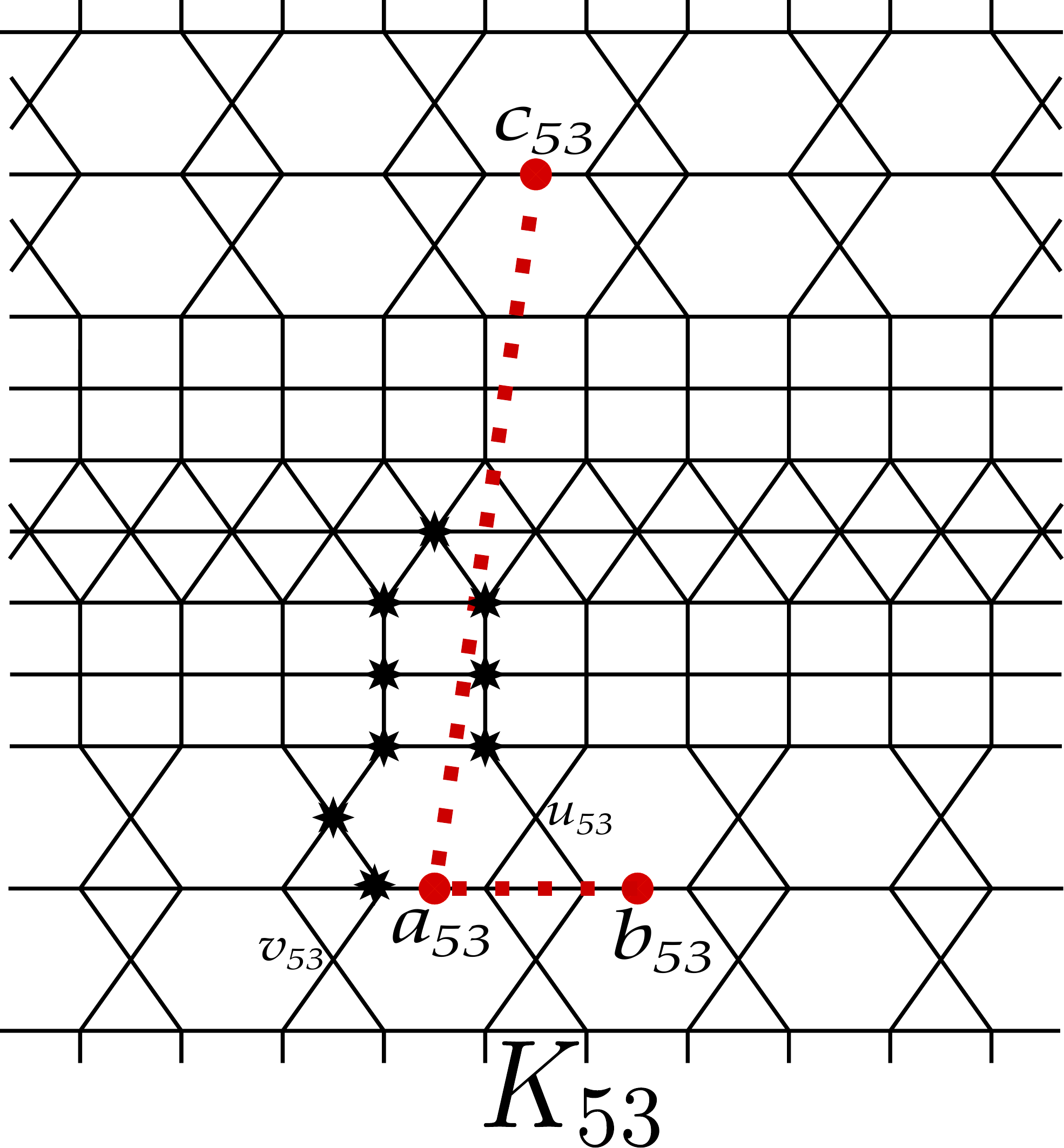}
    \includegraphics[height=2.9cm, width= 2.9cm]{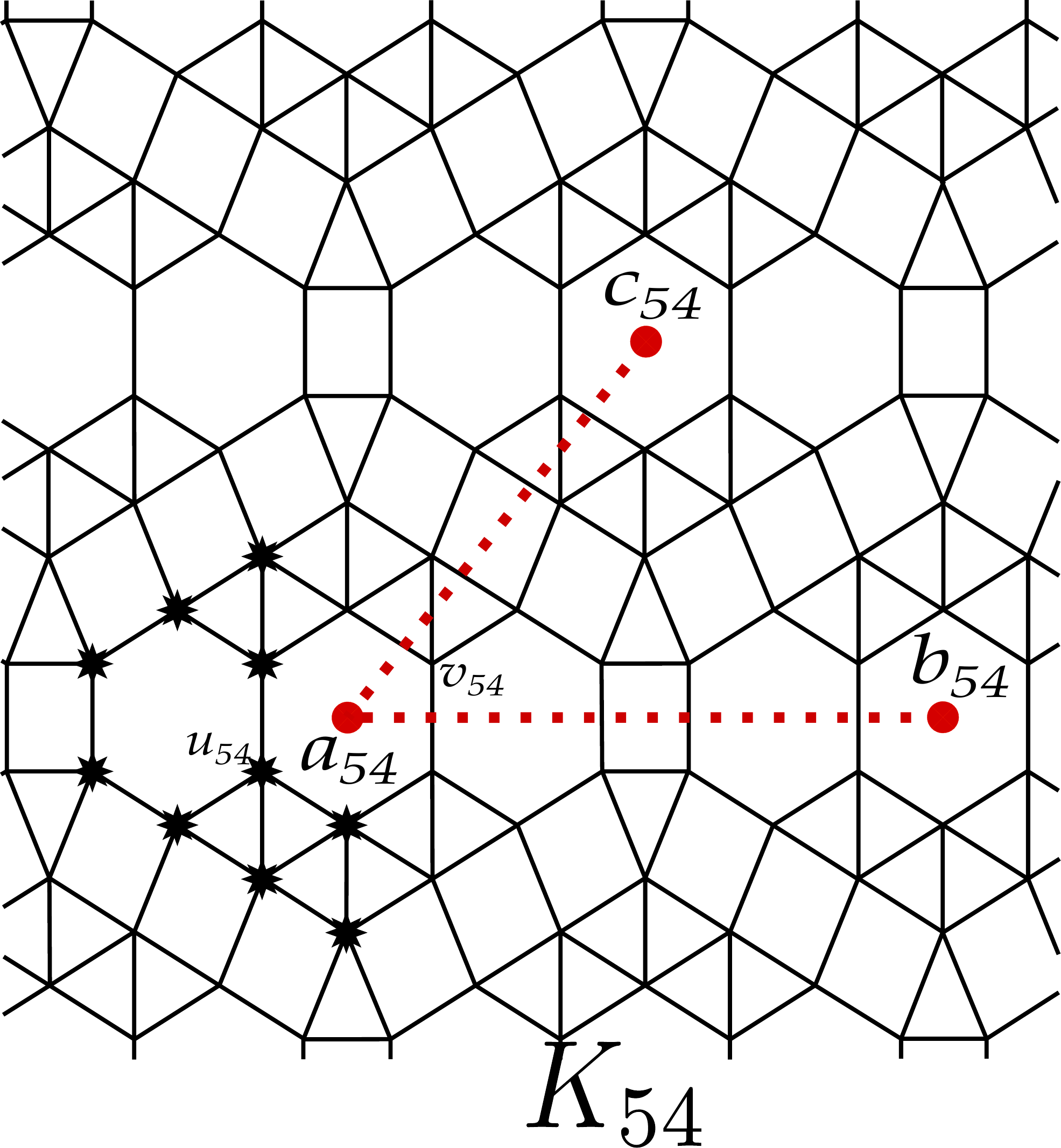}
    \includegraphics[height=2.9cm, width= 2.9cm]{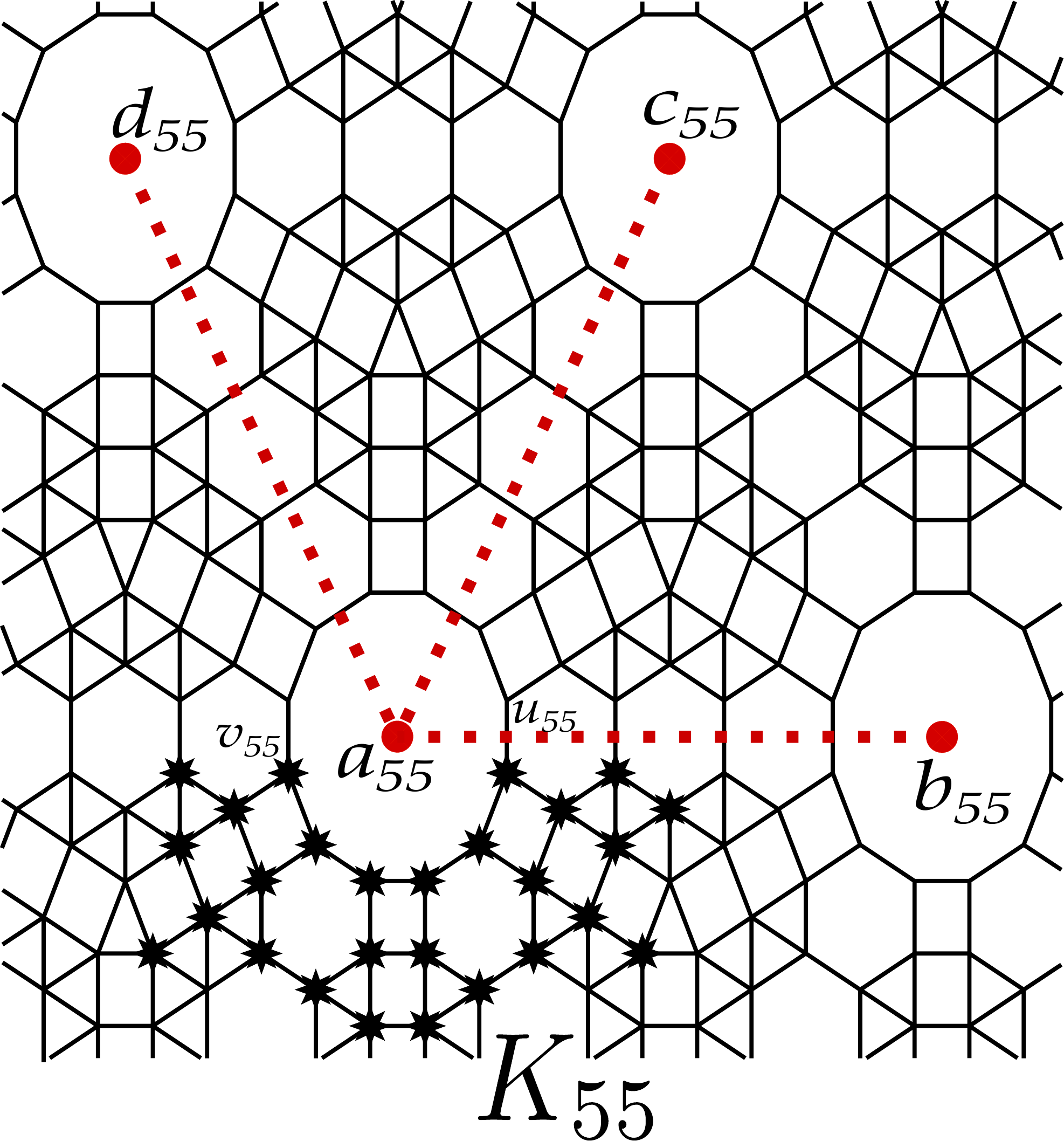}
    \end{figure}   
     \vspace{-7mm}
\begin{figure}[H]
    \centering
    \includegraphics[height=2.9cm, width= 2.9cm]{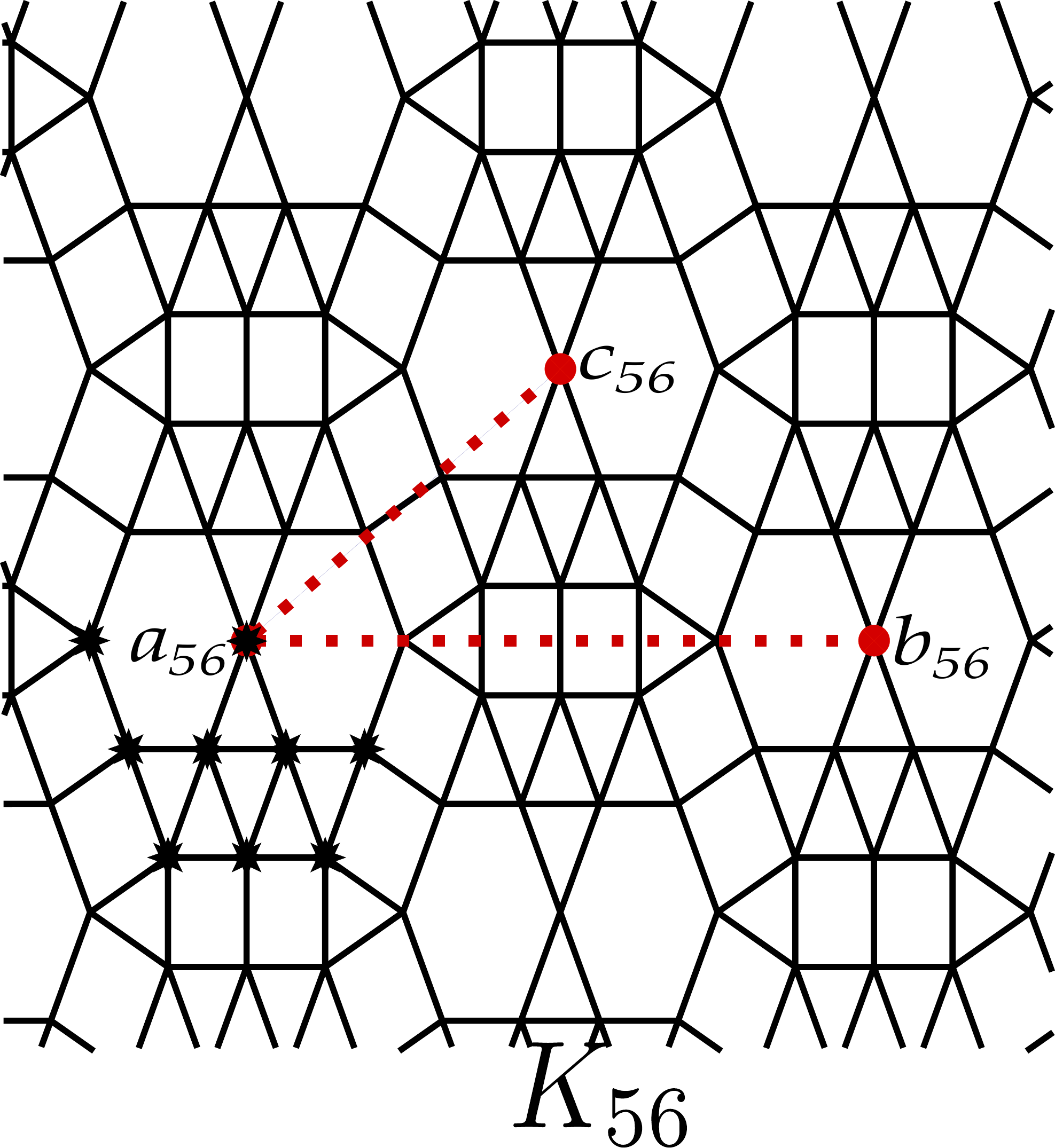}
    \includegraphics[height=2.9cm, width= 2.9cm]{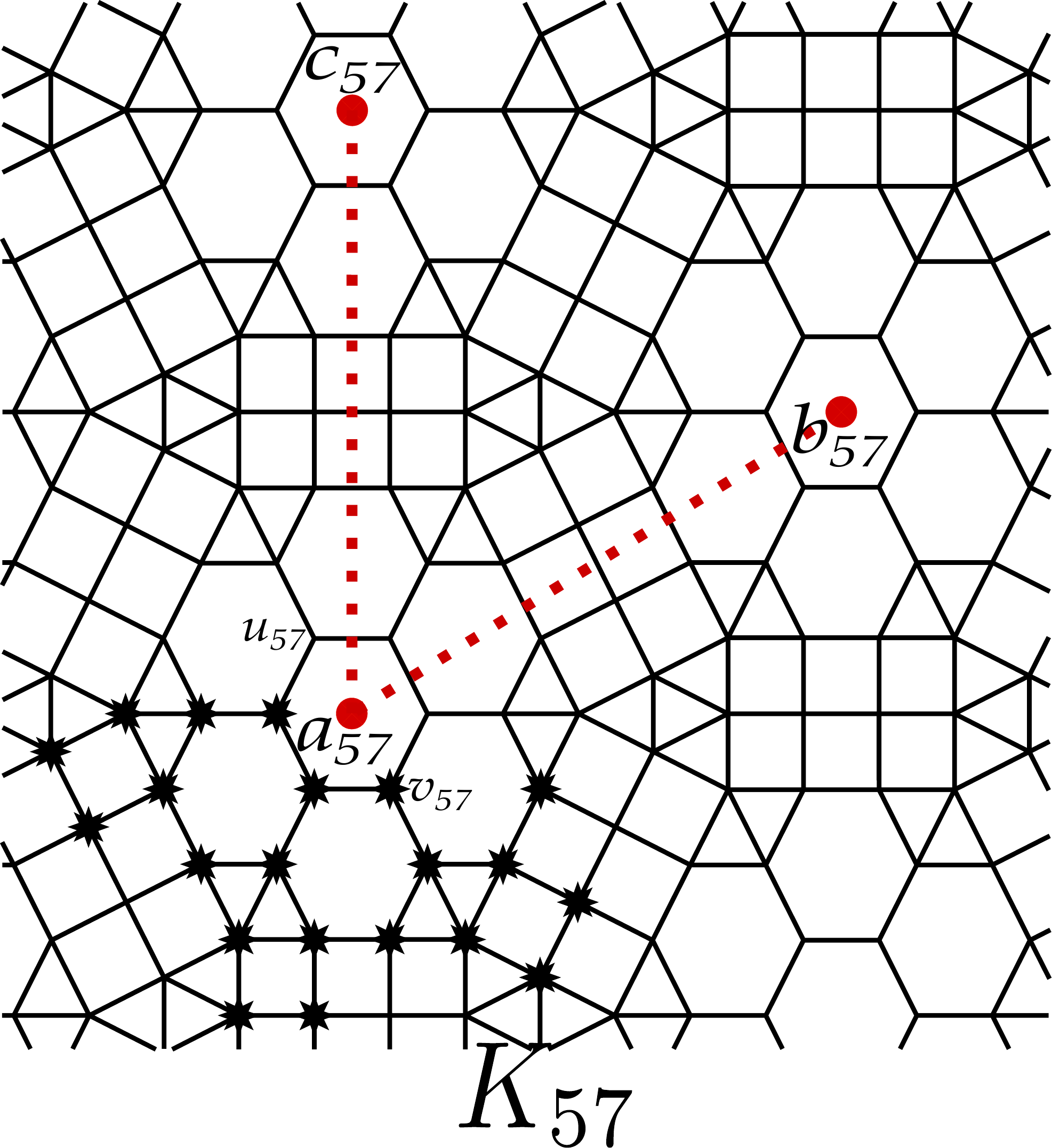}
    \includegraphics[height=2.9cm, width= 2.9cm]{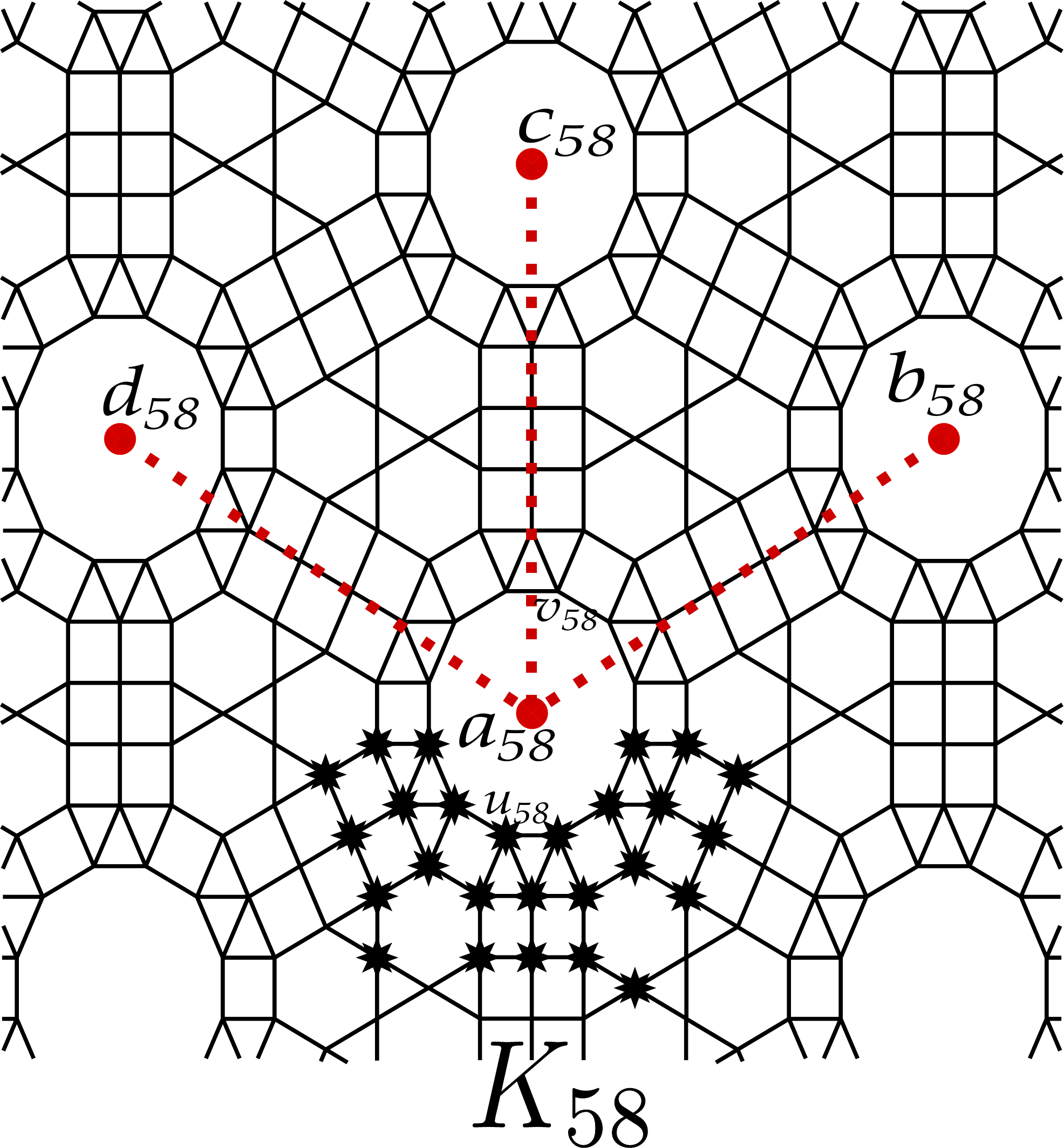}
    \includegraphics[height=2.9cm, width= 2.9cm]{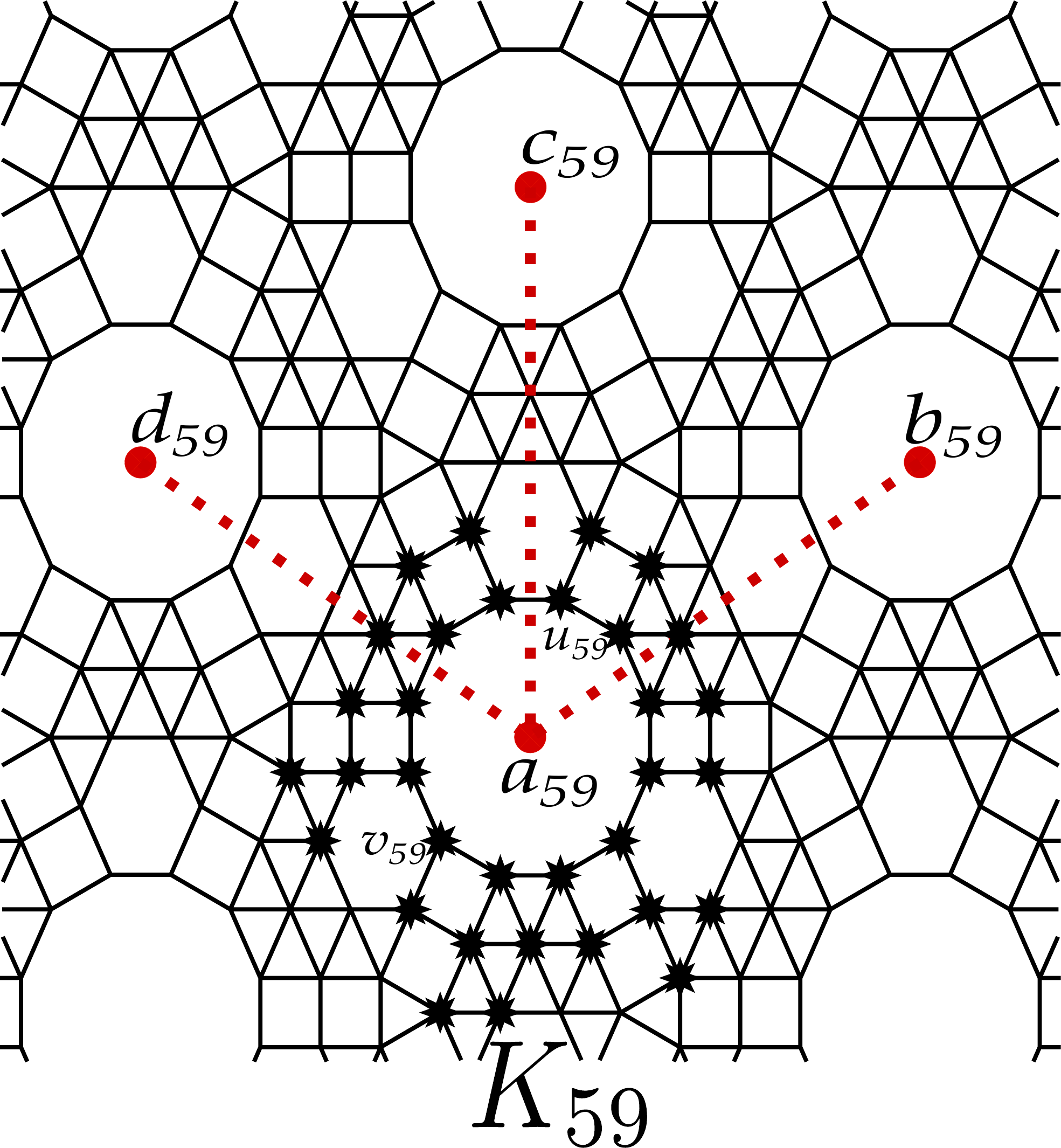}
    \includegraphics[height=2.9cm, width= 2.9cm]{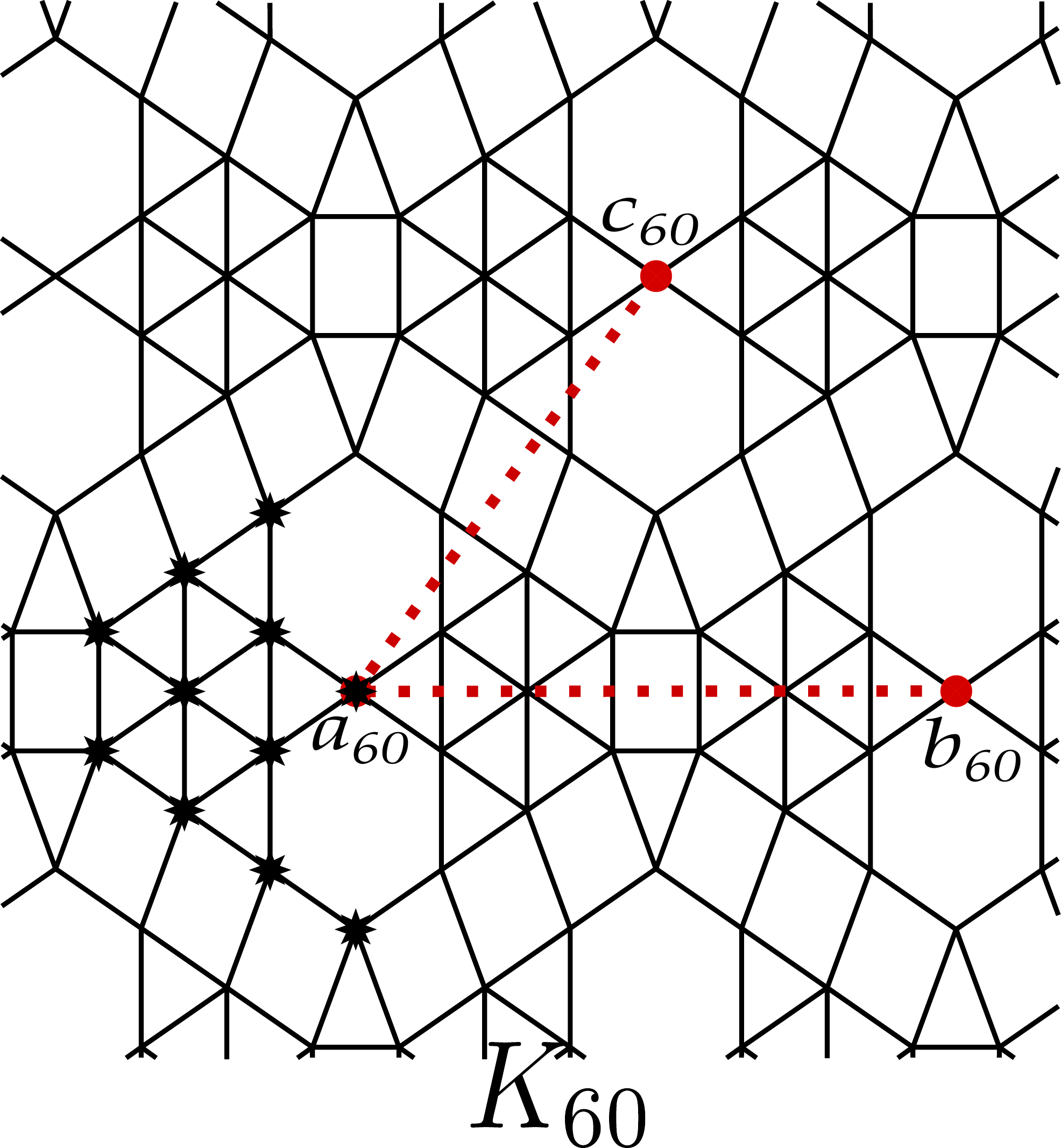}
    \end{figure}
     \vspace{-7mm}
\begin{figure}[H]
    \centering
    \includegraphics[height=2.9cm, width= 2.9cm]{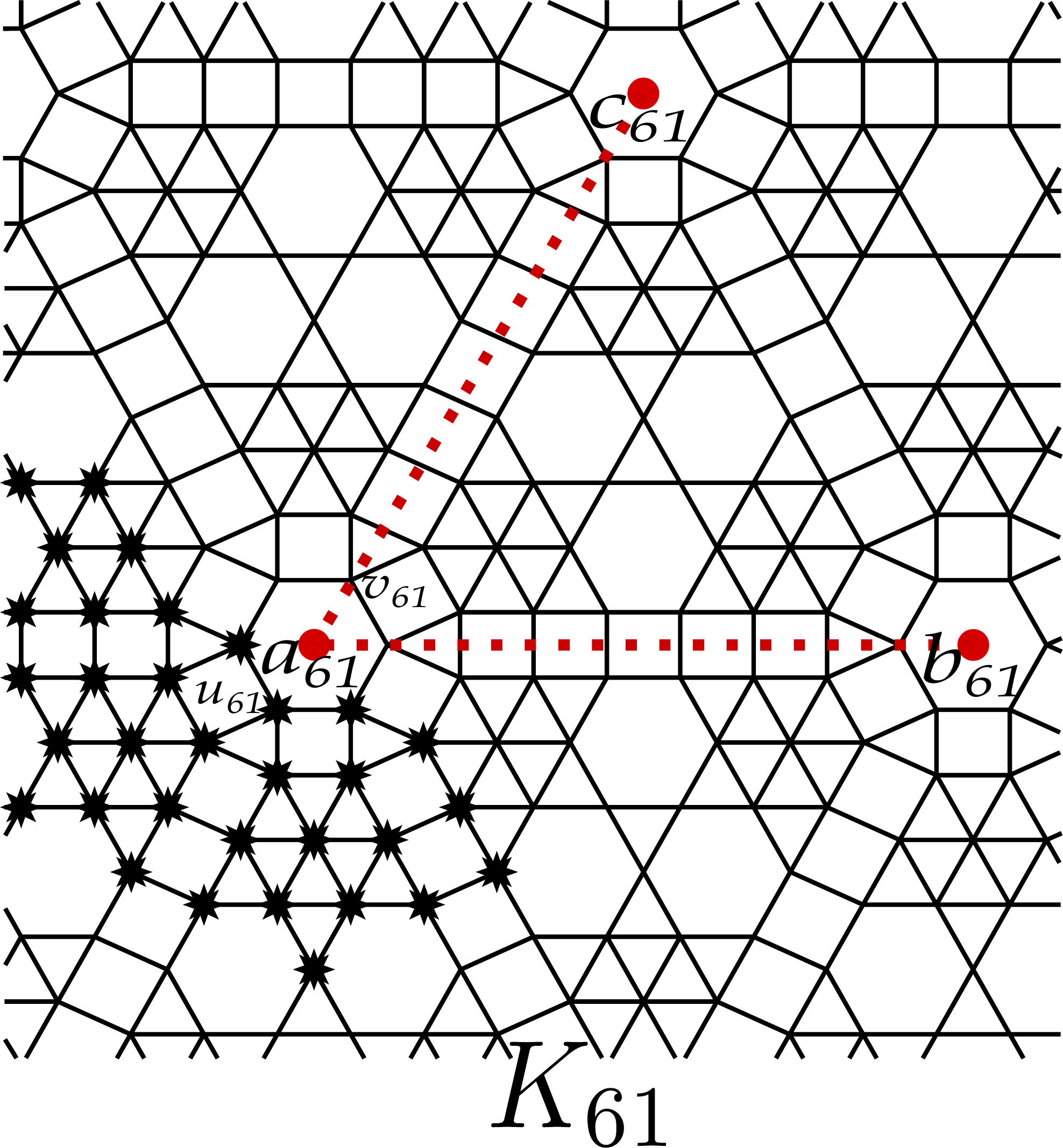}
    \includegraphics[height=2.9cm, width= 2.9cm]{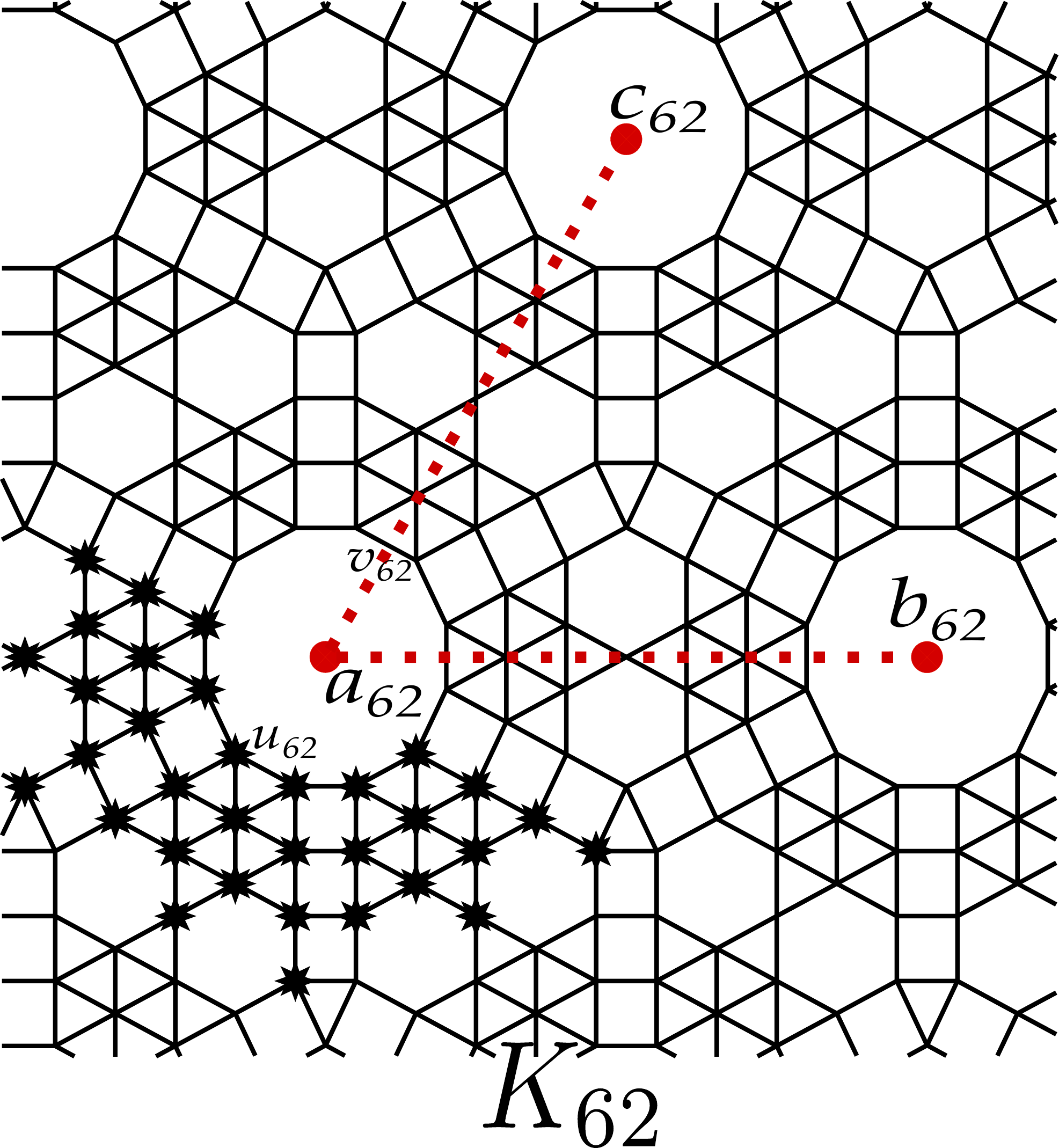}
    \includegraphics[height=2.9cm, width= 2.9cm]{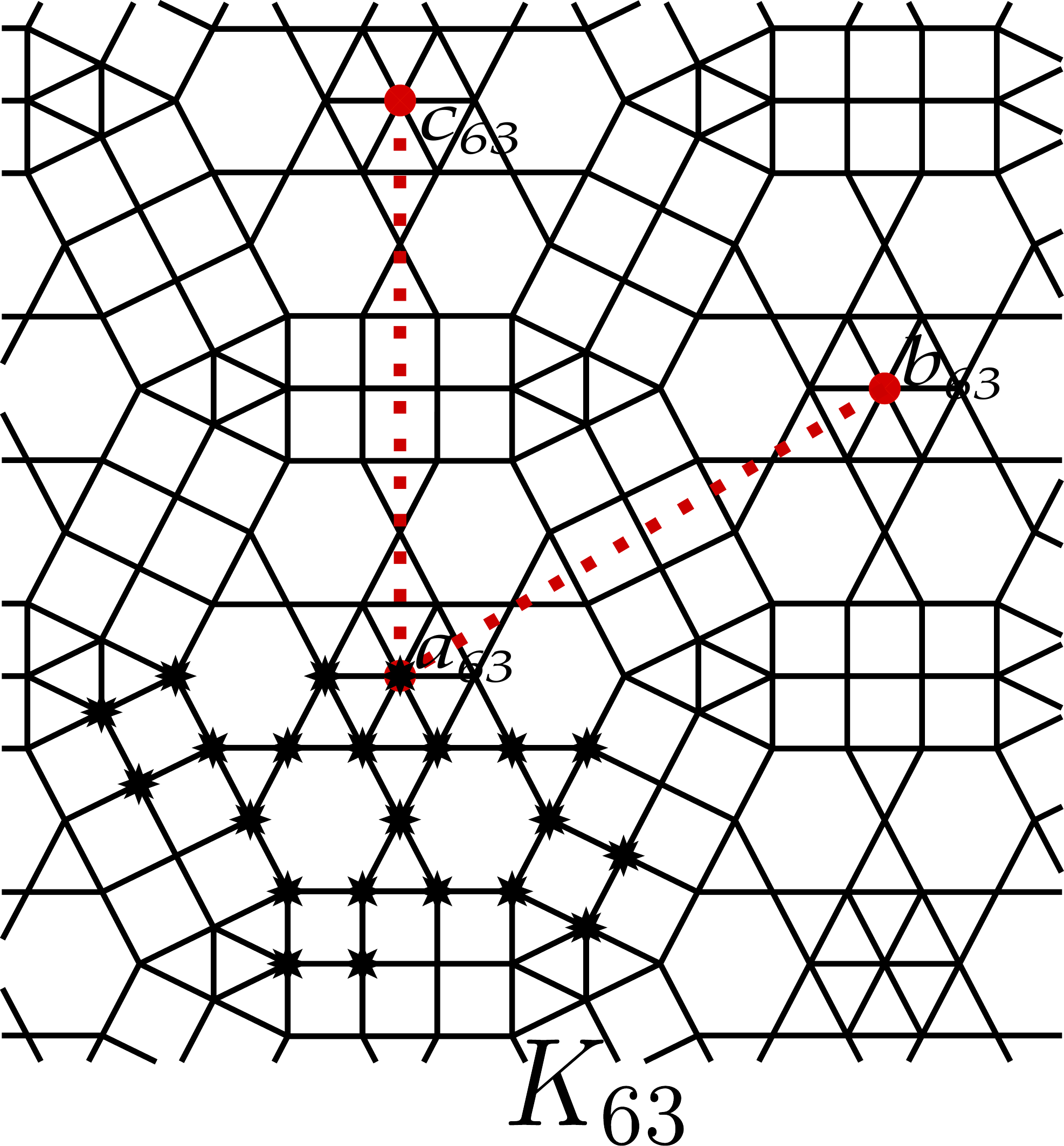}
    \includegraphics[height=2.9cm, width= 2.9cm]{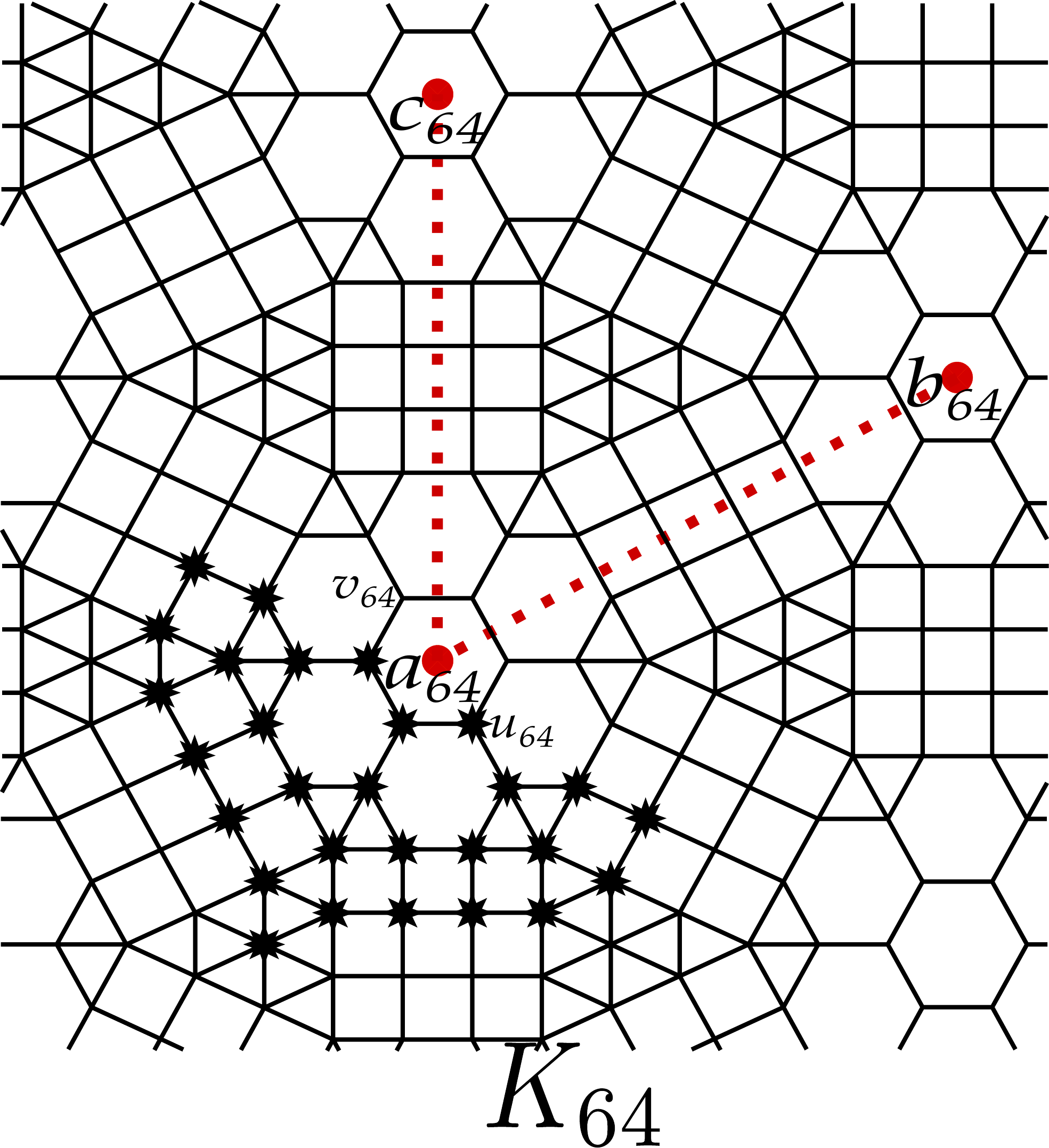}
    \includegraphics[height=2.9cm, width= 2.9cm]{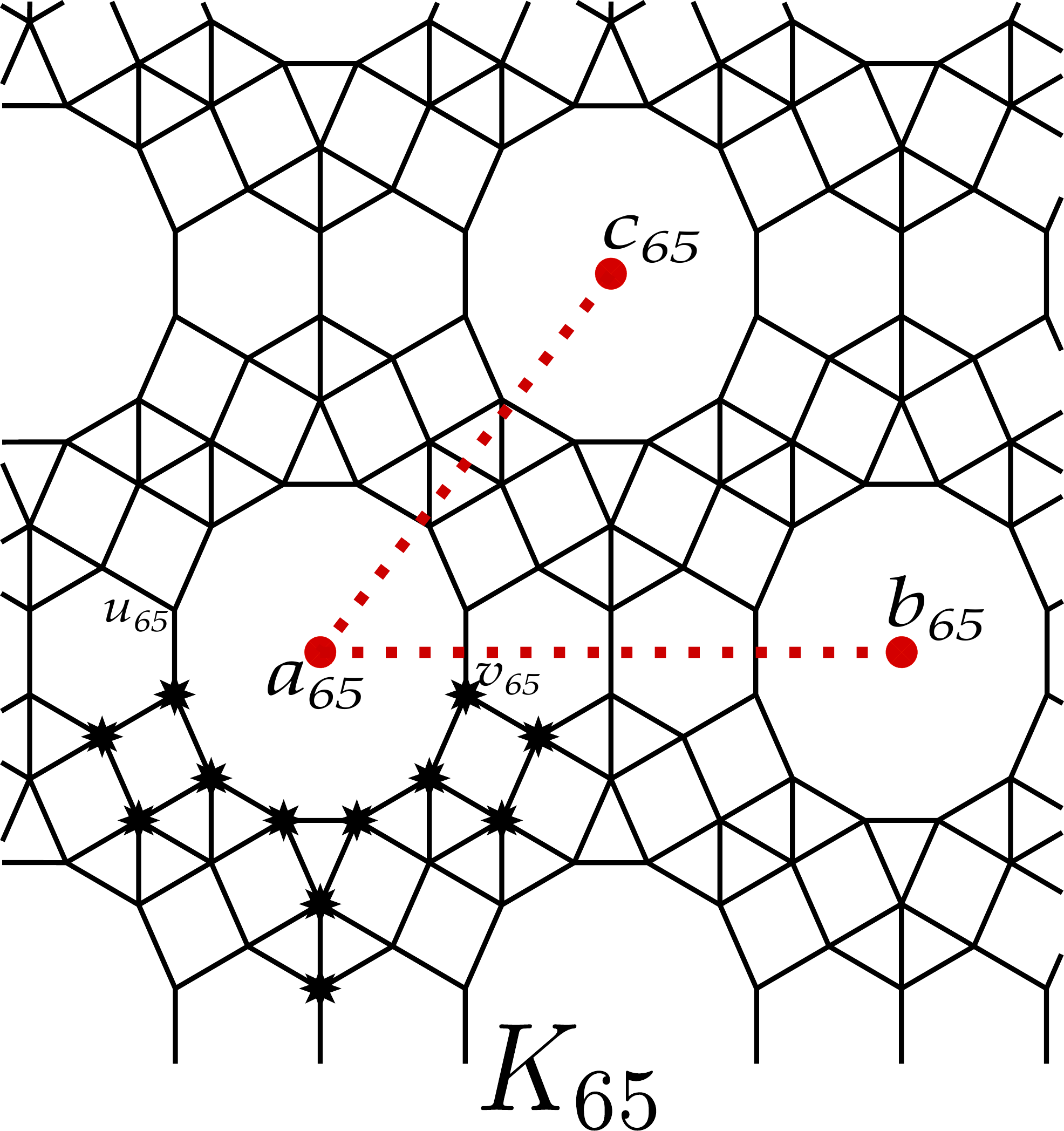}
    \end{figure}

\section{Proof} \label{proof-1}

Gr\"{u}nbaum and G. C. Shephard \cite{GS1977, GS1981} and Kr\"{o}tenheerdt \cite{Otto1977} have discussed the existence and uniqueness of the $k$-vertex-homogeneous lattice ($k \ge 4$) $K_i$, $i =1, 2, \dots, 65$ of the plane. Thus, we have the following. 

\begin{proposition}\label{prop1}
$K_{i}$ ($1 \le i \le 65$) (in Example \ref{exam:plane1}) are unique up to isomorphism. 
\end{proposition} 

\begin{proof}[Proof of Theorem \ref{theo1}]
Let $X_1$ be a map on the torus that is the quotient of the plane's $4$-uniform lattice $K_1$, where $K_1$ is as shown in Section \ref{fig:kuniform}. Let the vertices of $X_1$ form $m_1$ ${\rm Aut}(X_1)$-orbits and $V_{1} = V(K_1)$ be the vertex set of $K_1$. Let $H_{1}$ be the group of all the translations of $K_1$. So, $H_1 \leq {\rm Aut}(K_1)$.

Since $X_1$ is a map on the torus that is the quotient of the plane's $4$-uniform lattice $K_1$ (as by Proposition \ref{prop1}, $K_1$ is unique), so, we can assume that there exists a polyhedral covering map $\eta_{1} : K_1 \to X_1$, where $X_1 = K_1/\Gamma_{1}$  for some fixed element (vertex, edge or face) free subgroup $\Gamma_{1} \le {\rm Aut}(K_1)$. Hence, $\Gamma_{1}$
consists of translations and glide reflections. Since $X_1 =
K_1/\Gamma_{1}$ is orientable, $\Gamma_{1}$ does not contain any glide reflection. Thus, $\Gamma_{1} \leq H_{1}$.

 We take the middle point of the line segment joining vertices $B_{0}$ and $B_{1}$ as the origin $(0,0)$ of $K_1$. Let $\alpha_1 := B_2 - B_{0}$ and $\beta_1 := B_{2}- B_{0}$ in $K_1$. Then $$H_1 := \langle x\mapsto x+\alpha_1, x\mapsto x+\beta_1\rangle.$$ Under the action of $H_1$, vertices of $K_1$ form eleven orbits.
Consider the subgroup $G_1$ of ${\rm Aut}(K_1)$ generated by $H_1$ and the map (the half rotation) $x\mapsto -x$. So,
\begin{align*}
  G_1 & =\{ \alpha : x\mapsto \varepsilon x + m\alpha_1 + n\beta_1 \, : \, \varepsilon=\pm 1, m, n\in \ZZ\} \cong H_1\rtimes \mathbb{Z}_2.
\end{align*}
Clearly, under the action of $G_1$, vertices of $K_1$ form six orbits. The orbits are 
$
 O_1 :=\langle A_{0} \rangle, O_2 :=\langle B_0 \rangle, O_3 :=\langle C_{0} \rangle, O_4 :=\langle C_1 \rangle, O_5 :=\langle D_0 \rangle, O_6 :=\langle D_1 \rangle.
$

\begin{claim}\label{claim1}
If $S \leq H_1$ then $S \unlhd G_1$.
\end{claim} 

\smallskip

Let $g \in G_1$ and $s\in S$. Then $g(x) = \varepsilon x+ma+nb$ and $s(x) = x + pa+ qb$ for some $m, n, p, q \in \mathbb{Z}$ and $\varepsilon\in\{1, -1\}$.
Therefore, 
\begin{align*}
(g\circ s\circ g^{-1})(x) & = (g\circ s)(\varepsilon(x-ma-nb))\\
                          & = g(\varepsilon(x-ma-nb-rc)+pa+qb)\\                          
                          & =x-ma-nb+\varepsilon(pa+qb)+ma+nb\\
                          & = x+\varepsilon(pa+qb)\\
                          & =s s^{\varepsilon}(x).
\end{align*}
 Thus, $g\circ s\circ g^{-1} = s^{\varepsilon}\in S$. This completes the claim.

\smallskip

By Claim 1, $\Gamma_1$ is a normal subgroup of $G_1$. Therefore, $G_1/\Gamma_1$ acts on $X_1= K_1/\Gamma_1$.
Since 
$
 O_1 :=\langle A_{0} \rangle, O_2 :=\langle B_0 \rangle, O_3 :=\langle C_{0} \rangle, O_4 :=\langle C_1 \rangle, O_5 :=\langle D_0 \rangle, O_6 :=\langle D_1 \rangle
$
 are the $G_1$-orbits, it follows that $\eta_1(O_j)$ for $j=1, 2, \dots, 6$ are the $(G_1/\Gamma_1)$-orbits. Since the vertex set of $K_1$ is $\sqcup_{j=1}^{6}\eta_1(O_j)$ and $G_1/\Gamma_1 \leq {\rm Aut}(X_1)$, it follows that the number of ${\rm Aut}(X_1)$-orbits of vertices is $\leq 6$, and hence, $m_1 \le 6$. 
 
\medskip

Let $I_1= \{2,4,8,9,14,20,24,26,31,32,33\}$. For $i \in I_1$, let $X_{i} = K_{i}/\Gamma_{i}$ be a map on the torus, for some fixed element (vertex, edge or face) free subgroup $\Gamma_{i} \le {\rm Aut}(K_{i})$. $K_i, i \in I_1$, are given in Section \ref{fig:kuniform}. Let the vertices of $X_{i}$ form $m_{i}$ ${\rm Aut}(X_{i})$-orbits.
In $K_{2}$, we take the middle point of the line segment joining vertices $u_{2}$ and $v_{2}$ as the origin $(0,0)$ and $\alpha_{2} := b_{2} - a_{2}$, $\beta_{2} := c_{2} - a_{2}$ $\in \mathbb{R}^2$.
In $K_{4}$, we take the middle point of the line segment joining vertices $u_{4}$ and $v_{4}$ as the origin $(0,0)$ and $\alpha_{4} := b_{4} - a_{4}$, $\beta_{4} := c_{4} - a_{4}$ $\in \mathbb{R}^2$.
In $K_{8}$, we take the middle point of the line segment joining vertices $u_{8}$ and $v_{8}$ as the origin $(0,0)$ and $\alpha_{8} := b_{8} - a_{8}$, $\beta_{8} := c_{8} - a_{8}$ $\in \mathbb{R}^2$.
In $K_{9}$, we take the middle point of the line segment joining vertices $u_{9}$ and $v_{9}$ as the origin $(0,0)$ and $\alpha_{9} := b_{9} - a_{9}$, $\beta_{9} := c_{9} - a_{9}$ $\in \mathbb{R}^2$.
In $K_{14}$, we take the middle point of the line segment joining vertices $u_{14}$ and $v_{14}$ as the origin $(0,0)$ and $\alpha_{14} := b_{14} - a_{14}$, $\beta_{14} := c_{14} - a_{14}$ $\in \mathbb{R}^2$.
In $K_{20}$, we take the middle point of the line segment joining vertices $u_{20}$ and $v_{20}$ as the origin $(0,0)$ and $\alpha_{20} := b_{20} - a_{20}$, $\beta_{20} := c_{20} - a_{20}$ $\in \mathbb{R}^2$.
In $K_{24}$, we take the middle point of the line segment joining vertices $u_{24}$ and $v_{24}$ as the origin $(0,0)$ and $\alpha_{24} := b_{24} - a_{24}$, $\beta_{24} := c_{24} - a_{24}$ $\in \mathbb{R}^2$.
In $K_{26}$, we take the middle point of the line segment joining vertices $u_{26}$ and $v_{26}$ as the origin $(0,0)$ and $\alpha_{26} := b_{26} - a_{26}$, $\beta_{26} := c_{26} - a_{26}$ $\in \mathbb{R}^2$.
In $K_{31}$, we take the middle point of the line segment joining vertices $u_{31}$ and $v_{31}$ as the origin $(0,0)$ and $\alpha_{31} := b_{31} - a_{31}$, $\beta_{31} := c_{31} - a_{31}$ $\in \mathbb{R}^2$.
In $K_{32}$, we take the middle point of the line segment joining vertices $u_{32}$ and $v_{32}$ as the origin $(0,0)$ and $\alpha_{32} := b_{32} - a_{32}$, $\beta_{32} := c_{32} - a_{32}$ $\in \mathbb{R}^2$.
In $K_{33}$, we take the middle point of the line segment joining vertices $u_{33}$ and $v_{33}$ as the origin $(0,0)$ and $\alpha_{33} := b_{33} - a_{33}$, $\beta_{33} := c_{33} - a_{33}$ $\in \mathbb{R}^2$.
Similarly as done for $K_1$, we define $H_i$ and $G_i$, for $i \in I_1$. By the same arguments as above and in Claim 1, for $i \in I_1$, $\Gamma_i \unlhd G_i$ and the number of $G_i/\Gamma_i$-orbits of vertices of $X_i$ is six. Hence, $m_{i} \le 6$, for $i \in I_1$.
In $K_i$, the six vertex orbits under $G_i$ are marked by the symbol \textbf{$\ast$} respectively for $i \in I_1$.
This completes the part {\rm (1)}.

\medskip

Let $I_2= \{6,7\}$. For $i \in I_2$, let $X_{i} = K_{i}/\Gamma_{i}$ be a map on the torus, for some fixed element (vertex, edge or face) free subgroup $\Gamma_{i} \le {\rm Aut}(K_{i})$. $K_i, i \in I_2$, are given in Section \ref{fig:kuniform}. Let the vertices of $X_{i}$ form $m_{i}$ ${\rm Aut}(X_{i})$-orbits.
In $K_6$, we take the point $a_{0}$ as the origin $(0,0)$ and $\alpha_6 := b_6 - a_{6}$, $\beta_6 := c_{6} - a_{6}$, $\gamma_6 := d_{6} - a_{6}$ $\in \mathbb{R}^2$.
In $K_{7}$, we take the middle point of the line segment joining vertices $u_{7}$ and $v_{7}$ as the origin $(0,0)$ and $\alpha_{7} := b_7 - a_{7}$, $\beta_{7} := c_{7} - a_{7}$ $\in \mathbb{R}^2$.
For $K_6$, we define $H_{6} := \langle x\mapsto x+\alpha_{6}, x\mapsto x+\beta_{6}, x\mapsto x+\gamma_{6}\rangle$ and 
$ G_{6} :=\{ \alpha : x\mapsto \varepsilon x + m_{6}\alpha_{6} + n_{6}\beta_{6} +r_{6}\gamma_{6}   \, : \, \varepsilon=\pm 1, m_{6}, n_{6}, r_{6} \in \ZZ\}$. For $K_7$, we define $H_7$ and $G_7$ similarly as done for $K_1$ . By the same arguments as above and in Claim 1, for $i \in I_2$, $\Gamma_i \unlhd G_i$ and the number of $G_i/\Gamma_i$-orbits of vertices of $X_i$ is four. Hence, $m_{i} = 4$, for $i \in I_2$.
In $K_6$ and $K_7$, the four vertex orbits under $G_6$ and $G_7$ are marked by the symbol $\ast$ respectively.
This completes the part {\rm (2)}.

\medskip

Let $I_3= \{3,30\}$. For $i \in I_3$, let $X_{i} = K_{i}/\Gamma_{i}$ be a map on the torus, for some fixed element (vertex, edge or face) free subgroup $\Gamma_{i} \le {\rm Aut}(K_{i})$. $K_i, i \in I_3$, are given in Section \ref{fig:kuniform}. Let the vertices of $X_{i}$ form $m_{i}$ ${\rm Aut}(X_{i})$-orbits. 
In $K_{3}$, we take the middle point of the line segment joining vertices $u_{3}$ and $v_{3}$ as the origin $(0,0)$ and $\alpha_{3} := b_{3} - a_{3}$, $\beta_{3} := c_{3} - a_{3}$ $\in \mathbb{R}^2$.
In $K_{30}$, we take the middle point of the line segment joining vertices $u_{30}$ and $v_{30}$ as the origin $(0,0)$ and $\alpha_{30} := b_{30} - a_{30}$, $\beta_{30} := c_{30} - a_{30}$ $\in \mathbb{R}^2$
Similarly as done for $K_1$, we define $H_i$ and $G_i$, for $i \in I_3$. By the same arguments as above and in Claim 1, for $i \in I_3$, $\Gamma_i \unlhd G_i$ and the number of $G_i/\Gamma_i$-orbits of vertices of $X_i$ is five. Hence, $m_{i} \le 5$, for $i \in I_3$.
In $K_i$, the five vertex orbits under $G_i$ are marked by the symbol \textbf{$\ast$} respectively for $i \in I_3$.
This completes the part {\rm (3)}.

\medskip

Let $I_4 = \{5,21,25,29,34,40,43,47\}$. For $i \in I_4$, let $X_{i} = K_{i}/\Gamma_{i}$ be a map on the torus, for some fixed element (vertex, edge or face) free subgroup $\Gamma_{i} \le {\rm Aut}(K_{i})$. $K_i, i \in I_4$, are given in Section \ref{fig:kuniform}. Let the vertices of $X_{i}$ form $m_{i}$ ${\rm Aut}(X_{i})$-orbits.
In $K_{5}$, we take the middle point of the line segment joining vertices $u_{5}$ and $v_{5}$ as the origin $(0,0)$ and $\alpha_{5} := b_{5} - a_{5}$, $\beta_{5} := c_{5} - a_{5}$ $\in \mathbb{R}^2$.
In $K_{21}$, we take the middle point of the line segment joining vertices $u_{21}$ and $v_{21}$ as the origin $(0,0)$ and $\alpha_{21} := b_{21} - a_{21}$, $\beta_{21} := c_{21} - a_{21}$ $\in \mathbb{R}^2$.
In $K_{25}$, we take the middle point of the line segment joining vertices $u_{25}$ and $v_{25}$ as the origin $(0,0)$ and $\alpha_{25} := b_{25} - a_{25}$, $\beta_{25} := c_{25} - a_{25}$ $\in \mathbb{R}^2$.
In $K_{29}$, we take the middle point of the line segment joining vertices $u_{29}$ and $v_{29}$ as the origin $(0,0)$ and $\alpha_{29} := b_{29} - a_{29}$, $\beta_{29} := c_{29} - a_{29}$ $\in \mathbb{R}^2$.
In $K_{34}$, we take the vertex $a_{34}$ as the origin $(0,0)$ and $\alpha_{34} := b_{34} - a_{34}$, $\beta_{34} := c_{34} - a_{34}$ $\in \mathbb{R}^2$.
In $K_{40}$, we take the middle point of the line segment joining vertices $u_{40}$ and $v_{40}$ as the origin $(0,0)$ and $\alpha_{40} := b_{40} - a_{40}$, $\beta_{40} := c_{40} - a_{40}$ $\in \mathbb{R}^2$.
In $K_{43}$, we take the middle point of the line segment joining vertices $u_{43}$ and $v_{40}$ as the origin $(0,0)$ and $\alpha_{43} := b_{43} - a_{43}$, $\beta_{43} := c_{43} - a_{43}$ $\in \mathbb{R}^2$.
In $K_{47}$, we take the middle point of the line segment joining vertices $u_{47}$ and $v_{47}$ as the origin $(0,0)$ and $\alpha_{47} := b_{47} - a_{47}$, $\beta_{47} := c_{47} - a_{47}$ $\in \mathbb{R}^2$.
For all $i \in I_4$, we define $H_i$ and $G_i$ similarly as done for $K_1$. By the same arguments as above and in Claim 1, for $i \in I_4$, $\Gamma_i \unlhd G_i$ and the number of $G_i/\Gamma_i$-orbits of vertices of $X_i$ is seven. Hence, $m_{i} \le 7$, for $i \in I_4$.
In $K_i$, the seven vertex orbits under $G_i$ are marked by the symbol \textbf{$\ast$} respectively for $i \in I_4$.
This completes the part {\rm (4)}.

\medskip

Let $I_5 = \{22,35,39,41,42,44\}$. For $i \in I_5$, let $X_{i} = K_{i}/\Gamma_{i}$ be a  map on the torus, for some fixed element (vertex, edge or face) free subgroup $\Gamma_{i} \le {\rm Aut}(K_{i})$. $K_i, i \in I_5$, are given in Section \ref{fig:kuniform}. Let the vertices of $X_{i}$ form $m_{i}$ ${\rm Aut}(X_{i})$-orbits.
In $K_{22}$, we take the middle point of the line segment joining vertices $u_{22}$ and $v_{22}$ as the origin $(0,0)$ and $\alpha_{22} := b_{22} - a_{22}$, $\beta_{22} := c_{22} - a_{22}$ $\in \mathbb{R}^2$.
In $K_{35}$, we take the vertex $a_{35}$ as the origin $(0,0)$ and $\alpha_{35} := b_{35} - a_{35}$, $\beta_{35} := c_{35} - a_{35}$ $\in \mathbb{R}^2$.
In $K_{39}$, we take the middle point of the line segment joining vertices $u_{39}$ and $v_{39}$ as the origin $(0,0)$ and $\alpha_{39} := b_{39} - a_{39}$, $\beta_{39} := c_{39} - a_{39}$ $\in \mathbb{R}^2$.
In $K_{41}$, we take the middle point of the line segment joining vertices $u_{41}$ and $v_{41}$ as the origin $(0,0)$ and $\alpha_{41} := b_{41} - a_{41}$, $\beta_{41} := c_{41} - a_{41}$ $\in \mathbb{R}^2$.
In $K_{42}$, we take the middle point of the line segment joining vertices $u_{42}$ and $v_{42}$ as the origin $(0,0)$ and $\alpha_{42} := b_{42} - a_{42}$, $\beta_{42} := c_{42} - a_{42}$ $\in \mathbb{R}^2$.
In $K_{44}$, we take the vertex $a_{44}$ as the origin $(0,0)$ and $\alpha_{44} := b_{44} - a_{44}$, $\beta_{44} := c_{44} - a_{44}$ $\in \mathbb{R}^2$.
For all $i \in I_5$, we define $H_i$ and $G_i$ similarly as done for $K_1$. By the same arguments as above and in Claim 1, for $i \in I_5$, $\Gamma_i \unlhd G_i$ and the number of $G_i/\Gamma_i$-orbits of vertices of $X_i$ is eight. Hence, $m_{i} \le 8$, for $i \in I_5$.
In $K_i$, the eight vertex orbits under $G_i$ are marked by the symbol \textbf{$\ast$} respectively for $i \in I_5$.
This completes the part {\rm (5)}.

\medskip

Let $I_6 = \{18,52,53,56\}$. For $i \in I_6$, let $X_{i} = K_{i}/\Gamma_{i}$ be a  map on the torus, for some fixed element (vertex, edge or face) free subgroup $\Gamma_{i} \le {\rm Aut}(K_{i})$. $K_i, i \in I_6$, are given in Section \ref{fig:kuniform}.
Let the vertices of $X_{i}$ form $m_{i}$ ${\rm Aut}(X_{i})$-orbits.
In $K_{18}$, we take the middle point of the line segment joining vertices $u_{18}$ and $v_{18}$ as the origin $(0,0)$ and $\alpha_{18} := b_{18} - a_{18}$, $\beta_{18} := c_{18} - a_{18}$ $\in \mathbb{R}^2$.
In $K_{52}$, we take the vertex $a_{52}$ as the origin $(0,0)$ and $\alpha_{52} := b_{52} - a_{52}$, $\beta_{52} := c_{52} - a_{52}$, $\in \mathbb{R}^2$.
In $K_{53}$, we take the middle point of the line segment joining vertices $u_{53}$ and $v_{53}$ as the origin $(0,0)$ and $\alpha_{53} := b_{53} - a_{53}$, $\beta_{53} := c_{53} - a_{53}$ $\in \mathbb{R}^2$.
In $K_{56}$, we take the vertex $a_{56}$ as the origin $(0,0)$ and $\alpha_{56} := b_{56} - a_{56}$, $\beta_{56} := c_{56} - a_{56}$ $\in \mathbb{R}^2$.
For all $i \in I_6$, we define $H_i$ and $G_i$ similarly as done for $K_1$. By the same arguments as above and in Claim 1, for $i \in I_6$, $\Gamma_i \unlhd G_i$ and the number of $G_i/\Gamma_i$-orbits of vertices of $X_i$ is nine. Hence, $m_{i} \le 9$, for $i \in I_6$.
In $K_i$, the nine vertex orbits under $G_i$ are marked by the symbol \textbf{$\ast$} respectively for $i \in I_6$.
This completes the part {\rm (6)}.

\medskip

Let $I_7 = \{13,54\}$. For $i \in I_7$, let $X_{i} = K_{i}/\Gamma_{i}$ be a  map on the torus, for some fixed element (vertex, edge or face) free subgroup $\Gamma_{i} \le {\rm Aut}(K_{i})$. $K_i, i \in I_7$, are given in Section \ref{fig:kuniform}. Let the vertices of $X_{i}$ form $m_{i}$ ${\rm Aut}(X_{i})$-orbits.
In $K_{13}$, we take the vertex $a_{13}$ as the origin $(0,0)$ and $\alpha_{} := b_{13} - a_{13}$, $\beta_{13} := c_{13} - a_{13}$, $\gamma_{13} := d_{13} - a_{13}$ $\in \mathbb{R}^2$.
In $K_{54}$, we take the middle point of the line segment joining vertices $u_{54}$ and $v_{54}$ as the origin $(0,0)$ and $\alpha_{54} := b_{54} - a_{54}$, $\beta_{54} := c_{54} - a_{54}$ $\in \mathbb{R}^2$.
For $K_{13}$, we define $H_{13}$ and $G_{13}$ in a similar way that we did for $K_6$. For $K_{54}$, we define $H_{54}$ and $G_{54}$ in a similar way that we did for $K_1$. By the same arguments as above and in Claim 1, for $i \in I_7$, $\Gamma_i \unlhd G_i$ and the number of $G_i/\Gamma_i$-orbits of vertices of $X_i$ is ten. Hence, $m_{i} \le 10$, for $i \in I_7$. In $K_i$, the ten vertex orbits under $G_i$ are marked by the symbol \textbf{$\ast$} respectively for $i \in I_7$.
This completes the part {\rm (7)}.

\medskip

Let $I_8 = \{37,60\}$. For $i \in I_8$, let $X_{i} = K_{i}/\Gamma_{i}$ be a  map on the torus, for some fixed element (vertex, edge or face) free subgroup $\Gamma_{i} \le {\rm Aut}(K_{i})$. $K_i, i \in I_8$, are given in Section \ref{fig:kuniform}. Let the vertices of $X_{i}$ form $m_{i}$ ${\rm Aut}(X_{i})$-orbits.
In $K_{37}$, we take the vertex $a_{37}$ as the origin $(0,0)$ and $\alpha_{37} := b_{37} - a_{37}$, $\beta_{37} := c_{37} - a_{37}$, $\in \mathbb{R}^2$.
In $K_{60}$, we take the middle point of the line segment joining vertices $u_{60}$ and $v_{60}$ as the origin $(0,0)$ and $\alpha_{60} := b_{60} - a_{60}$, $\beta_{60} := c_{60} - a_{60}$ $\in \mathbb{R}^2$.
For all $i \in I_8$, we define $H_i$ and $G_i$ similarly as done for $K_1$. By the same arguments as above and in Claim 1, for $i \in I_8$, $\Gamma_i \unlhd G_i$ and the number of $G_i/\Gamma_i$-orbits of vertices of $X_i$ is eleven. Hence, $m_{i} \le 11$, for $i \in I_8$.
In $K_i$, the eleven vertex orbits under $G_i$ are marked by the symbol \textbf{$\ast$} respectively for $i \in I_8$.
This completes the part {\rm (8)}.

\medskip

Let $I_9 = \{19,65\}$. For $i \in I_9$, let $X_{i} = K_{i}/\Gamma_{i}$ be a  map on the torus, for some fixed element (vertex, edge or face) free subgroup $\Gamma_{i} \le {\rm Aut}(K_{i})$. $K_i, i \in I_9$, are given in Section \ref{fig:kuniform}. Let the vertices of $X_{i}$ form $m_{i}$ ${\rm Aut}(X_{i})$-orbits.
In $K_{19}$, we take the middle point of the line segment joining vertices $u_{19}$ and $v_{19}$ as the origin $(0,0)$ and $\alpha_{19} := b_{19} - a_{19}$, $\beta_{19} := c_{19} - a_{19}$, $\in \mathbb{R}^2$.
In $K_{65}$, we take the middle point of the line segment joining vertices $u_{65}$ and $v_{65}$ as the origin $(0,0)$ and $\alpha_{65} := b_{65} - a_{65}$, $\beta_{65} := c_{65} - a_{65}$ $\in \mathbb{R}^2$.
For all $i \in I_9$, we define $H_i$ and $G_i$ similarly as done for $K_1$. By the same arguments as above and in Claim 1, for $i \in I_9$, $\Gamma_i \unlhd G_i$ and the number of $G_i/\Gamma_i$-orbits of vertices of $X_i$ is twelve. Hence, $m_{i} \le 12$, for $i \in I_9$.
In $K_i$, the twelve vertex orbits under $G_i$ are marked by the symbol \textbf{$\ast$} respectively for $i \in I_9$.
This completes the part {\rm (9)}.

\medskip

Let $I_{10} = \{15,16,17,48\}$. For $i \in I_{10}$, let $X_{i} = K_{i}/\Gamma_{i}$ be a  map on the torus, for some fixed element (vertex, edge or face) free subgroup $\Gamma_{i} \le {\rm Aut}(K_{i})$. $K_i, i \in I_{10}$, are given in Section \ref{fig:kuniform}. Let the vertices of $X_{i}$ form $m_{i}$ ${\rm Aut}(X_{i})$-orbits.
In $K_{15}$, we take the vertex $a_{15}$ as the origin $(0,0)$ and $\alpha_{15} := b_{15} - a_{15}$, $\beta_{15} := c_{15} - a_{15}$, $\gamma_{15} := d_{15} - a_{15}$ $\in \mathbb{R}^2$.
In $K_{16}$, we take the vertex $a_{16}$ as the origin $(0,0)$ and $\alpha_{16} := b_{16} - a_{16}$, $\beta_{16} := c_{16} - a_{16}$, $\gamma_{16} := d_{16} - a_{16}$ $\in \mathbb{R}^2$.
In $K_{17}$, we take the vertex $a_{17}$ as the origin $(0,0)$ and $\alpha_{17} := b_{17} - a_{17}$, $\beta_{17} := c_{17} - a_{17}$, $\gamma_{17} := d_{17} - a_{17}$ $\in \mathbb{R}^2$.
In $K_{48}$, we take the vertex $a_{48}$ as the origin $(0,0)$ and $\alpha_{48} := b_{48} - a_{48}$, $\beta_{48} := c_{48} - a_{48}$ $\in \mathbb{R}^2$.
For all $i \in I_{10}$, except for $i=48$, we define $H_i$ and $G_i$ similarly as done for $K_6$. For $K_{48}$, we define $H_{48}$ and $G_{48}$ similarly as done for $K_1$. By the same arguments as above and in Claim 1, for $i \in I_{10}$, $\Gamma_i \unlhd G_i$ and the number of $G_i/\Gamma_i$-orbits of vertices of $X_i$ is thirteen. Hence, $m_{i} \le 13$, for $i \in I_{10}$.
In $K_i$, the thirteen vertex orbits under $G_i$ are marked by the symbol \textbf{$\ast$} respectively for $i \in I_{10}$.
This completes the part {\rm (10)}.

\medskip

Let $I_{11} = \{11,12,28\}$. For $i \in I_{11}$, let $X_{i} = K_{i}/\Gamma_{i}$ be a  map on the torus, for some fixed element (vertex, edge or face) free subgroup $\Gamma_{i} \le {\rm Aut}(K_{i})$. $K_i, i \in I_{11}$, are given in Section \ref{fig:kuniform}. Let the vertices of $X_{i}$ form $m_{i}$ ${\rm Aut}(X_{i})$-orbits.
In $K_{11}$, we take the middle point of the line segment joining vertices $u_{11}$ and $v_{11}$ as the origin $(0,0)$ and $\alpha_{11} := b_{11} - a_{11}$, $\beta_{11} := c_{11} - a_{11}$ $\in \mathbb{R}^2$.
In $K_{12}$, we take the middle point of the line segment joining vertices $u_{12}$ and $v_{12}$ as the origin $(0,0)$ and $\alpha_{12} := b_{12} - a_{12}$, $\beta_{12} := c_{12} - a_{12}$, $\gamma_{12} := d_{12} - a_{12}$ $\in \mathbb{R}^2$.
In $K_{28}$, we take the middle point of the line segment joining vertices $u_{28}$ and $v_{28}$ as the origin $(0,0)$ and $\alpha_{28} := b_{28} - a_{28}$, $\beta_{28} := c_{28} - a_{28}$ $\in \mathbb{R}^2$.
For all $i \in I_{11}$, except for $i=12$, we define $H_i$ and $G_i$ similarly as done for $K_1$. For $K_{12}$, we define $H_{12}$ and $G_{12}$ similarly as done for $K_6$. By the same arguments as above and in Claim 1, for $i \in I_{11}$, $\Gamma_i \unlhd G_i$ and the number of $G_i/\Gamma_i$-orbits of vertices of $X_i$ is fifteen. Hence, $m_{i} \le 15$, for $i \in I_{11}$.
In $K_i$, the fifteen vertex orbits under $G_i$ are marked by the symbol \textbf{$\ast$} respectively for $i \in I_{11}$.
This completes the part {\rm (11)}.

\medskip

Let $I_{12} = \{10,38\}$. For $i \in I_{12}$, let $X_{i} = K_{i}/\Gamma_{i}$ be a  map on the torus, for some fixed element (vertex, edge or face) free subgroup $\Gamma_{i} \le {\rm Aut}(K_{i})$. $K_i, i \in I_{12}$, are given in Section \ref{fig:kuniform}. Let the vertices of $X_{i}$ form $m_{i}$ ${\rm Aut}(X_{i})$-orbits.
In $K_{10}$, we take the middle point of the line segment joining vertices $u_{10}$ and $v_{10}$ as the origin $(0,0)$ and $\alpha_{10} := b_{10} - a_{10}$, $\beta_{10} := c_{10} - a_{10}$ $\in \mathbb{R}^2$.
In $K_{38}$, we take the vertex $a_{38}$ as the origin $(0,0)$ and $\alpha_{38} := b_{38} - a_{38}$, $\beta_{38} := c_{38} - a_{38}$ $\in \mathbb{R}^2$.
For all $i \in I_{12}$, we define $H_i$ and $G_i$ similarly as done for $K_1$. By the same arguments as above and in Claim 1, for $i \in I_{12}$, $\Gamma_i \unlhd G_i$ and the number of $G_i/\Gamma_i$-orbits of vertices of $X_i$ is sixteen. Hence, $m_{i} \le 16$, for $i \in I_{12}$.
In $K_i$, the sixteen vertex orbits under $G_i$ are marked by the symbol \textbf{$\ast$} respectively for $i \in I_{12}$.
This completes the part {\rm (12)}.

\medskip

Let $I_{13} = \{27,46\}$. For $i \in I_{13}$, let $X_{i} = K_{i}/\Gamma_{i}$ be a  map on the torus, for some fixed element (vertex, edge or face) free subgroup $\Gamma_{i} \le {\rm Aut}(K_{i})$. $K_i, i \in I_{13}$, are given in Section \ref{fig:kuniform}. Let the vertices of $X_{i}$ form $m_{i}$ ${\rm Aut}(X_{i})$-orbits.
In $K_{27}$, we take the middle point of the line segment joining vertices $u_{27}$ and $v_{27}$ as the origin $(0,0)$ and $\alpha_{27} := b_{27} - a_{27}$, $\beta_{27} := c_{27} - a_{27}$ $\in \mathbb{R}^2$.
In $K_{46}$, we take the vertex $a_{46}$ as the origin $(0,0)$ and $\alpha_{46} := b_{46} - a_{46}$, $\beta_{46} := c_{46} - a_{46}$, $\gamma_{46} := d_{46} - a_{46}$ $\in \mathbb{R}^2$.
For $K_{27}$, we define $H_{27}$ and $G_{27}$ similarly as done for $K_1$. For $K_{46}$, we define $H_{46}$ and $G_{46}$ similarly as done for $K_6$. By the same arguments as above and in Claim 1, for $i \in I_{13}$, $\Gamma_i \unlhd G_i$ and the number of $G_i/\Gamma_i$-orbits of vertices of $X_i$ is eighteen. Hence, $m_{i} \le 18$, for $i \in I_{13}$.
In $K_i$, the eighteen vertex orbits under $G_i$ are marked by the symbol \textbf{$\ast$} respectively for $i \in I_{13}$.
This completes the part {\rm (13)}.

\medskip
 
Let $X_{45} = K_{45}/\Gamma_{45}$ be a  map on the torus, for some fixed element (vertex, edge or face) free subgroup $\Gamma_{45} \le {\rm Aut}(K_{45})$. $K_{45}$ is given in Section \ref{fig:kuniform}. Let the vertices of $X_{45}$ form $m_{45}$ ${\rm Aut}(X_{45})$-orbits.
In $K_{45}$, we take the vertex $a_{45}$ as the origin $(0,0)$ and $\alpha_{45} := b_{45} - a_{45}$, $\beta_{45} := c_{45} - a_{45}$, $\gamma_{45} := d_{45} - a_{45}$ $\in \mathbb{R}^2$.
We define $H_{45}$ and $G_{45}$ similarly as done for $K_6$. By the same arguments as above and in Claim 1, $\Gamma_{45} \unlhd G_{45}$ and the number of $G_{45}/\Gamma_{45}$-orbits of vertices of $X_{45}$ is ninteen. Hence, $m_{45} \le 19$.
In $K_{45}$, the ninteen vertex orbits under $G_{45}$ are marked by the symbol \textbf{$\ast$}.
This completes the part {\rm (14)}.

\medskip

Let $I_{14} = \{23,57\}$. For $i \in I_{14}$, let $X_{i} = K_{i}/\Gamma_{i}$ be a  map on the torus, for some fixed element (vertex, edge or face) free subgroup $\Gamma_{i} \le {\rm Aut}(K_{i})$. $K_i, i \in I_{14}$, are given in Section \ref{fig:kuniform}. Let the vertices of $X_{i}$ form $m_{i}$ ${\rm Aut}(X_{i})$-orbits.
In $K_{23}$, we take the middle point of the line segment joining vertices $u_{23}$ and $v_{23}$ as the origin $(0,0)$ and $\alpha_{23} := b_{23} - a_{23}$, $\beta_{23} := c_{23} - a_{23}$ $\in \mathbb{R}^2$.
In $K_{57}$, we take the middle point of the line segment joining vertices $u_{57}$ and $v_{57}$ as the origin $(0,0)$ and $\alpha_{57} := b_{57} - a_{57}$, $\beta_{57} := c_{57} - a_{57}$, $\in \mathbb{R}^2$.
For all $i \in I_{14}$, we define $H_i$ and $G_i$ similarly as done for $K_1$. By the same arguments as above and in Claim 1, for $i \in I_{14}$, $\Gamma_i \unlhd G_i$ and the number of $G_i/\Gamma_i$-orbits of vertices of $X_i$ is twenty-one. Hence, $m_{i} \le 21$, for $i \in I_{14}$.
In $K_i$, the twenty-one vertex orbits under $G_i$ are marked by the symbol \textbf{$\ast$} respectively for $i \in I_{14}$.
This completes the part {\rm (15)}.

\medskip

Let $I_{15} = \{36,63\}$. For $i \in I_{15}$, let $X_{i} = K_{i}/\Gamma_{i}$ be a  map on the torus, for some fixed element (vertex, edge or face) free subgroup $\Gamma_{i} \le {\rm Aut}(K_{i})$. $K_i, i \in I_{15}$, are given in Section \ref{fig:kuniform}. Let the vertices of $X_{i}$ form $m_{i}$ ${\rm Aut}(X_{i})$-orbits.
In $K_{36}$, we take the middle point of the line segment joining vertices $u_{36}$ and $v_{36}$ as the origin $(0,0)$ and $\alpha_{36} := b_{36} - a_{36}$, $\beta_{36} := c_{36} - a_{36}$ $\in \mathbb{R}^2$.
In $K_{63}$, we take the vertex $a_{63}$ as the origin $(0,0)$ and $\alpha_{63} := b_{63} - a_{63}$, $\beta_{63} := c_{63} - a_{63}$, $\in \mathbb{R}^2$.
For all $i \in I_{15}$, we define $H_i$ and $G_i$ similarly as done for $K_1$. By the same arguments as above and in Claim 1, for $i \in I_{15}$, $\Gamma_i \unlhd G_i$ and the number of $G_i/\Gamma_i$-orbits of vertices of $X_i$ is twenty-two. Hence, $m_{i} \le 22$, for $i \in I_{15}$.
In $K_i$, the twenty-two vertex orbits under $G_i$ are marked by the symbol \textbf{$\ast$} respectively for $i \in I_{15}$.
This completes the part {\rm (16)}.

\medskip
 
Let $X_{49} = K_{49}/\Gamma_{49}$ be a  map on the torus, for some fixed element (vertex, edge or face) free subgroup $\Gamma_{49} \le {\rm Aut}(K_{49})$. $K_{49}$ is given in Section \ref{fig:kuniform}. Let the vertices of $X_{49}$ form $m_{49}$ ${\rm Aut}(X_{49})$-orbits.
In $K_{49}$, we take the middle point of the line segment joining vertices $u_{49}$ and $v_{49}$ as the origin $(0,0)$ and $\alpha_{49} := b_{49} - a_{49}$, $\beta_{49} := c_{49} - a_{49}$ $\in \mathbb{R}^2$.
We define $H_{49}$ and $G_{49}$ similarly as done for $K_1$. By the same arguments as above and in Claim 1, $\Gamma_{49} \unlhd G_{49}$ and the number of $G_{49}/\Gamma_{49}$-orbits of vertices of $X_{49}$ is twenty-four. Hence, $m_{49} \le 24$.
In $K_{49}$, the twenty-four vertex orbits under $G_{49}$ are marked by the symbol \textbf{$\ast$}.
This completes the part {\rm (17)}.

\medskip
 
Let $X_{51} = K_{51}/\Gamma_{51}$ be a  map on the torus, for some fixed element (vertex, edge or face) free subgroup $\Gamma_{51} \le {\rm Aut}(K_{51})$. $K_{51}$ is given in Section \ref{fig:kuniform}. Let the vertices of $X_{51}$ form $m_{51}$ ${\rm Aut}(X_{51})$-orbits.
In $K_{51}$, we take the vertex $a_{51}$ as the origin $(0,0)$ and $\alpha_{51} := b_{51} - a_{51}$, $\beta_{51} := c_{51} - a_{51}$ $\in \mathbb{R}^2$.
We define $H_{51}$ and $G_{51}$ similarly as done for $K_1$. By the same arguments as above and in Claim 1, $\Gamma_{51} \unlhd G_{51}$ and the number of $G_{51}/\Gamma_{51}$-orbits of vertices of $X_{51}$ is twenty-five. Hence, $m_{51} \le 25$.
In $K_{51}$, the twenty-five vertex orbits under $G_{51}$ are marked by the symbol \textbf{$\ast$}.
This completes the part {\rm (18)}.

\medskip
 
Let $X_{58} = K_{58}/\Gamma_{58}$ be a  map on the torus, for some fixed element (vertex, edge or face) free subgroup $\Gamma_{58} \le {\rm Aut}(K_{58})$. $K_{58}$ is given in Section \ref{fig:kuniform}. Let the vertices of $X_{58}$ form $m_{58}$ ${\rm Aut}(X_{58})$-orbits.
In $K_{58}$, we take the middle point of the line segment joining vertices $u_{58}$ and $v_{58}$ as the origin $(0,0)$ and $\alpha_{58} := b_{58} - a_{58}$, $\beta_{58} := c_{58} - a_{58}$, $\gamma_{58} := d_{58} - a_{58}$ $\in \mathbb{R}^2$.
We define $H_{58}$ and $G_{58}$ similarly as done for $K_6$. By the same arguments as above and in Claim 1, $\Gamma_{58} \unlhd G_{58}$ and the number of $G_{58}/\Gamma_{58}$-orbits of vertices of $X_{58}$ is twenty-six. Hence, $m_{58} \le 26$.
In $K_{58}$, the twenty-six vertex orbits under $G_{58}$ are marked by the symbol \textbf{$\ast$}.
This completes the part {\rm (19)}.

\medskip
 
Let $X_{55} = K_{55}/\Gamma_{55}$ be a  map on the torus, for some fixed element (vertex, edge or face) free subgroup $\Gamma_{55} \le {\rm Aut}(K_{55})$. $K_{55}$ is given in Section \ref{fig:kuniform}. Let the vertices of $X_{55}$ form $m_{55}$ ${\rm Aut}(X_{55})$-orbits.
In $K_{55}$, we take the middle point of the line segment joining vertices $u_{55}$ and $v_{55}$ as the origin $(0,0)$ and $\alpha_{55} := b_{55} - a_{55}$, $\beta_{55} := c_{55} - a_{55}$, $\gamma_{55} := d_{55} - a_{55}$ $\in \mathbb{R}^2$.
We define $H_{55}$ and $G_{55}$ similarly as done for $K_6$. By the same arguments as above and in Claim 1, $\Gamma_{55} \unlhd G_{55}$ and the number of $G_{55}/\Gamma_{55}$-orbits of vertices of $X_{55}$ is twenty-seven. Hence, $m_{55} \le 27$.
In $K_{55}$, the twenty-seven vertex orbits under $G_{55}$ are marked by the symbol \textbf{$\ast$}.
This completes the part {\rm (20)}.

\medskip
 
Let $X_{64} = K_{64}/\Gamma_{64}$ be a  map on the torus, for some fixed element (vertex, edge or face) free subgroup $\Gamma_{64} \le {\rm Aut}(K_{64})$. $K_{64}$ is given in Section \ref{fig:kuniform}. Let the vertices of $X_{64}$ form $m_{64}$ ${\rm Aut}(X_{64})$-orbits.
In $K_{64}$, we take the middle point of the line segment joining vertices $u_{64}$ and $v_{64}$ as the origin $(0,0)$ and $\alpha_{64} := b_{64} - a_{64}$, $\beta_{64} := c_{64} - a_{64}$ $\in \mathbb{R}^2$.
We define $H_{64}$ and $G_{64}$ similarly as done for $K_1$. By the same arguments as above and in Claim 1, $\Gamma_{64} \unlhd G_{64}$ and the number of $G_{64}/\Gamma_{64}$-orbits of vertices of $X_{64}$ is twenty-eight. Hence, $m_{64} \le 28$.
In $K_{64}$, the twenty-eight vertex orbits under $G_{64}$ are marked by the symbol \textbf{$\ast$}.
This completes the part {\rm (21)}.

\medskip
 
Let $X_{50} = K_{50}/\Gamma_{50}$ be a  map on the torus, for some fixed element (vertex, edge or face) free subgroup $\Gamma_{50} \le {\rm Aut}(K_{50})$. $K_{50}$ is given in Section \ref{fig:kuniform}. Let the vertices of $X_{50}$ form $m_{50}$ ${\rm Aut}(X_{50})$-orbits.
In $K_{50}$, we take the middle point of the line segment joining vertices $u_{50}$ and $v_{50}$ as the origin $(0,0)$ and $\alpha_{50} := b_{50} - a_{50}$, $\beta_{50} := c_{50} - a_{50}$, $\gamma_{50} := d_{50} - a_{50}$ $\in \mathbb{R}^2$.
We define $H_{50}$ and $G_{50}$ similarly as done for $K_6$. By the same arguments as above and in Claim 1, $\Gamma_{50} \unlhd G_{50}$ and the number of $G_{50}/\Gamma_{50}$-orbits of vertices of $X_{50}$ is thirty. Hence, $m_{50} \le 30$.
In $K_{50}$, the thirty vertex orbits under $G_{50}$ are marked by the symbol \textbf{$\ast$}.
This completes the part {\rm (22)}.

\medskip

Let $I_{16} = \{59,61,62\}$. For $i \in I_{16}$, let $X_{i} = K_{i}/\Gamma_{i}$ be a  map on the torus, for some fixed element (vertex, edge or face) free subgroup $\Gamma_{i} \le {\rm Aut}(K_{i})$. $K_i, i \in I_{16}$, are given in Section \ref{fig:kuniform}. Let the vertices of $X_{i}$ form $m_{i}$ ${\rm Aut}(X_{i})$-orbits.
In $K_{59}$, we take the middle point of the line segment joining vertices $u_{59}$ and $v_{59}$ as the origin $(0,0)$ and $\alpha_{59} := b_{59} - a_{59}$, $\beta_{59} := c_{59} - a_{59}$, $\gamma_{59} := d_{59} - a_{59}$ $\in \mathbb{R}^2$.
In $K_{61}$, we take the middle point of the line segment joining vertices $u_{61}$ and $v_{61}$ as the origin $(0,0)$ and $\alpha_{61} := b_{61} - a_{61}$, $\beta_{61} := c_{61} - a_{61}$ $\in \mathbb{R}^2$.
In $K_{62}$, we take the middle point of the line segment joining vertices $u_{62}$ and $v_{62}$ as the origin $(0,0)$ and $\alpha_{62} := b_{62} - a_{62}$, $\beta_{62} := c_{62} - a_{62}$ $\in \mathbb{R}^2$.
For all $i \in I_{16}$, except for $i=59$, we define $H_i$ and $G_i$ similarly as done for $K_1$. For $K_{59}$, we define $H_{59}$ and $G_{59}$ similarly as done for $K_6$. By the same arguments as above and in Claim 1, for $i \in I_{16}$, $\Gamma_i \unlhd G_i$ and the number of $G_i/\Gamma_i$-orbits of vertices of $X_i$ is thirty-three. Hence, $m_{i} \le 33$, for $i \in I_{16}$.
In $K_i$, the thirty-three vertex orbits under $G_i$ are marked by the symbol \textbf{$\ast$} respectively for $i \in I_{16}$.
This completes the part {\rm (23)}.
\\
Now we show the existence of a toroidal map $K_1/\Gamma_1$ with $m_1=6$. Let $X_1=K_1/\Gamma_1$ for some discrete fixed point free subgroup $\Gamma_1$ of Aut($K_1$). Aut($X_1$)$=$ Nor($\Gamma_1$)/$\Gamma_1$. Now $V(X_1)$ has $6$ $G_1/\Gamma_1$-orbits ($G_1$ defined in part (1)). If we can show that there exists some $\Gamma_1 \le H_1$ such that Nor($\Gamma_1$)$=G_1$ then we are done. Let $\alpha_1, \beta_1$ be the translations defined as in part (1).
Consider $\Gamma_1=\langle A_1^7,B_1^5 \rangle$ where $A_1^7$ and $B_1^5$ are translations by the vectors $7\alpha_1$ and $5\beta_1$ respectively. 
Nor$(\Gamma_1)=\{\gamma \in $ Aut$(K_1) \mid \gamma A_1^7\gamma^{-1},\gamma B_1^5\gamma^{-1}\in \Gamma_1 \} = \{\gamma \in$ Aut$(K_1) \mid \gamma(7\alpha_1), \gamma(5\beta_1) \in \ZZ7\alpha_1+\ZZ5\beta_1 \}$. Clearly $G_1 \le $ Nor$(\Gamma_1)$. Clearly, rotations and reflection about a line does not belong to Nor$(\Gamma_1)$. Hence Nor$(\Gamma_1) = G_1$. With this same method one can see that other bounds in part(1) - (23) are also sharp. This proves part {(24)}.\\
This completes the proof.
\end{proof}

\section{Acknowledgements}

Marbarisha M. Kharkongor is supported by University Grants Commission (UGC) (1182/(CSIR-UGC NET JUNE 2019)) and Dipendu Maity is supported by Science and Engineering Research Board (SERB), DST, India (SRG/2021/000055 dated 03 December, 2021).

\section{Conflicts of Interest Statement} 
On behalf of all authors, the corresponding author states that there is no conflict of interest.


{\small

}

\end{document}